\pdfoutput=1
\documentclass[10pt,a4paper]{article}
\usepackage{amsmath,amssymb,amsthm,amsfonts,graphicx}
\newtheorem{theorem}{Theorem}[section]
\newtheorem{lemma}[theorem]{Lemma}
\newtheorem{corollary}[theorem]{Corollary}

\theoremstyle{definition}
\newtheorem{definition}[theorem]{Definition}
\newtheorem{remark}[theorem]{Remark}
\newtheorem{example}[theorem]{Example} 

\numberwithin{equation}{section}  

\title{\sc On summability, multipliability, product integrability, and parallel translation}
\date{}
\author {Seppo Heikkil\"a$^*$ and Anton\'in Slav\'ik$^{**}$\\
\normalsize $^*$ Department of Mathematical Sciences, University of Oulu\\
\normalsize BOX 3000, FIN-90014, Oulu, Finland\\
\normalsize E-mail: heikki.sep@gmail.com\\
\normalsize $^{**}$  Charles University in Prague, Faculty of Mathematics and Physics\\
\normalsize Sokolovsk\'{a} 83, 186 75 Praha 8, Czech Republic\\ 
\normalsize E-mail: slavik@karlin.mff.cuni.cz\\ \\
\normalsize Dedicated to Professor Heikki Haahti on the occasion of his 85th birthday}

\begin{document}
\maketitle 

\begin{abstract} 
\noindent In this paper we  provide necessary and sufficient conditions for the existence of the  Kurzweil, McShane and Riemann product integrals of step mappings with well-ordered steps, and for right regulated mappings with values in Banach algebras. 
Our basic tools are the concepts of summability and multipliability of families in normed algebras indexed by well-ordered subsets of the real line. These concepts also lead to the generalization of some results from the usual theory of infinite series and products.
Finally, we present two applications of product integrals: First, we describe the relation between Stieltjes-type product integrals, Haahti products, and parallel translation operators. Second, we provide a link between the theory of strong Kurzweil product integrals and strong solutions of linear generalized differential equations.

\smallskip

generalized ordinary differential equation 

\smallskip


\end{abstract}

\baselineskip 13.5pt

\section{Introduction}\label{S1} 
The aim of this paper is to apply the concepts of summability and multipliability in order to generalize some results in the theory of infinite series and products,
and also to derive criteria for product integrability of mappings which take values in Banach algebras. Product integrals and Haahti products are then used to define parallel translation operators and to study their properties. The last part of the paper is devoted to
the relation between product integrals and linear generalized differential equations. 

The paper is organized as follows.

In Section 2, we begin by recalling the definition of summability introduced by S.~Heikkil\"a in \cite{SH13}. We consider sums of the form $\underset{\alpha\in\Lambda}{\Sigma}x_\alpha$, where the index set $\Lambda$ is a well-ordered subset of $\mathbb R\cup\{\infty\}$, and $x_\alpha$ are elements of 
a~normed vector space; sums of this type were used in \cite{SH13} as a tool in the  study of integrability and impulsive differential equations.   
Then we proceed to the related novel concept of multipliability and consider products of the form $\underset{\alpha\in\Lambda}{\Pi}x_\alpha$, where $x_\alpha$ are elements of a normed algebra.
In the case when $\Lambda=\mathbb N$, our definitions and results correspond to the usual  theory of infinite series and products in normed spaces and algebras.  

In Section 3, we recall the general definition of the Kurzweil and McShane product integrals of the form $\prod_a^b V(t,{\rm d}t)$, which were studied in \cite{JK,Sch94,Sch90,S15,ASSS08}, and which include the product integrals $\prod_a^b(I+A(t)\,{\rm d}t)$ and $\prod_a^b(I+{\rm d}A(t))$ considered in the next sections. In infinite-dimensional Banach algebras, the Kurzweil and McShane product integrals lose some of their pleasant properties. To overcome this difficulty, we follow the ideas from A.~Slav\'\i k's paper \cite{S15}, introduce the strong Kurzweil and McShane product integrals $\prod_a^b V(t,{\rm d}t)$, and establish some of their basic properties.

In Sections 4 and 5, we focus on the product integrals $\prod_a^b(I+A(t)\,{\rm d}t)$ in the sense of Kurzweil, McShane and Riemann. We apply the  results from Sections 2, 3 and from the papers \cite{SH13,S15} to derive new sufficient and necessary conditions for 
product integrability of right-continuous step mappings having well-ordered steps, and then for right regulated mappings.  

Section 6 is devoted to the Riemann-Stieltjes and  Kurzweil-Stieltjes product integrals $\prod_a^b(I+{\rm d}A(t))$.
The main result here is concerned with Kurzweil-Stieltjes product integrability of right-continuous step mappings with well-ordered steps. The results from Sections 4, 5, 6 are illustrated on a number of examples.

In Section 7, we present an application of Stieltjes-type product integrals to differential geometry. 
In \cite{HH78}, H.~Haahti and S.~Heikkil\"a studied operators corresponding to parallel translation of vectors along paths on manifolds, and used product and Riemann-Stieltjes product integration techniques to establish the existence of these operators; their results are recalled and generalized in Section 7. 

In Section 8, we provide a link between the theory of Kurzweil product integrals $\prod_a^b V(t,{\rm d}t)$ and  generalized differential equations. We show that under fairly general assumptions, strong Kurzweil product integrability is equivalent to the existence of a strong Kurzweil-Henstock solution of a certain linear generalized differential equation.

\section{Summability, multipliability, and their properties}\label{S2}
\setcounter{equation}{0}

In this section, we generalize some results of the theory of infinite series and products.
A nonempty subset $\Lambda$ of $\mathbb R\cup\{\infty\}$,  ordered by the natural ordering $<$ of $\mathbb R$ 
together with the relation $t<\infty$ for every $t\in\mathbb R$, is well-ordered if    
every nonempty subset of $\Lambda$ has the smallest element. In particular, to every number $\beta$ of 
$\Lambda$, different from its possible maximum, there corresponds the smallest element in $\Lambda$ that is greater than~$\beta$. It is called the successor of $\beta$ and is denoted by $S(\beta)$.  
There are no numbers of $\Lambda$ in the open interval $(\beta,S(\beta))$.  If an element $\gamma$ of $\Lambda$ is not a 
successor or the minimum of $\Lambda$, it is called a limit element. For every $\gamma\in\mathbb R$, we denote 
\begin{eqnarray*}
\Lambda^{<\gamma}&=&\{\alpha\in\Lambda;\,\alpha<\gamma\},\\
\Lambda^{\le\gamma}&=&\{\alpha\in\Lambda;\,\alpha \le \gamma\}.
\end{eqnarray*} 

One of our basic tools in this paper is the following principle of transfinite induction:

\smallskip
{\it If $\Lambda$ is well-ordered and $\mathcal P$ is a property such that  $\mathcal P(\gamma)$ is true whenever $\mathcal P(\beta)$ is true for 
all $\beta\in\Lambda^{<\gamma}$, then $\mathcal P(\gamma)$ is true of all $\gamma\in \Lambda$.}

\smallskip

The following definition of summability is adopted from \cite{SH13}.

\begin{definition}\label{D21} Let $E$ be a normed space, and let  $\Lambda$ be a well-ordered 
subset of $\mathbb R\cup\{\infty\}$. Denote $a=\min \Lambda$, and  $b=\sup\Lambda$. 
The family $(x_\alpha)_{\alpha\in\Lambda}$ with elements $x_\alpha\in E$  is called summable if for every $\gamma\in\Lambda\cup\{b\}$, there is an element 
$\underset{\alpha\in\Lambda^{<\gamma}}{\Sigma}x_\alpha$ of $E$, called the sum of the family $(x_\alpha)_{\alpha\in\Lambda^{<\gamma}}$, satisfying the following conditions: 
\begin{itemize}
\item[(i)] $\underset{\alpha\in\Lambda^{<a}}{\Sigma}x_\alpha=0$, and if $\gamma=S(\beta)$ for some $\beta\in\Lambda$, then 
$\underset{\alpha\in\Lambda^{<\gamma}}{\Sigma}x_\alpha=x_\beta+\underset{\alpha\in\Lambda^{<\beta}}{\Sigma}x_\alpha$.
\item[(ii)] If $\gamma$ is a limit element, then for each $\varepsilon > 0$  there is a  $\beta_\varepsilon\in\Lambda^{<\gamma}$ such that
$$\left\|\underset{\alpha\in\Lambda^{<\beta}}{\Sigma}x_\alpha -\underset{\alpha\in\Lambda^{<\gamma}}{\Sigma}x_\alpha\right\|
<\varepsilon,\quad \beta\in\Lambda\cap[\beta_\varepsilon,\gamma).$$
\end{itemize}
We define the sum $\underset{\alpha\in\Lambda}{\Sigma}x_\alpha$ of a summable family $(x_\alpha)_{\alpha\in\Lambda}$ as $\underset{\alpha\in\Lambda^{<b}}{\Sigma}x_\alpha$ if 
$b\not\in\Lambda$, and $x_b+\underset{\alpha\in\Lambda^{<b}}{\Sigma}x_\alpha$ if $b\in\Lambda$.
A~family  $(x_\alpha)_{\alpha\in\Lambda}$ is called absolutely summable if $(\|x_\alpha\|)_{\alpha\in\Lambda}$ is summable.
\end{definition}

Obviously, for a fixed well-ordered set $\Lambda$, the set of all summable families $(x_\alpha)_{\alpha\in\Lambda}$ forms a linear space.

\smallskip

In a unital normed algebra (i.e., a normed algebra with a unit element whose norm is 1), we introduce the following related concept of multipliability.
We point out that our concept of multipliable families is different from the definition of multipliable sequences given in the appendix of \cite{Bour1}.

\begin{definition}\label{D1} Let $E$ be a unital normed algebra with a unit element $I$, and let $\Lambda$ be a well-ordered 
subset of $\mathbb R\cup\{\infty\}$. Denote $a=\min \Lambda$, and  $b=\sup\Lambda$.
The family $(x_\alpha)_{\alpha\in\Lambda}$ with elements $x_\alpha\in E$  is called multipliable if for every $\gamma\in\Lambda\cup\{b\}$, 
there is an element $\underset{\alpha\in\Lambda^{<\gamma}}{\Pi}x_\alpha$ of $E$, called the product of the family $(x_\alpha)_{\alpha\in\Lambda^{<\gamma}}$, satisfying the following conditions: 
\begin{itemize}
\item[(i)] $\underset{\alpha\in\Lambda^{<a}}{\Pi}x_\alpha=I$, and if $\gamma=S(\beta)$ for some $\beta\in\Lambda$, then 
$\underset{\alpha\in\Lambda^{<\gamma}}{\Pi}x_\alpha=x_\beta\cdot\underset{\alpha\in\Lambda^{<\beta}}{\Pi}x_\alpha$.
\item[(ii)] If $\gamma$ is a limit element, then for each $\varepsilon > 0$  there is a  $\beta_\varepsilon\in\Lambda^{<\gamma}$ such that
$$\left\|\underset{\alpha\in\Lambda^{<\beta}}{\Pi}x_\alpha -\underset{\alpha\in\Lambda^{<\gamma}}{\Pi}x_\alpha\right\|
<\varepsilon,\quad \beta\in\Lambda\cap[\beta_\varepsilon,\gamma).$$
\end{itemize}
We define the product $\underset{\alpha\in\Lambda}{\Pi}x_\alpha$ of a multipliable family $(x_\alpha)_{\alpha\in\Lambda}$ as $\underset{\alpha\in\Lambda^{<b}}{\Pi}x_\alpha$ if 
$b\not\in\Lambda$, and $x_b\cdot\underset{\alpha\in\Lambda^{<b}}{\Pi}x_\alpha$ if $b\in\Lambda$.
\end{definition}

\begin{remark}
In Definition \ref{D21}, condition (ii) can be rephrased by saying that for each limit element $\gamma\in\Lambda$, we have $\lim\limits_{\beta\to\gamma-}\left(\underset{\alpha\in\Lambda^{<\beta}}{\Sigma}x_\alpha\right)=\underset{\alpha\in\Lambda^{<\gamma}}{\Sigma}x_\alpha$.
Similarly, condition (ii) in Definition \ref{D1} says that for each limit element $\gamma\in\Lambda$, we have $\lim\limits_{\beta\to\gamma-}\left(\underset{\alpha\in\Lambda^{<\beta}}{\Pi}x_\alpha\right)=\underset{\alpha\in\Lambda^{<\gamma}}{\Pi}x_\alpha$.
\end{remark}

In the rest of this section,  we assume that  $\Lambda$ is a well-ordered set in $\mathbb R\cup\{\infty\}$ with $a=\min\Lambda$ and $b=\sup \Lambda$.

\smallskip

Observe that if $(x_\alpha)_{\alpha\in\Lambda}$ is a summable family and $c\in(a,b)$, then $(x_\alpha)_{\alpha\in\Lambda\cap[c,b]}$ is summable, too; 
the corresponding partial sums from Definition~\ref{D21} are simply $\underset{\alpha\in(\Lambda\cap[c,b])^{<\gamma}}{\Sigma}x_\alpha
=\underset{\alpha\in\Lambda^{<\gamma}}{\Sigma}x_\alpha-\underset{\alpha\in\Lambda^{<c}}{\Sigma}x_\alpha$. 

On the other hand,
multipliability of $(x_\alpha)_{\alpha\in\Lambda}$ does not necessarily imply the multipliability of $(x_\alpha)_{\alpha\in\Lambda\cap[c,b]}$. However,
the statement becomes true if we assume that the elements of $(x_\alpha)_{\alpha\in\Lambda}$ and its product are invertible. In this case,
the partial products from Definition~\ref{D1} are given by $\underset{\alpha\in(\Lambda\cap[c,b])^{<\gamma}}{\Pi}x_\alpha
=\left(\underset{\alpha\in\Lambda^{<\gamma}}{\Pi}x_\alpha\right)\left(\underset{\alpha\in\Lambda^{<c}}{\Pi}x_\alpha\right)^{\!\!-1}$,
where the invertibility of the last product is guaranteed by the next lemma. 

\begin{lemma}\label{invertible-partial-products}
 Suppose that $(x_\alpha)_{\alpha\in\Lambda^{<b}}$  is a multipliable family in a unital Banach algebra, and that its members and product
are invertible.
Then all products $\underset{\alpha\in\Lambda^{<\gamma}}{\Pi}x_\alpha$, $\gamma\in\Lambda$, are invertible.  
\end{lemma}

\begin{proof}
Assume there is a $\gamma_0\in\Lambda$ such that $\underset{\alpha\in\Lambda^{<\gamma_0}}{\Pi}x_\alpha$ is not invertible. We use transfinite induction to show that for every $\gamma\in\Lambda$ such that $\gamma\ge\gamma_0$, the product $\underset{\alpha\in\Lambda^{<\gamma}}{\Pi}x_\alpha$ is not invertible; this will be in contradiction with the assumption that the product of the whole family is invertible.

We already know that  $\underset{\alpha\in\Lambda^{<\gamma_0}}{\Pi}x_\alpha$ is not invertible.

Given $\gamma>\gamma_0$, assume that $\underset{\alpha\in\Lambda^{<\beta}}{\Pi}x_\alpha$ is not invertible for any $\beta\in\Lambda\cap[\gamma_0,\gamma)$. If $\gamma=S(\beta)$, then
$$\underset{\alpha\in\Lambda^{<\gamma}}{\Pi}x_\alpha=x_\beta\cdot\underset{\alpha\in\Lambda^{<\beta}}{\Pi}x_\alpha$$
and $\underset{\alpha\in\Lambda^{<\gamma}}{\Pi}x_\alpha$ cannot be invertible; otherwise, 
$$\underset{\alpha\in\Lambda^{<\beta}}{\Pi}x_\alpha=x_\beta^{-1}\underset{\alpha\in\Lambda^{<\gamma}}{\Pi}x_\alpha$$
would be invertible, too.

If  $\gamma$ is a limit element, then
$$\underset{\alpha\in\Lambda^{<\gamma}}{\Pi}x_\alpha=\lim_{\beta\to\gamma-}\left(\underset{\alpha\in\Lambda^{<\beta}}{\Pi}x_\alpha\right)$$
cannot be invertible, since the limit of noninvertible elements is always noninvertible (the set of all invertible elements is open).
\end{proof}

The next lemma generalizes two well-known results from the theory of infinite series and products.

\begin{lemma}\label{L22}
 \begin{enumerate}
\item[(a)] If $(x_\alpha)_{\alpha\in\Lambda}$ is a summable family in a normed space, then $\underset{\alpha\to\gamma-}{\lim}x_\alpha =0$  for every limit element $\gamma\in\Lambda$.
\item[(b)] If $(x_\alpha)_{\alpha\in\Lambda}$ is a multipliable family in a unital Banach algebra, and if its elements as well as its product are invertible,
then $\underset{\alpha\to\gamma-}{\lim}x_\alpha =I$ for every limit element $\gamma\in\Lambda$.
\end{enumerate}
\end{lemma}

\begin{proof}
We prove only the second statement; the proof of the first one is similar.
Assume that $(x_\alpha)_{\alpha\in\Lambda}$ is multipliable and its elements as well as its product are invertible.  
Let $\gamma\in\Lambda$ be an arbitrary limit element. 
Given $\varepsilon > 0$, let $\beta_\varepsilon$ be as in Definition \ref{D1} (ii). For every 
$\beta\in \Lambda\cap[\beta_\varepsilon,\gamma)$, we have
$S(\beta)\in \Lambda\cap[\beta_\varepsilon,\gamma)$. Also, applying Definition~\ref{D1}~(ii) and the triangle inequality, we obtain
$$
\left\|\underset{\alpha\in\Lambda^{<S(\beta)}}{\Pi}x_\alpha-\underset{\alpha\in\Lambda^{<\beta}}{\Pi}x_\alpha\right\| < 2\varepsilon,\quad\quad \beta\in \Lambda\cap[\beta_\varepsilon,\gamma).
$$
Consequently,
$$\underset{\beta\to\gamma-}{\lim}\left(\underset{\alpha\in\Lambda^{<S(\beta)}}{\Pi}x_\alpha-\underset{\alpha\in\Lambda^{<\beta}}{\Pi}x_\alpha\right)=0.$$
In view of this result, Lemma \ref{invertible-partial-products} and the continuity of $x\mapsto x^{-1}$, we get  
\begin{equation*}
\lim_{\beta\to\gamma-}\left(x_{\beta}-I\right)=\lim_{\beta\to\gamma-}\left(\left(\underset{\alpha\in\Lambda^{<S(\beta)}}{\Pi}x_\alpha-\underset{\alpha\in\Lambda^{<\beta}}{\Pi}x_\alpha\right)\cdot
\left(\underset{\alpha\in\Lambda^{<\beta}}{\Pi}x_\alpha\right)^{-1}\right)=0\cdot\left(\underset{\alpha\in\Lambda^{<\gamma}}{\Pi}x_\alpha\right)^{-1}=0.\tag*{\qedhere}
\end{equation*}
\end{proof}

The next lemma generalizes the well-known result that in a Banach space, every absolutely convergent series is convergent in the ordinary sense.
The proof is based on the relation between summability and strong Henstock-Kurzweil
integrability of vector-valued step mappings described in \cite{SH13}. 
(For the definition of the strong Henstock-Kurzweil integral, see e.g.~\cite{Sye05,S15}. 
In \cite{SH13}, this integral is referred to as the Henstock-Lebesgue integral.)  

\begin{lemma}\label{L211} Assume that  $(x_\alpha)_{\alpha\in\Lambda^{<b}}$ is an absolutely summable family in a Banach space $E$. Then  $(x_\alpha)_{\alpha\in\Lambda^{<b}}$ is  summable.
\end{lemma}
\begin{proof} 
Without loss of generality, we can suppose that $b=\sup \Lambda<\infty$.
Otherwise, we can replace $\Lambda$ by the well-ordered set $\tilde\Lambda=\{1-\exp(a-\alpha);\,\alpha\in\Lambda\}$ 
with $\min\tilde\Lambda=0$ and $\sup\tilde\Lambda=1$; this transformation preserves (absolute) summability.

For every $\alpha\in\Lambda^{<b}$, let $z_\alpha=\frac{x_\alpha}{S(\alpha)-\alpha}$. Consider the mapping $A:[a,b]\to E$ given by
\begin{equation*}
A(t)=\begin{cases}
z_\alpha,& t\in[\alpha,S(\alpha)),\; \alpha\in\Lambda^{<b},\\
0,& t=b.\end{cases}
\end{equation*}
By \cite[Proposition 3.4]{SH13}, the absolute summability of $(x_\alpha)_{\alpha\in\Lambda^{<b}}=((S(\alpha)-\alpha)z_\alpha)_{\alpha\in\Lambda^{<b}}$ implies
 that $A$ is Bochner integrable. Consequently, $A$ is strongly Henstock-Kurzweil integrable, which means by \cite[Proposition 3.1]{SH13} that 
$((S(\alpha)-\alpha)z_\alpha)_{\alpha\in\Lambda^{<b}}=(x_\alpha)_{\alpha\in\Lambda^{<b}}$ is summable.
\end{proof}

\begin{remark}\label{expAndLogProperties}
In a unital Banach algebra $E$, we may introduce the exponential and logarithm function as follows:
$$\begin{aligned}
\exp x&=\sum_{n=0}^\infty\frac{x^n}{n!}=I+x+\frac{x^2}{2!}+\cdots+\frac{x^n}{n!}+\cdots, \quad x\in E,\\
\log x&=\sum_{n=1}^\infty (-1)^{n-1}\frac{(x-I)^n}{n},\quad \|x-I\|<1.
\end{aligned}$$
These functions have similar properties as in the familiar case when $E=\mathbb R^{n\times n}$, in particular:
\begin{enumerate}
\item The exponential and logarithm are continuous functions.
\item For every $x\in E$, $\exp x$ is an invertible element and its inverse is $\exp(-x)$. 
\item If $x$, $y\in E$ are such that $xy=yx$, then $\exp(x+y)=\exp x\exp y$.
\item $\log (\exp x)=x$ if $\|x\|<\log 2$, and $\exp(\log x)=x$ if $\|x-I\|<1$.
\item We have the estimates 
$$\begin{aligned}
\|\exp x\|&\le\exp\|x\|,\quad\quad x\in E,\\
\|\exp x-I\|&\le\|x\|\exp \|x\|,\quad\quad x\in E,
\end{aligned}$$ 
which follow easily from the definition of the exponential function.
\end{enumerate}
\end{remark}

We now show that the formula $\exp(x)\exp(y)=\exp(x+y)$ can be generalized to families of commutative elements.

\begin{lemma}\label{L25} Let $(x_\alpha)_{\alpha\in\Lambda^{<b}}$ be a summable family in a unital Banach algebra. If
$x_\alpha x_\beta=x_\beta x_\alpha$ whenever $\alpha,\beta\in\Lambda$, then the family $(\exp x_\alpha)_{\alpha\in\Lambda^{<b}}$ is multipliable, and 
\begin{equation}\label{E21}
\underset{\alpha\in\Lambda^{<b}}{\Pi}\exp x_\alpha=\exp\left(\underset{\alpha\in\Lambda^{<b}}{\Sigma}x_\alpha\right).
\end{equation}
\end{lemma}

\begin{proof} 
Using the assumption that $x_\alpha x_\beta=x_\beta x_\alpha$ for all $\alpha,\beta\in\Lambda$,
it follows by transfinite induction with respect to $\beta$ that
\begin{equation}\label{commutativity}
x_\beta\left(\underset{\alpha\in\Lambda^{<\beta}}{\Sigma}x_\alpha\right)=\left(\underset{\alpha\in\Lambda^{<\beta}}{\Sigma}x_\alpha\right)x_\beta,\quad \beta\in\Lambda.
\end{equation}
To prove that the  family $(\exp x_\alpha)_{\alpha\in\Lambda^{<b}}$ is multipliable and \eqref{E21} holds, it is enough to check that
the conditions from Definition \ref{D1} are satisfied with  
\begin{equation}\label{E22}
\underset{\alpha\in\Lambda^{<\gamma}}{\Pi}\exp x_\alpha=\exp\left(\underset{\alpha\in\Lambda^{<\gamma}}{\Sigma}x_\alpha\right), \ \gamma\in\Lambda\cup\{b\}.
\end{equation}
Clearly,
$$\underset{\alpha\in\Lambda^{< a}}{\Pi}\exp x_\alpha=\exp\left(\underset{\alpha\in\Lambda^{<a}}{\Sigma}x_\alpha\right)=\exp 0=I.$$
If $\gamma=S(\beta)$ for some $\beta\in\Lambda$, it follows that 
$$
\underset{\alpha\in\Lambda^{<\gamma}}{\Pi}\exp x_\alpha=
\exp\left(\underset{\alpha\in\Lambda^{<\gamma}}{\Sigma}x_\alpha\right)=\exp\left(x_\beta+\!\!\underset{\alpha\in\Lambda^{<\beta}}{\Sigma}x_\alpha\right)
=\exp x_\beta\cdot\exp\left(\underset{\alpha\in\Lambda^{<\beta}}{\Sigma}x_\alpha\right)
=\exp x_\beta\cdot \underset{\alpha\in\Lambda^{<\beta}}{\Pi}\exp x_\alpha,
$$
where the third equality is a consequence of \eqref{commutativity} and the third property mentioned in Remark \ref{expAndLogProperties}.
Thus condition (i) of Definition \ref{D1} is satisfied.

Assume next that $\gamma$ is a limit element of $\Lambda\cup\{b\}$, and let $\varepsilon > 0$ be given. 
Since the exponential is a~continuous function, it is possible to find $\delta>0$ such that
$$
\left\|\exp x -\exp\left(\underset{\alpha\in\Lambda^{<\gamma}}{\Sigma}x_\alpha\right)\right\|<\varepsilon
$$
for all $x\in E$ satisfying $\left\|x-\underset{\alpha\in\Lambda^{<\gamma}}{\Sigma}x_\alpha\right\|<\delta$.
Because the family $(x_\alpha)_{\alpha\in\Lambda^{<b}}$ is summable, there exists a~$\beta_\delta\in\Lambda^{<\gamma}$ such that 
$$
\left\|\underset{\alpha\in\Lambda^{<\beta}}{\Sigma}x_\alpha -\underset{\alpha\in\Lambda^{<\gamma}}{\Sigma}x_\alpha\right\|
<\delta, \quad \beta\in\Lambda\cap[\beta_\delta,\gamma).
$$
It then follows that 
\begin{equation*}
\left\|\underset{\alpha\in\Lambda^{<\beta}}{\Pi}\exp x_\alpha -\underset{\alpha\in\Lambda^{<\gamma}}{\Pi}\exp x_\alpha\right\|
=\left\|\exp\left(\underset{\alpha\in\Lambda^{<\beta}}{\Sigma}x_\alpha\right)-\exp\left(\underset{\alpha\in\Lambda^{<\gamma}}{\Sigma}x_\alpha\right)\right\|
<\varepsilon, \quad \beta\in\Lambda\cap[\beta_\delta,\gamma).
\end{equation*}
This proves that condition (ii) of Definition \ref{D1} is satisfied. To conclude the proof, we substitute $\gamma=b$ in \eqref{E22} to get \eqref{E21}. 
\end{proof}

\begin{lemma}\label{L25b} Let $(p_\alpha)_{\alpha\in\Lambda^{<b}}$ be a family of real numbers. If
$(\exp p_\alpha)_{\alpha\in\Lambda^{<b}}$ is multipliable and its product is nonzero, then $(p_\alpha)_{\alpha\in\Lambda^{<b}}$ is summable, and 
$$\underset{\alpha\in\Lambda^{<b}}{\Sigma}p_\alpha=\log\left(\underset{\alpha\in\Lambda^{<b}}{\Pi}\exp p_\alpha\right).$$
\end{lemma}

\begin{proof} 
Assume for contradiction that $(p_\alpha)_{\alpha\in\Lambda^{<b}}$ is not summable. Then there is a limit element $\gamma\in \Lambda\cup\{b\}$ such that $(p_\alpha)_{\alpha\in\Lambda^{<\beta}}$ is summable for every $\beta\in\Lambda^{<\gamma}$, but $(p_\alpha)_{\alpha\in\Lambda^{<\gamma}}$ is not summable. Lemma \ref{L25} implies that
$$
\underset{\alpha\in\Lambda^{<\beta}}{\Pi}\exp p_\alpha=\exp\left(\underset{\alpha\in\Lambda^{<\beta}}{\Sigma}p_\alpha\right), \quad \beta\in\Lambda^{<\gamma}.
$$
Since all partial products of $(\exp p_\alpha)_{\alpha\in\Lambda^{<b}}$ are nonzero by Lemma \ref{invertible-partial-products}, we get
$$\underset{\alpha\in\Lambda^{<\beta}}{\Sigma}p_\alpha=\log\left(\underset{\alpha\in\Lambda^{<\beta}}{\Pi}\exp p_\alpha\right), \quad \beta\in\Lambda^{<\gamma}.
$$
Using the continuity of the logarithm function, we obtain
$$\lim\limits_{\beta\to\gamma-}\left(\underset{\alpha\in\Lambda^{<\beta}}{\Sigma}p_\alpha\right)=\log\left(\underset{\alpha\in\Lambda^{<\gamma}}{\Pi}\exp p_\alpha\right), \quad \beta\in\Lambda^{<\gamma},
$$
which contradicts the fact that $(p_\alpha)_{\alpha\in\Lambda^{<\gamma}}$ is not summable.
\end{proof}

The following consequence of Lemmas \ref{L25} and \ref{L25b} shows that absolute summability of $(x_\alpha)_{\alpha\in\Lambda^{<b}}$ and 
multipliability of $(\exp\|x_\alpha\|)_{\alpha\in\Lambda^{<b}}$ are equivalent.

\begin{lemma}\label{L26}  
A family $(x_\alpha)_{\alpha\in\Lambda^{<b}}$ in a normed space is absolutely summable if and only if the family 
$(\exp \|x_\alpha\|)_{\alpha\in\Lambda^{<b}}$ is multipliable. In this case,
\begin{equation}\label{E200}
\underset{\alpha\in\Lambda^{<b}}{\Sigma}\|x_\alpha\|=\log\left(\underset{\alpha\in\Lambda^{<b}}{\Pi}\exp \|x_\alpha\|\right).
\end{equation}
\end{lemma}

\begin{proof} By Lemma \ref{L25}, summability of $(\|x_\alpha\|)_{\alpha\in\Lambda^{<b}}$ implies multipliability of 
$(\exp \|x_\alpha\|)_{\alpha\in\Lambda^{<b}}$ and the relation \eqref{E200}. 
Conversely, if $(\exp \|x_\alpha\|)_{\alpha\in\Lambda^{<b}}$ is multipliable, then it is obvious that $\underset{\alpha\in\Lambda^{<b}}{\Pi}\exp \|x_\alpha\|\ge 1$. It follows from Lemma \ref{L25b} with $p_\alpha=\|x_\alpha\|$, that $(\|x_\alpha\|)_{\alpha\in\Lambda^{<b}}$ is summable.
\end{proof}

The next two results generalize the well-known relations between infinite series and products.

\begin{lemma}\label{L27} A family $(x_\alpha)_{\alpha\in\Lambda^{<b}}$ in a normed space is absolutely summable if and only if
the family of real numbers $(1+\|x_\alpha\|)_{\alpha\in\Lambda^{<b}}$ is multipliable.
\end{lemma}

\begin{proof} Assume first that $(x_\alpha)_{\alpha\in\Lambda^{<b}}$ is absolutely summable.
We use transfinite recursion to define the partial products $\underset{\alpha\in\Lambda^{<\gamma}}{\Pi}(1+\|x_\alpha\|)$, $\gamma\in\Lambda\cup\{b\}$,
so that the conditions of Definition \ref{D1} will be satisfied.
 
First, let $\underset{\alpha\in\Lambda^{<a}}{\Pi}(1+\|x_\alpha\|)=1$. 
Next, assume that
$\underset{\alpha\in\Lambda^{<\beta}}{\Pi}(1+\|x_\alpha\|)$ is defined for each $\beta\in \Lambda^{<\gamma}$, where
$\gamma\in(\Lambda\cup\{b\})\setminus\{a\}$. 
If $\gamma=S(\beta)$, we let 
$$\underset{\alpha\in\Lambda^{<\gamma}}{\Pi}(1+\|x_\beta\|)= 
(1+\|x_\beta\|)\cdot\underset{\alpha\in\Lambda^{<\beta}}{\Pi}(1+\|x_\alpha\|),$$
which ensures  that condition (i) of Definition \ref{D1} is satisfied.

Finally, assume that $\gamma$ is a limit element of $\Lambda\cup\{b\}$.  
By Lemma~\ref{L25}, the family $(\exp \|x_\alpha\|)_{\alpha\in\Lambda^{<b}}$ is multipliable. Moreover, 
it is easy to show by transfinite induction that
$$
 \underset{\alpha\in\Lambda^{<\beta}}{\Pi}(1+\|x_\alpha\|)\le\underset{\alpha\in\Lambda^{<\gamma}}{\Pi}\exp \|x_\alpha\|, \quad\quad \beta\in\Lambda^{<\gamma}.
$$ 
Thus,
$$
s=\sup_{\beta\in \Lambda^{<\gamma}}\left(\underset{\alpha\in\Lambda^{<\beta}}{\Pi}(1+\|x_\alpha\|)\right)
$$
is finite. Given $\varepsilon > 0$, there exists a $\beta_\varepsilon\in\Lambda^{<\gamma}$ such that
$$\left|\underset{\alpha\in\Lambda^{<\beta_\varepsilon}}{\Pi}(1+\|x_\alpha\|)-s\right|<\varepsilon.$$
However, since $\left(\underset{\alpha\in\Lambda^{<\beta}}{\Pi}(1+\|x_\alpha\|)\right)_{\beta\in\Lambda^{<\gamma}}$ is a nondecreasing transfinite sequence, it follows that 
$$
\left|\underset{\alpha\in\Lambda^{<\beta}}{\Pi}(1+\|x_\alpha\|)-s\right|<\varepsilon, \quad\quad \beta\in [\beta_\varepsilon,\gamma)\cap\Lambda.
$$ 
Thus, condition (ii) of Definition~\ref{D1} will be satisfied if we define 
$\underset{\alpha\in\Lambda^{<\gamma}}{\Pi}(1+\|x_\alpha\|)=s$.

The above reasoning implies that $\underset{\alpha\in\Lambda^{<\gamma}}{\Pi}(1+\|x_\alpha\|)$
is defined for each $\gamma\in\Lambda\cup\{b\}$, whence the family  $(1+\|x_\alpha\|)_{\alpha\in\Lambda^{<b}}$ is multipliable.

Assume conversely that  $(1+\|x_\alpha\|)_{\alpha\in\Lambda^{<b}}$ is multipliable. We use transfinite recursion to define the 
partial sums $\underset{\alpha\in\Lambda^{<\gamma}}{\Sigma}\|x_\alpha\|$, $\gamma\in\Lambda\cup\{b\}$,
so that the conditions of Definition \ref{D21} will be satisfied.
At the same time, we are going to prove that
\begin{equation}\label{E203}
\underset{\alpha\in\Lambda^{<\beta}}{\Sigma}\|x_\alpha\|\le \underset{\alpha\in\Lambda^{<\beta}}{\Pi}(1+\|x_\alpha\|)
\end{equation}
for all $\beta\in\Lambda\cup\{b\}$.
First, let  $\underset{\alpha\in\Lambda^{<a}}{\Sigma}\|x_\alpha\|=0$ and note that 
$\underset{\alpha\in\Lambda^{<a}}{\Sigma}\|x_\alpha\|=0<1=\underset{\alpha\in\Lambda^{<a}}{\Pi}(1+\|x_\alpha\|)$. 
Next, assume
that $\underset{\alpha\in\Lambda^{<\beta}}{\Sigma}\|x_\alpha\|$ is defined for each $\beta\in\Lambda^{<\gamma}$, where $\gamma\in(\Lambda\cup\{b\})\setminus\{a\}$, and that \eqref{E203} holds for all $\beta\in\Lambda^{<\gamma}$.
If $\gamma=S(\beta)$, we let 
$$\underset{\alpha\in\Lambda^{<\gamma}}{\Sigma}\|x_\alpha\|=\|x_\beta\|+\underset{\alpha\in\Lambda^{<\beta}}{\Sigma}\|x_\alpha\|,$$
which ensures that condition (i) of Definition \ref{D21} is satisfied. Also, note that
$$
\begin{aligned}
\underset{\alpha\in\Lambda^{<\gamma}}{\Sigma}\|x_\alpha\|&=\|x_\beta\|+\underset{\alpha\in\Lambda^{<\beta}}{\Sigma}\|x_\alpha\|\le \|x_\beta\|+ \underset{\alpha\in\Lambda^{<\beta}}{\Pi}(1+\|x_\alpha\|)\\
&\le \|x_\beta\|\underset{\alpha\in\Lambda^{<\beta}}{\Pi}(1+\|x_\alpha\|)+\underset{\alpha\in\Lambda^{<\beta}}{\Pi}(1+\|x_\alpha\|)\\
&=(1+\|x_\beta\|)\underset{\alpha\in\Lambda^{<\beta}}{\Pi}(1+\|x_\alpha\|)=\underset{\alpha\in\Lambda^{<\gamma}}{\Pi}(1+\|x_\alpha\|),
\end{aligned}
$$
i.e., \eqref{E203} holds when $\beta=\gamma$.

Finally, assume that $\gamma$ is a limit element of $\Lambda\cup\{b\}$.  
We know from \eqref{E203}  that
$$
\underset{\alpha\in\Lambda^{<\beta}}{\Sigma}\|x_\alpha\|\le \underset{\alpha\in\Lambda^{<\gamma}}{\Pi}(1+\|x_\alpha\|), \quad \beta\in\Lambda^{<\gamma}.
$$
Thus,
$$
s'=\sup_{\beta\in \Lambda^{<\gamma}}\left(\underset{\alpha\in\Lambda^{<\beta}}{\Sigma}\|x_\alpha\|\right)
$$
is finite.
Given $\varepsilon > 0$, there exists a $\beta_\varepsilon\in\Lambda^{<\gamma}$ such that
$$\left|\underset{\alpha\in\Lambda^{<\beta_\varepsilon}}{\Sigma}\|x_\alpha\|-s'\right|<\varepsilon.$$
However, since $\left(\underset{\alpha\in\Lambda^{<\beta}}{\Sigma}\|x_\alpha\|\right)_{\beta\in\Lambda^{<\gamma}}$ is a nondecreasing 
transfinite sequence, it follows that 
$$
\left|\underset{\alpha\in\Lambda^{<\beta}}{\Sigma}\|x_\alpha\|-s'\right|<\varepsilon, \quad\quad \beta\in [\beta_\varepsilon,\gamma)\cap\Lambda.
$$ 
Thus, condition (ii) of Definition~\ref{D21} will be satisfied if we define 
$\underset{\alpha\in\Lambda^{<\gamma}}{\Sigma}\|x_\alpha\|= s'$. Also, we have
$$\underset{\alpha\in\Lambda^{<\gamma}}{\Sigma}\|x_\alpha\|=\sup_{\beta\in \Lambda^{<\gamma}}\left(\underset{\alpha\in\Lambda^{<\beta}}{\Sigma}\|x_\alpha\|\right)
\le \sup_{\beta\in \Lambda^{<\gamma}}\left(\underset{\alpha\in\Lambda^{<\beta}}{\Pi}(1+\|x_\alpha\|)\right)=\underset{\alpha\in\Lambda^{<\gamma}}{\Pi}(1+\|x_\alpha\|),$$
i.e., \eqref{E203} holds when $\beta=\gamma$.

The above reasoning implies that $\underset{\alpha\in\Lambda^{<\gamma}}{\Sigma}\|x_\alpha\|$
is defined for each $\gamma\in\Lambda\cup\{b\}$, whence the family  $(\|x_\alpha\|)_{\alpha\in\Lambda^{<b}}$ is summable.
\end{proof}

\begin{lemma}\label{L29} Let $(x_\alpha)_{\alpha\in\Lambda^{<b}}$ be a family in a normed space. 
Assume that $0\le\|x_\alpha\|< 1$ for all $\alpha\in\Lambda^{<b}$. Then  $(x_\alpha)_{\alpha\in\Lambda^{<b}}$ is absolutely summable 
if and only if the product of the family $(1-\|x_\alpha\|)_{\alpha\in\Lambda^{<b}}$ is positive.
\end{lemma}

\begin{proof} Since $0\le\|x_\alpha\|< 1$ for all $\alpha\in\Lambda^{<b}$,
it can be shown by transfinite induction that the family $(1-\|x_\alpha\|)_{\alpha\in\Lambda^{<b}}$ is multipliable, and that the  products
$\underset{\alpha\in\Lambda^{<\gamma}}{\Pi}(1-\|x_\alpha\|)$, $\gamma\in\Lambda\cup\{b\}$, form a decreasing transfinite sequence with values in $[0,1]$.

Suppose that  $(x_\alpha)_{\alpha\in\Lambda^{<b}}$ is absolutely summable. 
Assume for contradiction that $\underset{\alpha\in\Lambda^{<\gamma}}{\Pi}(1-\|x_\alpha\|)=0$ for some $\gamma\in\Lambda\cup\{b\}$. 
Because $\Lambda\cup\{b\}$ is well-ordered, there is the smallest element $\gamma\in\Lambda\cup\{b\}$ with that property. It is a limit element of 
$\Lambda\cup\{b\}$ since $1-\|x_\alpha\|>0$ for each $\alpha\in\Lambda^{<b}$. The assumption that $(x_\alpha)_{\alpha\in\Lambda^{<b}}$ is summable implies by 
Lemma \ref{L27} the existence of a~$\beta\in \Lambda^{<\gamma}$ such that $\|x_\alpha\|\le\frac 12$ when $\alpha\in[\beta,\gamma)\cap\Lambda$. 
Thus $1-\|x_\alpha\|\ge\exp(-2\|x_\alpha\|)$ when $\alpha\in[\beta,\gamma)\cap\Lambda$, so that
$$\underset{\alpha\in[\beta,\gamma)\cap\Lambda}{\Pi}(1-\|x_\alpha\|)\ge\underset{\alpha\in[\beta,\gamma)\cap\Lambda}{\Pi}\exp(-2\|x_\alpha\|)
=\exp\left(-2\underset{\alpha\in[\beta,\gamma)\cap\Lambda}{\Sigma}\|x_\alpha\|\right)>0.$$
Consequently,
$$\underset{\alpha\in\Lambda^{<\gamma}}{\Pi}(1-\|x_\alpha\|)
\ge\exp\left(-2\underset{\alpha\in[\beta,\gamma)\cap\Lambda}{\Sigma}\|x_\alpha\|\right)\underset{\alpha\in\Lambda^{<\beta}}{\Pi}(1-\|x_\alpha\|)>0,$$
a contradiction. Thus $\underset{\alpha\in\Lambda^{<\gamma}}{\Pi}(1-\|x_\alpha\|)>0$ for every $\gamma\in\Lambda\cup\{b\}$, and hence also when $\gamma=b$.

Assume conversely that $\underset{\alpha\in\Lambda^{<b}}{\Pi}(1-\|x_\alpha\|)>0$. Since  $0\le\|x_\alpha\|< 1$, we have 
$1-\|x_\alpha\|\le \exp(-\|x_\alpha\|)\le 1$  for all $\alpha\in\Lambda^{<b}$. Using transfinite induction,
we conclude that the family $(\exp(-\|x_\alpha\|))_{\alpha\in\Lambda^{<b}}$ is multipliable and 
$$\underset{\alpha\in\Lambda^{<b}}{\Pi}\exp(-\|x_\alpha\|)\ge \underset{\alpha\in\Lambda^{<b}}{\Pi}(1-\|x_\alpha\|)>0.$$
By Lemma \ref{L25b} with $p_\alpha=-\|x_\alpha\|$, 
the family  $(-\|x_\alpha\|)_{\alpha\in\Lambda^{<b}}$ is summable. Consequently, the family $(x_\alpha)_{\alpha\in\Lambda^{<b}}$  is absolutely summable.
\end{proof}

If a family of nonnegative real numbers $(p_\alpha)_{\alpha\in\Lambda^{<b}}$ is not summable, we 
write $\underset{\alpha\in\Lambda^{<b}}{\Sigma}p_\alpha=\infty$.
Then we get the following consequence of Lemma \ref{L29}, which generalizes 
\cite[Lemma 8.3.3]{D89}.

\begin{corollary}\label{C210} 
Let $(p_\alpha)_{\alpha\in\Lambda^{<b}}$ be a family of real numbers. Assume that  $0\le p_\alpha < 1$ for all $\alpha\in\Lambda^{<b}$. Then $\underset{\alpha\in\Lambda^{<b}}{\Pi}(1-p_\alpha)=0$ if and only if $\underset{\alpha\in\Lambda^{<b}}{\Sigma}p_\alpha=\infty$. 
\end{corollary}

\begin{example}\label{Ex201} 
The increasing sequence formed by the numbers
\begin{equation}\label{E201}
b-2^{-n}(b-a), \quad n\in\mathbb N_0,
\end{equation}
is a well-ordered subset of the interval $[a,b)\subset\mathbb R$.
The smallest number of this sequence is $a$ and its supremum is $b$. When $a=0$ and $b=1$, the numbers in \eqref{E201} form the increasing sequence
$$
\Lambda_0=\{\alpha(n_0)=1-2^{-n_0};\,n_0\in\mathbb N_0\}.
$$
Clearly, $\Lambda_0$ is a well-ordered subset of $[0,1)$.
The points of $\Lambda_0$ divide the interval $[0,1)$ into disjoint subintervals $[1-2^{-n_0},
1-2^{-n_0-1})$, $n_0\in\mathbb N_0$. Choosing $a=1-2^{-n_0}$, $b=1-2^{-n_0-1}$ in (\ref{E201}) and renaming $n$ to $n_1$, 
we obtain in each of these subintervals increasing sequences, which together form the 
well-ordered set
$$
\Lambda_1=\{\alpha(n_0,n_1)=1-2^{-n_0-1}-2^{-n_0-n_1-1};\,n_0,n_1\in\mathbb N_0\}.
$$
All numbers of $\Lambda_0\setminus\{0\}$ are limit elements of $\Lambda_1$. 

If the above process is repeated, one can obtain additional examples of well-ordered sets
$\Lambda_m$, $m\in\mathbb N$, with a~more complicated structure; see \cite[Example 2.1]{SH13}.

We now construct a family $(x_{\alpha})_{\alpha\in\Lambda_1}$ in the following way:
Choose a vector $z\ne 0$ of $E$, and let
$$
x_{\alpha(n_0,n_1)}=\frac{(-1)^{n_0+n_1}}{(n_0+1)(n_1+1)}z, \quad n_0,n_1\in\mathbb N_0.
$$
The family $(x_{\alpha})_{\alpha\in\Lambda_1}$ is summable, and its sum can be evaluated as the double sum (cf.~\cite[page 4]{SH13}) 
$$
\underset{\alpha\in\Lambda_1}{\Sigma}x_\alpha=\sum_{n_0=0}^\infty\sum_{n_1=0}^\infty\frac{(-1)^{n_0+n_1}}{(n_0+1)(n_1+1)}z
=\sum_{n_0=0}^\infty\left(\frac{(-1)^{n_0}}{n_0+1}\sum_{n_1=0}^\infty\frac{(-1)^{n_1}}{n_1+1}\right)z=(\log 2)^2 z.
$$
Clearly, $x_\alpha x_\beta=x_\beta x_\alpha$ whenever $\alpha,\beta\in\Lambda_1$. 
It then follows from Lemma \ref{L25} that the  family $(\exp x_\alpha)_{\alpha\in\Lambda_1}$ is multipliable,
and its  product is
$$
\underset{\alpha\in\Lambda_1}{\Pi}\exp x_\alpha=\exp\left(\sum_{n_0=0}^\infty\sum_{n_1=0}^\infty\frac{(-1)^{n_0+n_1}}{(n_0+1)(n_1+1)}z\right)= 
\exp\left((\log 2)^2 z\right).
$$
Note that $(x_{\alpha})_{\alpha\in\Lambda_1}$ is not absolutely summable. Thus neither  $(\exp \|x_\alpha\|)_{\alpha\in\Lambda_1}$ nor $(1+\|x_\alpha\|)_{\alpha\in\Lambda_1}$ is multipliable, and the product of $(1-\|x_\alpha\|)_{\alpha\in\Lambda_1}$ is zero.
\end{example}

\section{Product integrals and their properties}\label{S3}

The concept of product integration was originally introduced by V.~Volterra (see e.g.~\cite{AS07,VH}): Given a~continuous matrix-valued function $A:[a,b]\to\mathbb R^{n\times n}$, he considered products of the form
\begin{equation}\label{product}
(I+A(\xi_m)(t_m-t_{m-1}))(I+A(\xi_{m-1})(t_{m-1}-t_{m-2}))\cdots(I+A(\xi_1)(t_1-t_{0})),
\end{equation}
where $a=t_0<t_1<\cdots <t_m=b$ and $\xi_i\in[t_{i-1},t_i]$, $i\in\{1,\ldots, m\}$. The product integral $\prod_a^b(I+A(t)\,{\rm d}t)$ is then defined as the limit of the product \eqref{product} when the lengths of all subintervals  $[t_{i-1},t_i]$ approach zero.
The motivation for introducing this concept stems from the fact that the indefinite product integral
$t\mapsto\prod_a^t(I+A(s)\,{\rm d}s)$, $t\in[a,b]$, corresponds to the fundamental matrix of a system of $n$ homogeneous linear ordinary differential equations $x'(t)=A(t)x(t)$. In \cite{Mas}, P.~R.~Masani generalized this concept to mappings $A:[a,b]\to E$, where $E$ is a unital normed algebra, and $A$ is Riemann integrable. Other authors have considerably extended the class of product integrable mappings by introducing new definitions of product integrals in the spirit of Lebesgue, Bochner, Kurzweil, or McShane; see \cite{DF,JK,Sch94,Sch90,S15,AS07,ASSS08}. 
Product integration of vector-valued functions is applicable in the study of various evolution equations; see e.g.~\cite{DF,LM,SS}.

If the products \eqref{product} are replaced by
$$(I+A(t_m)-A(t_{m-1}))(I+A(t_{m-1})-A(t_{m-2}))\cdots(I+A(t_1)-A(t_{0})),$$
we obtain the Stieltjes-type product integral $\prod_a^b(I+{\rm d}A(t))$. The basic references on this topic are the books \cite{DN, DN2} by R.~M.~Dudley and R.~Norvai\v{s}a, and the paper  \cite{GJ} by R.~D.~Gill and S.~Johansen, who also provide a detailed overview of applications to survival analysis and Markov processes. 
Another motivation for considering Stieltjes-type product integrals comes from the theory of integral equations (also known as generalized linear differential equations; see \cite{Sch,Sch90,MT}) of the form
\begin{equation}\label{GODE}
x(t)=x(a)+\int_a^t{\rm d}[A(s)]x(s),\quad t\in[a,b],
\end{equation}
where $A:[a,b]\to\mathbb R^{n\times n}$, the unknown function $x$ takes values in $\mathbb R^n$, and  the integral on the right-hand side is the Kurzweil-Stieltjes integral. Equations of this form encompass other types of  equations, such as ordinary differential equations with impulses, dynamic equations on time scales, or functional differential equations (see~\cite{MS, MT, Sch, sla}). It turns out that under certain assumptions on $A$, the indefinite Stieltjes product integral $t\mapsto\prod_a^t(I+{\rm d}A(s))$, $t\in[a,b]$, corresponds to the fundamental matrix of Eq.~\eqref{GODE}; see \cite{Sch90}. In Section 7, we show that Stieltjes product integrals are also interesting because of their relation to the differential-geometric concept of parallel translation.

\smallskip

We now summarize some basic facts about product integration that will be needed later, including several new results about strong Kurzweil product integrals. Throughout the rest of the paper, we assume that $E$ is a unital Banach algebra.
 
\smallskip

A tagged partition of an interval $[a,b]$ is a collection of point-interval pairs $D=(\xi_i,[t_{i-1},t_i])_{i=1}^m$, where
$a=t_0<t_1< \cdots<t_m=b$ and $\xi_i\in[t_{i-1},t_i]$ for every $i\in\{1,\ldots,m\}$. 
If we relax the assumption $\xi_i\in[t_{i-1},t_i]$ and replace it by $\xi_i\in[a,b]$, then the collection $D$ is called a free tagged partition. (Note that each tagged partition is also a free tagged partition.)  

Given a function $\delta:[a,b]\to\mathbb R^+$ (called a gauge on $[a,b]$), a free tagged partition is called $\delta$-fine if 
$$[t_{i-1},t_i]\subset(\xi_i-\delta(\xi_i),\xi_i+\delta(\xi_i)),\quad i\in\{1,\ldots,m\}.$$

Let  $\mathcal I$ be the set of all compact subintervals of
$[a,b]$. Assume that a~point-interval function $V:[a,b] \times
\mathcal I \to E$ is given. For an arbitrary free tagged partition $D=(\xi_i,[t_{i-1},t_i])_{i=1}^m$ of the interval $[a,b]$,  
we denote
$$P(V,D)=\prod_{i=m}^1 V(\xi_i,[t_{i-1},t_i])=V(\xi_m,[t_{m-1},t_m])V(\xi_{m-1},[t_{m-2},t_{m-1}])\cdots V(\xi_1,[t_{0},t_1]).$$

\begin{definition}\label{E4def1}
A function $V:[a,b] \times \mathcal I \to E$ is called 
Kurzweil product integrable, if there exists an invertible element $P_V\in E$ 
with the following property: For each $\varepsilon > 0$, there exists a gauge $\delta:[a,b]\to\mathbb R^+$ such that 
\begin{equation}\label{E41}
\left\|P(V,D)-P_V\right\|<\varepsilon
\end{equation}
for all $\delta$-fine tagged partitions of $[a,b]$. In this case, $P_V$ is called the Kurzweil product integral of $V$ and will be denoted 
by ${\prod}_a^b V(t,{\rm d}t)$. 

If \eqref{E41} holds for all $\delta$-fine free tagged partitions of $[a,b]$,
then $V$ is called McShane product integrable over $[a,b]$. The McShane product integral $P_V$ will again be denoted by ${\prod}_a^b V(t,{\rm d}t)$.

The definition of Riemann product integrability  is obtained from the definition 
of Kurzweil product integrability if the gauge $\delta$ is assumed to be constant on $[a,b]$. In this case, the integral 
${\prod}_a^b V(t,{\rm d}t)$  is called the Riemann product integral.
\end{definition}

It follows from the definition that Riemann or McShane product integrability implies Kurzweil product integrability.

In practice, the most common types of product integrals are obtained by taking a function $A:[a,b]\to E$ and defining 
$V:[a,b] \times \mathcal I \to E$  as follows:
\begin{itemize}
\item For $V(t,[x,y])=I+A(t)(y-x)$,  the corresponding product integrals ${\prod}_a^b V(t,{\rm d}t)$ are 
simply referred to as the product integrals of $A$ and  are usually denoted by $\prod_a^b(I+A(t)\,{\rm d}t)$; see Definition \ref{K-prod-int}
\item For $V(t,[x,y])=\exp(A(t)(y-x))$,  the corresponding product integrals ${\prod}_a^b V(t,{\rm d}t)$ are 
called the exponential product integrals of $A$ and  are denoted by $\prod_a^b\exp(A(t)\,{\rm d}t)$; see \cite[Definition 3.7]{S15}.
\item For $V(t,[x,y])=I+A(y)-A(x)$,  the corresponding product integrals ${\prod}_a^b V(t,{\rm d}t)$ are 
called the Stieltjes product integrals of $A$ and  are denoted by $\prod_a^b (I+{\rm d}A(t))$; see Definition \ref{KS-prod-int}.
\end{itemize}

Although the definition of the general product integral $\prod_a^b V(t,{\rm d}t)$ seems unmotivated, it is a convenient concept since it includes the above-mentioned types of product integrals as special cases. By developing the theory of this general product integral, we can avoid the process of repeatedly proving similar theorems for the three particular types of integrals. Moreover, we will demonstrate in Section 8 that the strong version of the product integral $\prod_a^b V(t,{\rm d}t)$ is closely related to J.~Kurzweil's theory of generalized differential equations.

\smallskip
To obtain a reasonable theory, we need to impose certain additional assumptions on the function\break $V:[a,b] \times \mathcal I \to E$.
The following conditions are taken over from \cite{Sch90}, where they are collectively referred to as the condition $\mathcal C$:
\begin{itemize}
\item[(V1)] $V(t,[t,t])=I$ for every $t\in[a,b]$. 
\item[(V2)] For every $t\in[a,b]$ and $\varepsilon > 0$ there is a $\sigma> 0$ such that
$$
\|V(t,[x,y])-V(t,[t,y])V(t,[x,t])\|< \varepsilon$$
for all $x,y\in[a,b]$, $t-\sigma<x\le t\le y< t+\sigma$.
\item[(V3)]  For every $t\in[a,b)$, there exists an invertible element $V_+(t)\in E$ such that
$\underset{y\to t+}{\lim} V(t,[t,y])=V_+(t)$.
\item[(V4)]  For every $t\in(a,b]$, there exists an invertible element $V_-(t)\in E$ such that
$\underset{x\to t-}{\lim} V(t,[x,t])=V_-(t)$.
\end{itemize}

The next statement from \cite[Theorem 1.7]{Sch90} summarizes some basic properties of the Kurzweil/McShane product integrals.
(In \cite{Sch90}, the statement is formulated for $E=\mathbb R^{n\times n}$, but the proof remains valid in every unital Banach algebra; see also \cite[Remark 1.17]{Sch90}.)

\begin{theorem}
Assume that $V:[a,b]\times \mathcal I \to E$ satisfies conditions (V1)--(V4) and  the Kurzweil/McShane product
integral $\prod _a^bV(t,{\rm d}t)$ exists. Then for every $c \in (a,b)$, the Kurzweil/McShane
product integrals
$\prod _a^cV(t,{\rm d}t)$ and $\prod _c^bV(t,{\rm d}t)$ exist and
the equality
\begin{equation}\label{union-equality}
\prod _c^bV(t,{\rm d}t) \prod _a^cV(t,{\rm d}t) = \prod _a^bV(t,{\rm d}t).
\end{equation}
holds. Moreover, the functions $s\mapsto \prod _a^sV(t,{\rm d}t)$ and $s\mapsto \left(\prod_a^sV(t,{\rm d}t)\right)^{-1}$ are bounded on $[a,b]$.
\end{theorem} 

According to the next proposition from \cite[Lemma 1.11]{Sch90}, conditions (V1)--(V4) imply that 
the indefinite Kurzweil product integral is a regulated function.

\begin{theorem}\label{indefinite-regulated}
Assume that $V:[a,b]\times \mathcal I \to E$ satisfies conditions (V1)--(V4) and  the Kurzweil  product
integral $\prod _a^bV(t,{\rm d}t)$ exists. Then 
\begin{eqnarray*}
\underset{\beta\to s-}{\lim}\prod_a^\beta V(t,{\rm d}t)&=&V_{-}(s)^{-1}\cdot\prod _a^s V(t,{\rm d}t),\quad\quad s\in(a,b],\\
\underset{\beta\to s+}{\lim}\prod_a^\beta V(t,{\rm d}t)&=&V_{+}(s)\cdot\prod _a^s V(t,{\rm d}t),\quad\quad\;\;\;\;s\in[a,b).
\end{eqnarray*}
\end{theorem}

We now define the concept of the strong product integral $\prod_a^b V(t,{\rm d}t)$, which generalizes the definitions of 
the strong product integrals 
$\prod_a^b(I+A(t)\,{\rm d}t)$ and $\prod_a^b\exp(A(t)\,{\rm d}t)$ from \cite[Definitions 3.4 and 3.8]{S15}. The motivation for
introducing  strong product integrals is explained in \cite[Section 3]{S15}; the main reason is that in infinite dimension,
ordinary product integrals no longer possess the same pleasant properties as their finite-dimensional counterparts,
while the theory of strong product integrals closely parallels the finite-dimensional case.

\begin{definition}\label{strongDef}
{\em A function $V:[a,b] \times \mathcal I \to E$ is called strongly Kurzweil product integrable if there is a~function
$W:[a,b]\to E$ such that $W(t)^{-1}$ exists for all $t\in[a,b]$, both $W$ and $W^{-1}$ are bounded, and for every $\varepsilon>0$, 
there is a gauge $\delta:[a,b]\to\mathbb R^+$
such that
\begin{equation}\label{strong}
\sum_{i=1}^m\|V(\xi_i,[t_{i-1},t_i])-W(t_i)W(t_{i-1})^{-1}\|<\varepsilon
\end{equation}
for every $\delta$-fine tagged partition of $[a,b]$. In this case, we define the strong Kurzweil product integral as $\prod_a^b V(t,{\rm d}t)=W(b)W(a)^{-1}$.

If \eqref{strong} holds for all $\delta$-fine free tagged partitions of $[a,b]$,
then $A$ is called strongly McShane product integrable over $[a,b]$. The strong McShane product integral is again defined  as $\prod_a^b V(t,{\rm d}t)=W(b)W(a)^{-1}$.}
\end{definition}

The next statement  is a generalization of \cite[Theorem 3.5]{S15}.

\begin{theorem}\label{K-prod-existence}
If $V:[a,b] \times \mathcal I \to E$ is strongly Kurzweil/McShane product integrable, 
then it is also Kurzweil/McShane product integrable and the values of the product integrals coincide.
\end{theorem}

\begin{proof}
Let us prove the statement concerning Kurzweil product integrals; the proof of the McShane counterpart is a straightforward modification.
Consider the function $W$ from Definition \ref{strongDef}. There exists a constant $M>0$ such that 
$\|W(t)\|\le M$ and $\|W(t)^{-1}\|\le M$ for all $t\in[a,b]$. Take an arbitrary $\varepsilon\in\left(0,\frac{1}{M^2}\right)$. There exists a~gauge $\delta:[a,b]\to\mathbb R^+$
such that
\begin{equation*}
\sum_{i=1}^m\|V(\xi_i,[t_{i-1},t_i])-W(t_i)W(t_{i-1})^{-1}\|<\varepsilon
\end{equation*}
for every $\delta$-fine tagged partition of $[a,b]$. Consequently,
$$\sum_{i=1}^m\|W(t_i)^{-1}V(\xi_i,[t_{i-1},t_i])W(t_{i-1})-I\|<M^2\varepsilon<1.$$
We need the following estimate, which follows from \cite[Lemma 2.1]{JK}: 
If $y_1,\ldots,y_m\in E$ are such that $\sum_{i=1}^m\|y_i\|\le 1$, then
$$\left\|(I+y_m)\cdots(I+y_1)-I\right\|\le \sum_{i=1}^m \|y_i\|+\left(\sum_{i=1}^m\|y_i\|\right)^2.$$

By letting $y_i=W(t_i)^{-1}V(\xi_i,[t_{i-1},t_i])W(t_{i-1})-I$, $i\in\{1,\ldots,m\}$, we get
$$\left\|W(t_m)^{-1}\left(\prod_{i=m}^1 V(\xi_i,[t_{i-1},t_i])\right)W(t_{0})-I\right\|
=\left\|\prod_{i=m}^1 W(t_i)^{-1}V(\xi_i,[t_{i-1},t_i])W(t_{i-1})-I\right\|$$
$$=\|(I+y_m)\cdots(I+y_1)-I\|\le  \sum_{i=1}^m\|y_i\|+\left(\sum_{i=1}^m\|y_i\|\right)^2
<M^2\varepsilon+M^4\varepsilon^2.$$
It follows that 
$$\left\|\prod_{i=m}^1 V(\xi_i,[t_{i-1},t_i])-W(b)W(a)^{-1}\right\|= \left\|\prod_{i=m}^1 V(\xi_i,[t_{i-1},t_i])-W(t_m)W(t_0)^{-1}\right\|<M^4\varepsilon+M^6\varepsilon^2$$
for every $\delta$-fine tagged partition of $[a,b]$, which proves that the Kurzweil product integral $\prod_a^b V(t,{\rm d}t)$ exists and equals $W(b)W(a)^{-1}$. 
\end{proof}

It is straightforward to see that strong Kurzweil product integrability on $[a,b]$ implies strong Kurzweil product integrability on every subinterval of $[a,b]$.
In the next theorem, we show that strong Kurzweil product integrability on two adjacent intervals $[a,c]$ and $[c,b]$ implies strong Kurzweil product integrability on $[a,b]$.

\begin{theorem}
Assume that $V:[a,b]\times \mathcal I \to E$ satisfies conditions (V1)--(V4). Moreover, suppose that for a~certain $c \in [a,b]$,
the strong Kurzweil product integrals
$\prod _a^cV(t,{\rm d}t)$ and $\prod _c^bV(t,{\rm d}t)$ exist. Then the strong Kurzweil product integral  $\prod _a^bV(t,{\rm d}t)$ exists as well.
\end{theorem}

\begin{proof}
By the assumption, we have a pair of functions $W_1:[a,c]\to E$, $W_2:[c,b]\to E$
with the properties specified in Definition \ref{strongDef}. Without loss of generality, assume that $W_1(c)=W_2(c)$; otherwise, we can replace $W_2$ by the function $\tilde W_2$ given by $\tilde W_2(t)=W_2(t)W_2(c)^{-1}W_1(c)$. 
Let $W:[a,b]\to E$ be given by $W(t) = W_1(t)$
for $t\in[a, c]$, and $W(t) = W_2(t)$ for $t\in[c, b]$.

There exists an $M>0$ such that $\|W_i(t)\|\le M$ and $\|W_i(t)^{-1}\|\le M$ for all $t$ and $i\in\{1,2\}$. 

For an arbitrary $\varepsilon>0$, we have a pair of gauges $\delta_1:[a,c]\to\mathbb R^+$, $\delta_2:[c,b]\to\mathbb R^+$ having the properties specified in Definition \ref{strongDef}. Also, thanks to the conditions (V2) and (V4), there exists a $\delta_c>0$ such that
\begin{equation}\label{H-prop}
\|V(c,[x,y])-V(c,[c,y])V(c,[x,c])\|<\varepsilon\quad\mbox{ and }\quad \|V(c,[x,c])-V_-(c)\|<\varepsilon
\end{equation}
for all $x,y\in [a,c]$ with $c-\delta_c<x\le c\le y<c+\delta_c$.
Let $\delta:[a,b]\to\mathbb R^+$ be given by
$$\delta(t)=\begin{cases}
\min(\delta_1(t),c-t),& t\in[a,c),\\
\delta_c, & t=c,\\
\min(\delta_2(t),t-c), & t\in(c,b].
\end{cases}$$
Consider an arbitrary $\delta$-fine partition $(\xi_i,[t_{i-1},t_i])_{i=1}^m$ of $[a,b]$. Our choice of $\delta$ implies the existence of a~unique index $j\in\{1,\ldots,m\}$ such that $t_{j-1}\le\xi_j=c\le t_j$. Obviously, we have
\begin{align*}
\sum_{i=1}^{j-1}\|V(\xi_i,[t_{i-1},t_i])-W(t_i)W(t_{i-1})^{-1}\|&=\sum_{i=1}^{j-1}\|V(\xi_i,[t_{i-1},t_i])-W_1(t_i)W_1(t_{i-1})^{-1}\|<\varepsilon,\\
\sum_{i=j+1}^{m}\|V(\xi_i,[t_{i-1},t_i])-W(t_i)W(t_{i-1})^{-1}\|&=\sum_{i=1}^{j-1}\|V(\xi_i,[t_{i-1},t_i])-W_2(t_i)W_2(t_{i-1})^{-1}\|<\varepsilon.
\end{align*}
Moreover, using \eqref{H-prop}, we get
$$\|V(\xi_j,[t_{j-1},t_j])-W(t_j)W(t_{j-1})^{-1}\|=\|V(c,[t_{j-1},t_j])-V(c,[c,t_j])V(c,[t_{j-1},c])\|$$
$$+\|V(c,[c,t_j])V(c,[t_{j-1},c])-W(t_j)W(t_{j-1})^{-1}\|<\varepsilon+\|V(c,[c,t_j])V(c,[t_{j-1},c])-W(t_j)W(c)^{-1}V(c,[t_{j-1},c])\|$$
$$+\|W(t_j)W(c)^{-1}V(c,[t_{j-1},c])-W(t_j)W(c)^{-1}W(c)W(t_{j-1})^{-1}\|$$
$$\leq \varepsilon+\|V(c,[c,t_j])-W(t_j)W(c)^{-1}\|\cdot\|V(c,[t_{j-1},c])\|+\|W(t_j)W(c)^{-1}\|\cdot\|V(c,[t_{j-1},c])-W(c)W(t_{j-1})^{-1}\|$$
$$\leq \varepsilon+\varepsilon\cdot(\varepsilon+\|V_-(c)\|)+M^2\varepsilon=\varepsilon\cdot(1+\varepsilon+\|V_-(c)\|+M^2).$$
Consequently, 
$$\sum_{i=1}^{m}\|V(\xi_i,[t_{i-1},t_i])-W(t_i)W(t_{i-1})^{-1}\|<\varepsilon\cdot(3+\varepsilon+\|V_-(c)\|+M^2),$$
which proves that $V$ is strongly Kurzweil product integrable on $[a,b]$.
\end{proof} 

For strong Kurzweil product integrals, we have the following Hake-type theorem (the corresponding statement for ordinary Kurzweil product integrals can be found in \cite[Theorem 1.13]{Sch90}; the proof still works in unital Banach algebras). 

\begin{theorem}\label{strongHake}
Assume that $V:[a,b]\times \mathcal I \to E$ satisfies conditions (V1)--(V4) and that for every $c\in[a,b)$, the strong Kurzweil product
integral $\prod _a^c V(t,{\rm d}t)$ exists. Suppose also that 
\begin{equation}\label{eqq}
\underset{c\to b-}{\lim}V(b,[c,b])\prod _a^c V(t,{\rm d}t)=L,
\end{equation}
where $L\in E$ is invertible. Then the strong Kurzweil product integral $\prod_a^b V(t,{\rm d}t)$ exists and equals $L$.
\end{theorem}

\begin{proof}
Let $W(t)=\prod_a^t V(s,{\rm d}s)$, $t\in[a,b)$, and $W(b)=L$. 
Eq.~\eqref{eqq} together with condition (V4) imply that $\underset{c\to b-}{\lim}W(c)=V_-(b)^{-1}W(b)$, and therefore
\begin{equation}\label{eqq2}
\underset{c\to b-}{\lim}W(c)^{-1}=W(b)^{-1}V_-(b).
\end{equation}
Let $M>0$ be such that $\|W(t)\|\le M$ for all $t\in[a,b]$. 

Consider an arbitrary $\varepsilon>0$. 
Let $\{b_n\}_{n=1}^\infty$ be an increasing sequence in $(a,b)$ with $\underset{n\to\infty}{\lim}b_n=b$.  For every $n\in\mathbb N$, there exists a~gauge
$\delta_n:[a,b_n]\to\mathbb R^+$ such that the inequality
$$\sum_{i=1}^{m}\|V(\xi_i,[t_{i-1},t_i])-W(t_i)W(t_{i-1})^{-1}\|<\frac{\varepsilon}{2^{n}}$$
holds for each $\delta$-fine partition $(\xi_i,[t_{i-1},t_i])_{i=1}^m$ of $[a,b_n]$.

For an arbitrary $t\in[a,b)$, there is an $n\in\mathbb N$ such that $t\in[a,b_n)$. Let $\delta(t)>0$ be an arbitrary number satisfying
$\delta(t)<\min(\delta_n(t),b_n-t)$. Also, thanks to condition (V4) and Eq.~\eqref{eqq2}, there is a $\delta(b)>0$ such that
$\|V(b,[t,b])-V_-(b)\|<\varepsilon$ and  $\|W(t)^{-1}-W(b)^{-1}V_-(b)\|<\varepsilon$ whenever $t\in(b-\delta(b),b]$. We have now defined a gauge $\delta:[a,b]\to\mathbb R^+$. 
Consider an arbitrary $\delta$-fine partition $(\xi_i,[t_{i-1},t_i])_{i=1}^m$ of $[a,b]$. Our choice of~$\delta$ guarantees that 
$\xi_i<b$ for $i\in\{1,\ldots,m-1\}$ and $\xi_m=b$. Moreover, if $\xi_i\in[a,b_n]$, then $[t_{i-1},t_i]\subset[a,b_n]$. 

Consequently,
$$\sum_{i=1}^{m}\|V(\xi_i,[t_{i-1},t_i])-W(t_i)W(t_{i-1})^{-1}\|=
\sum_{i=1}^{m-1}\|V(\xi_i,[t_{i-1},t_i])-W(t_i)W(t_{i-1})^{-1}\|$$
$$+\|V(b,[t_{m-1},b])-W(b)W(t_{i-1})^{-1}\|\le\sum_{n=1}^\infty\sum_{{i;\, \xi_i\in[a,b_n]}}\|V(\xi_i,[t_{i-1},t_i])-W(t_i)W(t_{i-1})^{-1}\|$$
$$+\|V(b,[t_{m-1},b])-V_-(b)\|+\|W(b)W(b)^{-1}V_-(b)-W(b)W(t_{i-1})^{-1}\|
\le \sum_{n=1}^\infty\frac{\varepsilon}{2^{n}}+\varepsilon+M\varepsilon=\varepsilon(2+M),$$
which proves that $V$ is strongly Kurzweil product integrable on $[a,b]$.
\end{proof} 

We conclude our overview of product integration theory with some information about the product integrals of the form 
$\prod_a^b(I+A(t)\,{\rm d}t)$, which are defined as follows.

\begin{definition}\label{K-prod-int}
A mapping $A:[a,b]\to E$ is called Kurzweil/McShane/Riemann product integrable if the function $V:[a,b] \times \mathcal I \to E$
given by $V(t,[x,y])=I+A(t)(y-x)$ is Kurzweil/McShane/Riemann product integrable in the sense of Definition \ref{E4def1}.
In this case, the product integral of $A$ is defined as $\prod_a^b(I+A(t)\,{\rm d}t)={\prod}_a^b V(t,{\rm d}t)$.

$A$ is called strongly Kurzweil/McShane product integrable if $V$ is strongly Kurzweil/McShane product integrable in the sense of Definition \ref{strongDef}. 
\end{definition}

\begin{remark}\label{prodIntTypes}
For an arbitrary $A:[a,b]\to E$, consider the function $V:[a,b] \times \mathcal I \to E$ given by $V(t,[x,y])=I+A(t)(y-x)$.
Then the condition (V1) is obviously satisfied, and (V3), (V4) hold with $V_+(t)=V_-(t)=I$. Finally, if $t\in[a,b]$ and $\varepsilon>0$,
take an arbitrary $\sigma>0$ such that $\|A(t)\|^2\sigma^2<\varepsilon$.   
Then, if $x,y\in[a,b]$ and 
$t-\sigma<x\le t\le y< t+\sigma$, we have
$$\|V(t,[x,y])-V(t,[t,y])V(t,[x,t])\|=\|A(t)^2\|(y-t)(t-x)<\|A(t)\|^2\sigma^2<\varepsilon,$$
which shows that (V2) is satisfied. 

According to Theorem \ref{indefinite-regulated}, the indefinite Kurzweil product integral $s\mapsto \prod _a^s(I+A(t)\,{\rm d}t)$ is  continuous. 
\end{remark}

The next theorem provides a simple criterion for the existence of the Riemann product integral; the proof can be found in \cite[Section 5]{Mas} or \cite[Section 5.5]{AS07}.

\begin{theorem}
A function $A:[a,b]\to E$ is Riemann product integrable if and only if it is Riemann integrable.
\end{theorem}
 
Next, let us recall the so-called strong Luzin condition. 

\begin{definition}\label{SL}
{\em A mapping $W:[a,b]\to E$ is said to satisfy the strong Luzin condition on $[a,b]$ if for every 
$\varepsilon >0$ and $Z\subset [a,b]$ of measure zero, there
exists a function $\delta:Z\to\mathbb R^+$ such that 
\begin{equation*}
\sum_{j=1}^m \|W(v_j) - W(u_j)\| < \varepsilon
\end{equation*}
for every  collection of point-interval pairs $(\tau_j,[u_j,v_j])_{j=1}^m$ with 
$[u_j,v_j]\subset[a,b]$, $\tau_j \in Z$, and $\tau_j-\delta(\tau_j)<u_j\le\tau_j\le v_j<\tau_j+\delta(\tau_j)$ for all $j\in\{1,2,\dots,m\}$.}
\end{definition}

It is easily verified that every mapping which satisfies the strong Luzin condition is necessarily continuous, and
that a product of two mappings satisfying the strong Luzin condition again satisfies the same condition.

\smallskip

The strong Luzin condition appears in the 
following characterization of strongly Kurzweil product integrable mappings from \cite[Corollary 4.8]{S15}. 

\begin{theorem}\label{T22}
For every mapping $A:[a,b]\to E$, the following statements are equivalent:
\begin{enumerate}
\item $A$ is strongly Kurzweil product integrable.
\item There is a mapping $W:[a,b]\to E$ which satisfies the strong Luzin condition, $W(t)^{-1}$ exists for all $t\in[a,b]$, and $W'(t)=A(t)W(t)$ 
for almost all $t\in[a,b]$.   
\end{enumerate}
\end{theorem}

\begin{remark}\label{T23}
If $A$ is strongly Kurzweil product integrable, then the mapping $W$ from the second statement of Theorem \ref{T22}
can be chosen as the indefinite product integral  $W(t)=\prod_a^t(I+A(s)\,{\rm d}s)$, $t\in[a,b]$; this follows from \cite[Theorems 4.2 and 4.6]{S15}.
Consequently, the indefinite product integral provides a solution of the linear differential equation $W'(t)=A(t)W(t)$ a.e.~in $[a,b]$, $W(a)=I$.
\end{remark}

\smallskip

The next theorem is concerned with the question whether the sum $A_1+A_2$ of two strongly Kurzweil product integrable mappings $A_1, A_2$
is again strongly Kurzweil product integrable. Although we do not know the answer in general,
the next result, which is sufficient for our purposes and will be needed in Section \ref{S6}, provides an affirmative answer in the simpler case when one of the mappings
is Bochner integrable.  Recall that by \cite[Theorem 4.14]{S15}, Bochner integrability is equivalent to strong McShane product integrability,
which  in turn implies strong Kurzweil product integrability.

\begin{lemma}\label{sum}
If $A_1,A_2:[a,b]\to E$ are such that $A_1$ is strongly Kurzweil product integrable and $A_2$ is Bochner integrable, then $A_1+A_2$
is strongly Kurzweil product integrable. 
\end{lemma}

\begin{proof}
According to Remark \ref{T23}, the indefinite product integrals
$$W_i(t)=\prod_a^t (I+A_i(s)\,{\rm d}s),\quad\quad t\in[a,b],\quad i\in\{1,2\},$$
satisfy the strong Luzin condition, $W_i(t)^{-1}$ exists for every $t\in[a,b]$, and 
$$W_i'(t)W_i(t)^{-1}=A_i(t)\quad\quad\mbox{for almost all }t\in[a,b].$$
Next, observe that $W_1^{-1}A_2W_1$ is the product of two continuous mappings and one Bochner integrable mapping,
and is therefore Bochner integrable. Let
\begin{eqnarray*}
V(t)&=&\prod_a^t(I+W_1(s)^{-1}A_2(s)W_1(s)\,{\rm d}s),\quad\quad t\in[a,b],\\
U(t)&=&W_1(t)V(t),\quad\quad t\in[a,b].
\end{eqnarray*}
By Remark \ref{T23},  $V(t)^{-1}$ exists for all $t\in[a,b]$, and we have 
$$V'(t)V(t)^{-1}=W_1(t)^{-1}A_2(t)W_1(t)$$
for almost all $t\in[a,b]$.   Consequently,
$$U'(t)U(t)^{-1}=(W_1'(t)V(t)+W_1(t)V'(t))V(t)^{-1}W_1(t)^{-1}$$
$$=A_1(t)W_1(t)V(t)V(t)^{-1}W_1(t)^{-1}+W_1(t)W_1(t)^{-1}A_2(t)W_1(t)V(t)V(t)^{-1}W_1(t)^{-1}=A_1(t)+A_2(t)$$
for almost all $t\in[a,b]$.   
Since $W_1$ and $V$ satisfy the strong Luzin condition,
it follows that $U$ satisfies the same condition.
By Theorem \ref{T22}, the existence of a mapping $U$ with the properties described above implies that $A_1+A_2$ is strongly Kurzweil product integrable.   
\end{proof} 

\section{Product integrability of step mappings}\label{S4}

In the present section, we focus on the existence of the product integral 
$\prod_a^b(I+A(t)\,{\rm d}t)$ corresponding to a~mapping $A:[a,b]\to E$, where 
$E$ is a unital Banach algebra (see Definition \ref{K-prod-int}).

For step mappings with finitely many steps, the Riemann, strong Kurzweil and strong McShane product integrals always exist 
and are easy to calculate: if there is a partition $a=t_0<t_1<\cdots<t_m=b$ and $A(t)=A_i\in E$ for all $t\in(t_{i-1},t_i)$, then
it was shown in \cite[Example 4.15]{S15} that
$$\prod_a^b(I+A(t)\,{\rm d}t)=\prod_{i=m}^1 \prod_{t_{i-1}}^{t_i}(I+A(t)\,{\rm d}t)=\prod_{i=m}^1\exp(A_i(t_i-t_{i-1})).$$

In this section, we study the existence of the product integral $\prod_a^b(I+A(t)\,{\rm d}t)$ in
the case when $A$~is a step mapping with well-ordered steps. More precisely, we assume the existence 
of a well-ordered subset $\Lambda$ of $[a,b]$ such that
$\min\Lambda=a$ and $\max\Lambda = b$, and a family $(z_\alpha)_{\alpha\in\Lambda}$ of $E$ such that
\begin{equation}\label{E30}
A(t)=\begin{cases}
z_\alpha,& t\in[\alpha,S(\alpha)),\; \alpha\in\Lambda^{<b},\\
z_b,& t=b.\end{cases}
\end{equation}

Because $[a,b)$ is a countable union of the disjoint intervals $[\alpha,S(\alpha))$, $\alpha\in\Lambda$, $A$ is well-defined on $[a,b]$ by \eqref{E30}. 
Each mapping of this form has at most countably many discontinuities. Hence, the following result from \cite[Theorem 5.3]{S15} is applicable
in our situation. 

\begin{theorem}\label{K-right-regulated}
If $A:[a,b]\to E$ has countably many discontinuities,
then the following conditions are equivalent: 
\begin{enumerate}
\item $A$ is strongly Kurzweil product integrable.
\item $A$ is Kurzweil product integrable.
\item There is a continuous mapping $W:[a,b]\to E$ such that $W(t)^{-1}$ exists for all $t\in[a,b]$ and $W'(t)=A(t)W(t)$ for all $t\in[a,b]\backslash Z$,
where $Z$ is countable.
\end{enumerate}
\end{theorem}

We now show that Kurzweil product integrability of step mappings is closely related to the concept of multipliability introduced in Section \ref{S2}.
The proof is inspired by the proof of \cite[Proposition 3.1]{SH13}.

\begin{theorem}\label{Kurzweil-step} Let $A:[a,b]\to E$ be a step mapping with representation (\ref{E30}).
Then the following conditions are equivalent: 
\begin{enumerate}
\item $A$ is strongly Kurzweil product integrable.
\item The family $(\exp({(S(\alpha)-\alpha)z_\alpha}))_{\alpha\in\Lambda^{<b}}$ is multipliable and its
product is invertible.
\end{enumerate}
If any of these conditions is satisfied, we have 
$$
\prod_a^b(I+A(t)\,{\rm d}t)=\underset{\alpha\in\Lambda^{<b}}{\Pi}\exp((S(\alpha)-\alpha)z_\alpha).
$$
In particular, the product on the right-hand side is an invertible element of $E$.
\end{theorem}

\begin{proof} 
We begin by proving the implication $1\Rightarrow 2$. Denote $x_\alpha=\exp((S(\alpha)-\alpha)z_\alpha)$, $\alpha\in\Lambda^{<b}$. To prove that the family $(x_\alpha)_{\alpha\in\Lambda^{<b}}$ is multipliable,
it suffices to show that conditions (i) and (ii) of Definition~\ref{D1} are satisfied with 
\begin{equation*}
\underset{\alpha\in\Lambda^{<\beta}}{\Pi}x_\alpha= \prod_a^\beta(I+A(t)\,{\rm d}t), \quad \beta\in\Lambda.
\end{equation*}
Clearly, $\underset{\alpha\in\Lambda^{<a}}{\Pi}x_\alpha=\prod_a^a (I+A(t)\,{\rm d}t)=I$.
Assume next that  $\gamma\in\Lambda$ is a successor, i.e., $\gamma=S(\beta)$ for some $\beta\in\Lambda$.
Then \cite[Example 4.11]{S15} implies that $\prod_\beta^\gamma (I+A(t)\,{\rm d}t)=\exp((S(\beta)-\beta)z_\beta)=x_\beta$.
Thus
$$
x_\beta\cdot\underset{\alpha\in\Lambda^{<\beta}}{\Pi}x_\alpha=\prod_\beta^{\gamma}(I+A(t)\,{\rm d}t)\cdot\prod_a^\beta (I+A(t)\,{\rm d}t)= \prod_a^{\gamma} (I+A(t)\,{\rm d}t)=\underset{\alpha\in\Lambda^{<\gamma}}{\Pi}x_\alpha
$$
and condition (i) of Definition \ref{D1} is satisfied.

Assume finally that $\gamma$ is a limit element, and let $\varepsilon > 0$ be given.  
Since $t\mapsto \overset{t}{\underset{a}{\prod}}(I+A(s)\,{\rm d}s)$ is continuous at $t=\gamma$
and $\gamma$ is a limit element, there exists a $\beta_\varepsilon \in \Lambda^{<\gamma}$ such that
$$
\left\|\prod_a^\beta (I+A(t)\,{\rm d}t)-\prod_a^\gamma (I+A(t)\,{\rm d}t)\right\|< \varepsilon,\quad \beta\in\Lambda\cap[\beta_\varepsilon,\gamma).
$$
Consequently,
$$\left\|\underset{\alpha\in\Lambda^{<\beta}}{\Pi}x_\alpha -\underset{\alpha\in\Lambda^{<\gamma}}{\Pi}x_\alpha\right\|
<\varepsilon,\quad \beta\in\Lambda\cap[\beta_\varepsilon,\gamma)$$
and condition (ii) of Definition \ref{D1} is also satisfied. 

It remains to prove the implication $2\Rightarrow 1$. 
Assume that the family $(\exp((S(\alpha)-\alpha)z_\alpha))_{\alpha\in\Lambda^{<b}}$ is multipliable 
and its product is invertible. Consider the mapping $W:[a,b]\to E$ given by
$$\begin{aligned}
W(t)&=\exp((t-\gamma)z_\gamma)\left(\underset{\alpha\in\Lambda^{<\gamma}}{\Pi}\exp((S(\alpha)-\alpha)z_\alpha)\right),\quad\quad t\in[\gamma,S(\gamma)),\quad\gamma\in\Lambda^{<b},\\
W(b)&=\underset{\alpha\in\Lambda^{<b}}{\Pi}\exp((S(\alpha)-\alpha)z_\alpha).
\end{aligned}$$
To finish the proof, it is enough to verify that $W$ satisfies condition 3 of Theorem \ref{K-right-regulated}.
Note that $W(t)^{-1}$ exists for every $t\in[a,b]$, and  
$$W'(t)=z_\gamma\exp((t-\gamma)z_\gamma)\left(\underset{\alpha\in\Lambda^{<\gamma}}{\Pi}\exp((S(\alpha)-\alpha)z_\alpha)\right)=A(t)W(t),\quad\quad t\in(\gamma,S(\gamma)),\quad\gamma\in\Lambda^{<b},$$
i.e., $W'(t)=A(t)W(t)$ for every $t\in(a,b)\backslash\Lambda^{<b}$.
In particular, $W$ is continuous at every point $t\in(a,b)\backslash\Lambda^{<b}$.
Let us show that $W$ is in fact continuous on the whole interval $[a,b]$. By definition, $W$ is right-continuous at every point $t\in[a,b)$.
We need to show that $W$ is left-continuous at every point $\gamma\in\Lambda$.

If $\gamma=S(\beta)$ for some $\beta\in\Lambda$, then
$$\underset{t\to\gamma-}{\lim}W(t)=\exp((\gamma-\beta)z_\beta)\left(\underset{\alpha\in\Lambda^{<\beta}}{\Pi}\exp((S(\alpha)-\alpha)z_\alpha)\right)
=\left(\underset{\alpha\in\Lambda^{<\gamma}}{\Pi}\exp((S(\alpha)-\alpha)z_\alpha)\right)=W(\gamma).$$

If $\gamma$ is a limit element, we know that  
\begin{equation}\label{finiteLimit}
\underset{\beta\to\gamma-}{\lim}\underset{\alpha\in\Lambda^{<\beta}}{\Pi}\exp((S(\alpha)-\alpha)z_\alpha)=\underset{\alpha\in\Lambda^{<\gamma}}{\Pi}\exp((S(\alpha)-\alpha)z_\alpha)
\end{equation}
Also, the second part of Lemma \ref{L22} implies ${\lim}_{\beta\to\gamma-}\exp((S(\beta)-\beta)z_\beta) =I$;
using the continuity of the logarithm function, we get 
\begin{equation}\label{zeroLimit}
\underset{\beta\to\gamma-}{\lim}(S(\beta)-\beta)z_\beta =
\underset{\beta\to\gamma-}{\lim}\log\left(\exp((S(\beta)-\beta)z_\beta)\right)=\log I=0.
\end{equation}
Now, for an arbitrary $t\in[a,\gamma)$, there exists a $\beta\in\Lambda\cap[a,\gamma)$ such that $t\in[\beta,S(\beta))$.
Note that
$$\|W(t)-W(\beta)\|= 
\left\|\left(\exp((t-\beta)z_\beta)-I\right)\underset{\alpha\in\Lambda^{<\beta}}{\Pi}\exp((S(\alpha)-\alpha)z_\alpha)\right\|$$
$$\le\|(t-\beta)z_\beta\|\exp(\|(t-\beta)z_\beta\|)\left\|\underset{\alpha\in\Lambda^{<\beta}}{\Pi}\exp((S(\alpha)-\alpha)z_\alpha)\right\|$$
\begin{equation}\label{rhs}
\le\|(S(\beta)-\beta)z_\beta\|\exp(\|(S(\beta)-\beta)z_\beta\|)\left\|\underset{\alpha\in\Lambda^{<\beta}}{\Pi}\exp((S(\alpha)-\alpha)z_\alpha)\right\|.
\end{equation}
For $t\to\gamma -$, we have $\beta\to\gamma -$
and the expression in \eqref{rhs} tends to 0 because of \eqref{finiteLimit} and \eqref{zeroLimit}. 
Hence,
$$\underset{t\to\gamma-}{\lim}W(t)=\underset{t\to\gamma-}{\lim}W(\beta)+\underset{t\to\gamma-}{\lim}(W(t)-W(\beta))=\lim_{\substack{\beta\to\gamma-,\\\beta\in\Lambda}}W(\beta)=W(\gamma),$$ 
where the last equality follows from \eqref{finiteLimit}. This proves that $W$ is left-continuous at every point $\gamma\in\Lambda$.
\end{proof}

In the commutative case, we obtain the following criterion. 

\begin{theorem}\label{T61}
Let $A:[a,b]\to E$ be a step mapping with representation (\ref{E30}).
Assume that 
$z_\alpha z_\beta=z_\beta z_\alpha$ whenever $\alpha,\beta\in\Lambda^{<b}$.
Then the following conditions are equivalent:
\begin{enumerate}
\item $A$ is strongly Kurzweil product integrable.
\item $A$ is strongly Henstock-Kurzweil integrable.
\item The family $((S(\alpha)-\alpha)z_\alpha)_{\alpha\in\Lambda^{<b}}$ is summable. 
\item The family  $(\exp({(S(\alpha)-\alpha)z_\alpha}))_{\alpha\in\Lambda^{<b}}$ is multipliable and its
product is invertible.
\end{enumerate}
If any of these conditions is satisfied, we have
\begin{equation}\label{E61}
\prod_a^b(I+A(t)\,{\rm d}t)=\exp\left(\int_a^bA(t)\,{\rm d}t\right)=\underset{\alpha\in\Lambda^{<b}}{\Pi}\exp((S(\alpha)-\alpha)z_\alpha)=\exp\left(\underset{\alpha\in\Lambda^{<b}}{\Sigma}(S(\alpha)-\alpha)z_\alpha\right).
\end{equation}
\end{theorem}

\begin{proof}
Conditions 1 and 4 are equivalent by  Theorem~\ref{Kurzweil-step}, 
and conditions 2 and 3 are equivalent by \cite[Proposition~3.1]{SH13}. 
The commutativity assumption and \cite[Theorems 3.12, 3.13]{S15}  imply that conditions 1 and~2 are equivalent and the first equality in~\eqref{E61} holds. 
The second equality follows from Theorem~\ref{Kurzweil-step}, and the third one from Lemma~\ref{L25}.
\end{proof}

For the Riemann product integral, we have an even simpler condition.

\begin{theorem}\label{RiemannStep}  Let $A:[a,b]\to E$ be a step mapping with representation (\ref{E30}).
Then the following conditions are equivalent: 
\begin{enumerate}
\item $A$ is Riemann product integrable.
\item The family $(z_\alpha)_{\alpha\in\Lambda}$ is bounded. 
\end{enumerate}
\end{theorem}

\begin{proof} 
Recall that $A$ is Riemann product integrable if and only if it is Riemann integrable.
By \cite[Proposition 3.5]{SH13}, this happens if and only if the family $(z_\alpha)_{\alpha\in\Lambda}$ is bounded. 
\end{proof} 

Finally, we have the following characterization of strong McShane/Bochner product integrability.

\begin{theorem}\label{BochnerStep} 
Let $A:[a,b]\to E$ be a step mapping with representation (\ref{E30}).
Then the following conditions are equivalent: 
\begin{enumerate}
\item $A$ is strongly McShane product integrable.
\item $A$ is Bochner integrable.
\item The family $((S(\alpha)-\alpha)z_\alpha)_{\alpha\in\Lambda^{<b}}$ is absolutely summable.
\item The family $(\exp({(S(\alpha)-\alpha)\|z_\alpha\|}))_{\alpha\in\Lambda^{<b}}$ is multipliable.
\item The family $(1+(S(\alpha)-\alpha)\|z_\alpha\|)_{\alpha\in\Lambda^{<b}}$ is multipliable.
\end{enumerate}
\end{theorem} 

\begin{proof}
Conditions 1 and 2 are equivalent by \cite[Theorem 4.14]{S15}, conditions 2 and 3 by \cite[Proposition 3.4]{SH13},
 conditions 3 and 4 by  Lemma \ref{L26}, and conditions 3 and 5 by Lemma \ref{L27}. 
\end{proof}

In the previous theorems, the concept of multipliability was used to obtain new criteria of product integrability.
Conversely, we can apply existing results about product integrals to obtain new results about multipliability. 
As an illustration, we prove the following converse to Lemma \ref{L25}.

\begin{lemma}\label{converse-L25}
Let $\Lambda$ be a well-ordered set in $\mathbb R\cup\{\infty\}$ with $a=\min\Lambda$ and $b=\sup \Lambda$. 
Assume that  $(x_\alpha)_{\alpha\in\Lambda^{<b}}$ is a family in a unital Banach algebra $E$ such that
$x_\alpha x_\beta=x_\beta x_\alpha$ whenever $\alpha,\beta\in\Lambda$.
If the family $(\exp x_\alpha)_{\alpha\in\Lambda^{<b}}$ is multipliable and its product is invertible, then $(x_\alpha)_{\alpha\in\Lambda^{<b}}$ is summable.
\end{lemma}

\begin{proof}
As in the proof of Lemma \ref{L211}, we can suppose that $b<\infty$. 
For every $\alpha\in\Lambda^{<b}$, let $z_\alpha=\frac{x_\alpha}{S(\alpha)-\alpha}$ and consider the mapping $A:[a,b]\to E$ given by
\begin{equation*}
A(t)=\begin{cases}
z_\alpha,& t\in[\alpha,S(\alpha)),\; \alpha\in\Lambda^{<b},\\
0,& t=b.\end{cases}
\end{equation*} 
We know that $(\exp x_\alpha)_{\alpha\in\Lambda^{<b}}=(\exp((S(\alpha)-\alpha)z_\alpha))_{\alpha\in\Lambda^{<b}}$ is multipliable and its product is invertible. By Theorem \ref{T61}, the family $((S(\alpha)-\alpha)z_\alpha)_{\alpha\in\Lambda^{<b}}=(x_\alpha)_{\alpha\in\Lambda^{<b}}$ is summable.
\end{proof}

Using a similar approach, we get the next statement, which generalizes one part of Lemma \ref{L26}.

\begin{lemma}
Let $\Lambda$ be a well-ordered set in $\mathbb R\cup\{\infty\}$ with $a=\min\Lambda$ and $b=\sup \Lambda$. 
Assume that  $(x_\alpha)_{\alpha\in\Lambda^{<b}}$ is an absolutely summable family in a unital Banach algebra $E$.
Then  $(\exp x_\alpha)_{\alpha\in\Lambda^{<b}}$ is multipliable and has an invertible product.
\end{lemma}

\begin{proof}
As in the proof of Lemma \ref{L211}, we can suppose that $b<\infty$. 
Let $(z_\alpha)_{\alpha\in\Lambda^{<b}}$ and $A:[a,b]\to E$ have the same meaning as in the proof of Lemma \ref{converse-L25}.
By Theorem \ref{BochnerStep}, the absolute summability of $(x_\alpha)_{\alpha\in\Lambda^{<b}}=((S(\alpha)-\alpha)z_\alpha)_{\alpha\in\Lambda^{<b}}$ 
implies that $A$ is strongly McShane product integrable, and therefore also strongly Kurzweil product integrable. By Theorem~\ref{T61},  
the family  $(\exp x_\alpha)_{\alpha\in\Lambda^{<b}}=(\exp ((S(\alpha)-\alpha)z_\alpha))_{\alpha\in\Lambda^{<b}}$ is multipliable and its product is invertible.
\end{proof}

\begin{example}\label{Ex301} As noticed in Example \ref{Ex201}, the set
$$
\Lambda_1=\{\alpha(n_0,n_1)=1-2^{-n_0-1}-2^{-n_0-n_1-1};\,n_0,n_1\in\mathbb N_0\}
$$
is a well-ordered subset of $[0,1)$. Routine calculations show that for every $\alpha=\alpha(n_0,n_1)\in\Lambda_1$, we have 
$$S(\alpha)-\alpha=\alpha(n_0,n_1+1)-\alpha(n_0,n_1)=2^{-n_0-n_1-2}.$$
Choose a vector $z\ne 0$ of $E$, and let $A:[0,1]\to E$ have the representation \eqref{E30}, where $\Lambda=\Lambda_1\cup\{1\}$ and
$$
z_\alpha=z_{\alpha(n_0,n_1)}=\frac{(-2)^{n_0+n_1+2}}{(n_0+1)(n_1+1)}z, \quad \alpha=\alpha(n_0,n_1)\in\Lambda_1, \quad z_1=0.
$$
Note that $$(S(\alpha)-\alpha)z_\alpha=\frac{(-1)^{n_0+n_1}}{(n_0+1)(n_1+1)}z.$$
Hence, the family 
$((S(\alpha)-\alpha)z_\alpha)_{\alpha\in\Lambda^{<1}}$ is equal to the family $(x_\alpha)_{\alpha\in\Lambda{_1}}$ considered in Example \ref{Ex201}, and
$$
\underset{\alpha\in\Lambda^{<1}}{\Sigma}(S(\alpha)-\alpha)z_\alpha=\underset{\alpha\in\Lambda_1}{\Sigma}x_\alpha=\sum_{n_0=0}^\infty\sum_{n_1=0}^\infty\frac{(-1)^{n_0+n_1}}{(n_0+1)(n_1+1)}z=(\log 2)^2 z.
$$
Since $z_\alpha z_\beta=z_\beta z_\alpha$ whenever $\alpha,\beta\in\Lambda$, 
it follows from Theorem \ref{T61} that $A$ is strongly Kurzweil product integrable and 
$$
\prod_0^1(I+A(t)\,{\rm d}t)=\exp\left(\int_0^1 A(t)\,{\rm d}t\right)=\exp\left(\underset{\alpha\in\Lambda^{<1}}{\Sigma}(S(\alpha)-\alpha)z_\alpha\right)=\exp\left((\log 2)^2 z\right).
$$
On the other hand, since $(x_\alpha)_{\alpha\in\Lambda{_1}}$ is neither bounded nor absolutely summable, 
Theorems \ref{RiemannStep} and \ref{BochnerStep} imply that $A$ is neither Riemann product integrable nor strongly McShane product integrable.
\end{example}

\begin{example}\label{Ex302} Let $\Lambda_1$ be as in Example \ref{Ex301}.
Choose a vector $z\ne 0$ of $E$, and let $A:[0,1]\to E$ have the representation \eqref{E30}, where $\Lambda=\Lambda_1\cup\{1\}$ and
$$
z_\alpha=z_{\alpha(n_0,n_1)}=\frac{2^{n_0+n_1+2}}{(n_0+1)^2(n_1+1)^2}z, \quad \alpha=\alpha(n_0,n_1)\in\Lambda_1, \quad z_1=0.
$$
In this case, the family $((S(\alpha)-\alpha)z_\alpha)_{\alpha\in\Lambda^{<1}}$ is absolutely summable, and
$$
\underset{\alpha\in\Lambda^{<1}}{\Sigma}(S(\alpha)-\alpha)z_\alpha=\sum_{n_0=0}^\infty\sum_{n_1=0}^\infty\frac 1{(n_0+1)^2(n_1+1)^2}z
=\left(\frac{\pi^2}{6}\right)^2z.
$$
It follows from Theorem \ref{BochnerStep} that $A$ is strongly McShane product integrable. By Theorem \ref{T61}, we get
$$
\prod_0^1(I+A(t)\,{\rm d}t)=\exp\left(\int_0^1 A(t)\,{\rm d}t\right)
=\exp\left(\underset{\alpha\in\Lambda^{<1}}{\Sigma}(S(\alpha)-\alpha)z_\alpha\right)=\exp\left(\left(\frac{\pi^2}{6}\right)^2z\right).
$$
On the other hand, $A$ is not Riemann product integrable because the family $(z_\alpha)_{\alpha\in\Lambda_1}$ is unbounded.
\end{example}

\section{Product integrability of right regulated mappings}\label{S6}

In this section we study product integrability of  mappings $A$ from $[a,b]$ to a unital Banach algebra $E$ which are right regulated, 
i.e.,which have right limits at all points of $[a,b)$.
The main difference between right regulated mappings and regulated mappings, 
which have also left limits at every point of $(a,b]$, is that the former ones 
may have discontinuities of the second kind, while regulated mappings can have only discontinuities of the first kind. 
Another difference is that regulated mappings are always Riemann product integrable, 
whereas right regulated mappings need not be even Kurzweil product integrable. 

By \cite[Lemma 2.6]{SH13}, every right regulated mapping  is strongly measurable and has at most countably many discontinuities. Thus,
Theorem \ref{K-right-regulated} is applicable.
In this section we  provide additional necessary and sufficient conditions
for Kurzweil product integrability of right regulated mappings.
Our basic tool is the following lemma; it is a consequence of \cite[Lemma 2.5]{SH13} and its proof, 
which is based on a generalized iteration method presented in \cite{HL94}.

\begin{lemma}\label{L400} 
Let $A:[a,b]\to E$ be right regulated. Then for every $\varepsilon>0$, there is a  well-ordered set $\Lambda_\varepsilon\subset [a,b]$  
such that $[a,b)$ is a disjoint union of the intervals $[\beta,S(\beta))$, $\beta\in \Lambda_\varepsilon^{<b}$, 
and $\|A(s)-A(t)\|\le \varepsilon$ whenever $s,\,t\in (\beta,S(\beta))$ and $\beta\in \Lambda_\varepsilon^{<b}$.  

$\Lambda_\varepsilon$ is determined by the following properties:
\begin{equation}\label{E403}
a=\min \Lambda_\varepsilon,\ \hbox{ and $a <\gamma\in \Lambda_\varepsilon$ if and only if } \ \gamma=\sup\{G_\varepsilon(x); x\in\Lambda_\varepsilon^{< \gamma}\},
\end{equation}
where $G_\varepsilon:[a,b]\to[a,b]$ is defined by
\begin{equation}\label{E402}
G_\varepsilon(x) = \sup\{y\in(x,b];\, \|A(s)-A(t)\|\le \varepsilon\ \hbox{ for all } s,\,t\in (x,y)\}, \quad x\in[a,b),\quad\quad G_\varepsilon(b)=b. 
\end{equation}
\end{lemma}

For a right regulated mapping $A:[a,b]\to E$ and an arbitrary $\varepsilon>0$, we introduce the step mapping $A_\varepsilon:[a,b]\to E$ given by 
\begin{equation}\label{E405}
A_\varepsilon(t)=A(\beta+),\quad t\in[\beta,S(\beta)),\; \beta\in \Lambda_\varepsilon^{<b}, \quad\quad A_\varepsilon(b)=A(b).
\end{equation}
Note that $\|A_\varepsilon(t)-A(t)\|\le\varepsilon$ for all $t\in(\beta,S(\beta))$
and $\beta\in \Lambda_\varepsilon^{<b}$, i.e., for all $t\in[a,b]$ with countably many exceptions.
In this way, we can approximate right regulated mappings by step mappings. Moreover, the following results
show that this approximation preserves the existence or nonexistence of product integrals. Hence, we can use
criteria from Section \ref{S4} to study product integrability of right regulated mappings.

\smallskip 
Our first result provides necessary and sufficient conditions for strong Kurzweil product integrability of right regulated mappings.
The proof is inspired by the proof of \cite[Proposition 4.1]{SH13}; note however that it relies on Lemma~\ref{sum},
whose statement is far from obvious.

\begin{theorem}
Let $A:[a,b]\to E$ be a right regulated mapping. Given an arbitrary $\varepsilon>0$, 
let $\Lambda_\varepsilon$ be the well-ordered subset from Lemma \ref{L400}.
Then the following properties are equivalent:
\begin{enumerate} 
\item $A$ is strongly Kurzweil product integrable.
\item The step mapping $A_\varepsilon:[a,b]\to E$ given by \eqref{E405}
is strongly Kurzweil product integrable.
\item The family $(\exp((S(\beta)-\beta)A(\beta+)))_{\beta\in\Lambda_\varepsilon^{<b}}$ is multipliable and has an invertible product.
\end{enumerate}
\end{theorem}

\begin{proof}
The equivalence $2\Leftrightarrow 3$ follows immediately from Theorem \ref{Kurzweil-step}; it remains 
to prove the equivalence $1\Leftrightarrow 2$. Both $A_\varepsilon$ and $A$ are strongly measurable. 
We know that $\|A_\varepsilon(t)-A(t)\|\le\varepsilon$ for all $t\in(\beta,S(\beta))$
and $\beta\in \Lambda_\varepsilon^{<b}$. Consequently, the inequality 
$\|A_\varepsilon(t)-A(t)\|\le\varepsilon$ holds almost everywhere on $[a,b]$,
and $\|A_\varepsilon-A\|$ is Lebesgue integrable.
This means that both $A_\varepsilon-A$ and $A-A_\varepsilon$ are Bochner integrable.
According to Lemma~\ref{sum}, if $A$ is strongly Kurzweil product integrable, then $A_\varepsilon=A+(A_\varepsilon-A)$ is strongly Kurzweil product integrable;
conversely, if $A_\varepsilon$ is strongly Kurzweil product integrable, then $A=A_\varepsilon+(A-A_\varepsilon)$ is strongly Kurzweil product integrable.     
\end{proof}

The next theorem provides additional criteria applicable in the commutative case.

\begin{theorem}\label{P401} 
Let $A:[a,b]\to E$ be a right regulated mapping such that $A(t_1)A(t_2)=A(t_2)A(t_1)$ for all $t_1,t_2\in[a,b]$. Given an arbitrary $\varepsilon>0$, 
let $\Lambda_\varepsilon$ be the well-ordered subset from Lemma \ref{L400}.
Then the following properties are equivalent:
\begin{enumerate} 
\item $A$ is strongly Kurzweil product integrable.
\item $A$ is strongly Henstock-Kurzweil integrable.
\item The step mapping $A_\varepsilon:[a,b]\to E$ given by \eqref{E405}
 is  strongly Henstock-Kurzweil integrable.
\item The  family $((S(\beta)-\beta)A(\beta+))_{\Lambda_\varepsilon^{<b}}$ is summable.
\end{enumerate}
\end{theorem}

\begin{proof} 
The conditions 1 and 2 are equivalent by \cite[Theorems 3.12 and 3.13]{S15},
while conditions 2, 3 and 4 are equivalent by \cite[Proposition~4.1]{SH13}. 
\end{proof} 

The next result is concerned with strong McShane product integrability and extends \cite[Proposition 4.3]{SH13}.

\begin{theorem}\label{P402} 
Let $A:[a,b]\to E$ be a right regulated mapping. Given an arbitrary $\varepsilon>0$, 
let $\Lambda_\varepsilon$ be the well-ordered subset from Lemma \ref{L400}.
Then the following properties are equivalent:
\begin{enumerate}
\item  $A$ is strongly McShane product integrable.
\item  $A$ is Bochner integrable.
\item The step mapping $A_\varepsilon:[a,b]\to E$ given by \eqref{E405}
is Bochner integrable.
\item  The  family $((S(\beta)-\beta)A(\beta+))_{\Lambda_\varepsilon^{<b}}$ is absolutely summable.
\item   The  family $(\exp((S(\beta)-\beta)\|A(\beta+)\|))_{\Lambda_\varepsilon^{<b}}$ is  multipliable.
\item   The  family $(1+(S(\beta)-\beta)\|A(\beta+)\|)_{\Lambda_\varepsilon^{<b}}$ is  multipliable.
\end{enumerate}
\end{theorem}

\begin{proof}
The conditions 1 and 2 are equivalent by  \cite[Theorem 4.14]{S15}, conditions 2 and 3  are equivalent by \cite[Proposition 4.3]{SH13},
and conditions 3, 4, 5, 6 are equivalent by Theorem \ref{BochnerStep}. 
\end{proof}

\begin{remark}\label{R401}  
The  integrability results derived above have  analogous counterparts for left regulated mappings, i.e., for mappings
which have left limits at every point of $(a,b]$. 
\end{remark}

We now present an example of a product integrable right regulated mapping $A:[0,1]\to E$, which has discontinuities of the second kind at every rational point of $(0,1]$.  
The example is a slight modification of \cite[Example 4.1]{SH13}. We use the symbol $\left\lceil x \right\rceil$ to denote the smallest integer greater than or equal to $x$.

\begin{example}\label{Ex401} 
Let $E$ be an arbitrary unital Banach algebra $E$. Denote by $Z$ the set of all rational numbers in $[0,1]$, and
define a mapping $A:[0,1]\to E$ by $A(t)=I$ for all $t\in Z$, and 
\begin{equation}\label{E47}
	A(t)=\sum_{n=1}^\infty\frac 1{n^2}\left(2(nt-\left\lceil nt\right\rceil)\cos\left(\frac {\pi}{2(nt-\left\lceil nt\right\rceil)}\right)+\frac {\pi}{2}\sin\left(\frac{\pi}{2(nt-\left\lceil nt\right\rceil)}\right)\right)I, 
\end{equation}
for all $t\in [0,1]\setminus Z$.
For each $m\in\mathbb N$, denote 
$$
Z_m=\left\{\frac ij;\, j\in\{1,\dots,m\}, \ 0\le i\le m\right\}.
$$ 
Let $A_m:[0,1]\to E$ be a mapping such that $A_m(t)=I$ for all $t\in Z_m$, and $A_m(t)$ is the $m$-th partial sum  of the series (\ref{E47}) for all $t\in[0,1]\setminus Z_m$. It is easy to verify that $Z_m$ is the set of all discontinuity points of~$A_m$.
Thus $A$ is discontinuous at every point of the set $\cup_{m=1}^\infty Z_m$, which is the set $Z$ of all rational numbers in $[0,1]$. Moreover, if $t\in Z\setminus\{0\}$, then the sine term on the right hand side of (\ref{E47}) does not have a left limit at $t$, i.e., $t$ is a discontinuity of the second kind. 
On the other hand, for each $t\in [0,1]\setminus Z$, the functions $A_m$ are continuous at $t$. Since $(A_m)_{m=1}^\infty$ is uniformly convergent to $A$ on $[0,1]$, it follows that $A$ is continuous at each point $t\in [0,1]\setminus Z$. Similarly, since each $A_m$ has a right limit at all  points of $[0,1)$, then $A$ has the same property, i.e., it is right regulated.  Because $A$ is bounded, it is Riemann integrable, and hence also Riemann product integrable.
\end{example}

\begin{example} Define a mapping $B:[0,1]\to E$ by   $B(1)=I$, and 
	$$
	B(t)=\left(\cos\left(\frac{1}{t-1}\right) +\frac{\sin(\frac{1}{t-1})}{t-1}\right)I, \quad t\in[0,1). 
	$$
	The only discontinuity point of $B$ is $1$. The mapping $C:[0,1]\to E$ defined by  
	$$
	C(t)=(t-1)\cos\left(\frac 1{t-1}\right)I, \quad t\in[0,1), \quad C(1)=0, 
	$$ is continuous on $[0,1]$ and $C'(t)=B(t)$ for all $t\in[0,1)$. Consequently, $B$ is strongly Henstock-Kurzweil integrable with $\int_0^1 B(t)\,{\rm d}t=C(1)-C(0)$, but neither Riemann integrable (since $B$ is unbounded) nor Bochner integrable (since $C$ is not absolutely continuous).
	
	Let $A$ be the mapping from Example \ref{Ex401}. Then $A+B$
	is right regulated and has the same discontinuity points as $A$. It is  strongly Kurzweil product integrable, but neither Riemann product integrable nor strongly McShane product integrable.
\end{example}

\section{Stieltjes product integrability}\label{3}

This section is devoted to Stieltjes-type product integrals of the form $\prod_a^b(I+{\rm d}A(t))$, which are defined as follows.

\begin{definition}\label{KS-prod-int}
A mapping $A:[a,b]\to E$ is called Kurzweil-Stieltjes/Riemann-Stieltjes product integrable if the function $V:[a,b] \times \mathcal I \to E$
given by $V(t,[x,y])=I+A(y)-A(x)$ is Kurzweil/Riemann product integrable in the sense of Definition \ref{E4def1}.
In this case, the Kurzweil-Stieltjes/Riemann-Stieltjes product integral of $A$ is defined as $\prod_a^b(I+{\rm d}A(t))={\prod}_a^b V(t,{\rm d}t)$.

$A$ is called strongly Kurzweil-Stieltjes product integrable if $V$ is strongly Kurzweil product integrable 
in the sense of Definition \ref{strongDef}. 
\end{definition}

We begin by recalling an elegant criterion for Riemann-Stieltjes product integrability, which was derived by R.~M.~Dudley and R.~Norvai\v{s}a
and is based on the notion of $p$-variation.

\smallskip

Given a mapping $A:[a,b]\to E$ and a number $p>0$, the $p$-variation
of $A$ is defined as 
$$\sup\left\{\left(\sum_{i=1}^m\|A(t_i)-A(t_{i-1})\|^p\right)^{\frac{1}{p}}\!\!;\;\{t_i\}_{i=0}^m\mbox{ is a partition of }[a,b]\right\}.$$ 

It is known that each mapping with finite $p$-variation, for some $p\in(0,\infty)$, is regulated (see \cite[Lemma 2.4]{DN}).
We use the notation $\Delta^+A(t)=A(t+)-A(t)$ for $t\in[a,b)$, and $\Delta^-A(t)=A(t)-A(t-)$, $t\in(a,b]$. 

The next theorem  combines \cite[Theorem 4.26]{DN} and \cite[Proposition 4.30]{DN}.

\begin{theorem}\label{L31}
Assume that $A:[a,b]\to E$ satisfies the following conditions:
\begin{enumerate}
\item $A$ has a finite $p$-variation for a certain $p\in(0,2)$.
\item $A$ is left-continuous or right-continuous at each point of $(a,b)$.
\item $A$ is regulated, $I+\Delta^+A(t)$ is invertible for all $t\in[a,b)$, and $I+\Delta^-A(t)$ is invertible for all $t\in(a,b]$.
\end{enumerate}
Then $A$ is Riemann-Stieltjes product integrable.
\end{theorem}

For real-valued functions, we have the following necessary condition for Riemann-Stieltjes product integrability; see \cite[Theorem 4.3]{DN}.

\begin{theorem}\label{2variation}
	If $f:[a,b]\to\mathbb R$ is Riemann-Stieltjes product integrable, then it has finite 2-variation.
\end{theorem}

The next theorem, which is a combination of Theorem 4.4 and Proposition 2.23 from \cite{DN}, provides a complete characterization of real-valued  Riemann-Stieltjes product integrable functions.

\begin{theorem}\label{scalarRS}
	A function $f:[a,b]\to\mathbb R$ is Riemann-Stieltjes product integrable if and only if the following conditions are satisfied:
	\begin{enumerate}
		\item $f$ is regulated, $1+\Delta^+f(t)\ne 0$ for all $t\in[a,b)$, and $1+\Delta^-f(t)\ne 0$ for all $t\in(a,b]$.
		\item If $\{\xi_j\}_{j}$ is a sequence consisting of all discontinuity points of $f$, then 
		$$\sum_{j} (|\Delta^-f(\xi_j)|^2+|\Delta^+f(\xi_j)|^2)<\infty.$$
		\item For each $\varepsilon>0$, there exists a partition $D$ of $[a,b]$ such that if $a=y_0<y_1<\cdots<y_m=b$ is a~refinement of $D$, then
		$$\sum_{i=1}^m|f(y_i-)-f(y_{i-1}+)|^2<\varepsilon.$$
	\end{enumerate}
\end{theorem}

We now turn our attention to Kurzweil-Stieltjes product integrals. If $F$ is a continuously differentiable mapping with $F'=f$, it is reasonable to expect that the Stieltjes product integral $\prod_a^b(I+{\rm d}F(t))$ can be reduced to the ordinary product integral $\prod_a^b(I+f(t)\,{\rm d}t)$. In fact, the statement holds under the following weaker assumptions on $F$.

\begin{theorem}\label{substitution}
	Assume that $f:[a,b]\to E$ is strongly Henstock-Kurzweil integrable and $F(t)=\int_a^t f(s)\,{\rm d}s$, $t\in[a,b]$. Then the strong Kurzweil-Stieltjes product integral $\prod_a^b(I+{\rm d}F(t))$ exists if and only if the strong Kurzweil product integral $\prod_a^b(I+f(t)\,{\rm d}t)$ exists; in this case, both integrals have the same value.
\end{theorem}

\begin{proof}
	Consider an arbitrary $\varepsilon>0$. Since $f$ is strongly Henstock-Kurzweil integrable, there exists a gauge $\delta_0:[a,b]\to\mathbb R^+$ such that
	$$\sum_{i=1}^m\|f(\xi_i)(t_i-t_{i-1})-(F(t_i)-F(t_{i-1}))\|<\varepsilon$$
	for each $\delta_0$-fine partition of $[a,b]$.
	
	If the strong Kurzweil-Stieltjes product integral $\prod_a^b(I+{\rm d}F(t))$ exists, there is a gauge $\delta:[a,b]\to\mathbb R^+$ such that
	$$\sum_{i=1}^m\|I+F(t_i)-F(t_{i-1})-W(t_i)W(t_{i-1})^{-1}\|<\varepsilon$$
	holds for each $\delta$-fine partition of $[a,b]$, where $W(t)=\prod_a^t(I+{\mathrm d}F(s))$, $t\in[a,b]$. Without loss of generality, we can assume that $\delta\le\delta_0$ on $[a,b]$. Consequently,
	$$\sum_{i=1}^m\|I+f(\xi_i)(t_i-t_{i-1})-W(t_i)W(t_{i-1})^{-1}\|$$
	$$\le\sum_{i=1}^m\|I+F(t_i)-F(t_{i-1})-W(t_i)W(t_{i-1})^{-1}\|+\sum_{i=1}^m\|f(\xi_i)(t_i-t_{i-1})-(F(t_i)-F(t_{i-1}))\|<2\varepsilon,$$
	which proves that the strong Kurzweil product integral $\prod_a^b(I+f(t)\,{\rm d}t)$ exists and equals $W(b)W(a)^{-1}=\prod_a^b(I+{\mathrm d}F(t))$.
	 
	Conversely, if the strong Kurzweil product integral $\prod_a^b(I+f(t)\,{\rm d}t)$ exists, there is a gauge $\delta:[a,b]\to\mathbb R^+$ such that
	$$\sum_{i=1}^m\|I+f(\xi_i)(t_i-t_{i-1})-W(t_i)W(t_{i-1})^{-1}\|<\varepsilon$$
	holds for each $\delta$-fine partition of $[a,b]$, where $W(t)=\prod_a^t(I+f(s)\,{\mathrm d}s)$, $t\in[a,b]$. Without loss of generality, we can assume that $\delta\le\delta_0$ on $[a,b]$. Consequently,
	$$\sum_{i=1}^m\|I+F(t_i)-F(t_{i-1})-W(t_i)W(t_{i-1})^{-1}\|$$
	$$\le\sum_{i=1}^m\|I+f(\xi_i)(t_i-t_{i-1})-W(t_i)W(t_{i-1})^{-1}\|+\sum_{i=1}^m\|F(t_i)-F(t_{i-1})-f(\xi_i)(t_i-t_{i-1})\|<2\varepsilon,$$
	and therefore the strong Kurzweil-Stieltjes product integral $\prod_a^b(I+{\rm d}F(t))$ exists and equals $W(b)W(a)^{-1}=\prod_a^b(I+f(t)\,{\rm d}t)$.
\end{proof}

We now provide a simple example of a Kurzweil-Stieltjes product integrable mapping, which is not Riemann-Stieltjes product integrable.

\begin{example}
	Let $A:[0,1]\to E$ be given by 
	$$A(t)=\begin{cases}
	\left(\sqrt{t}\cos\frac{\pi}{t}\right)I,&t\ne 0,\\
	0,&t=0.
	\end{cases}
$$
This mapping is essentially a scalar one: We have $A(t)=F(t)I$, where
	$$F(t)=\begin{cases}
	\sqrt{t}\cos\frac{\pi}{t},&t\ne 0,\\
	0,&t=0.
	\end{cases}
	$$
$F$ is continuous on $[0,1]$ and differentiable on $(0,1]$. Let $f(t)=F'(t)$, $t\in(0,1]$, and $f(0)=0$. Then $f$ is Henstock-Kurzweil integrable and $F(t)=\int_0^t f(s)\,{\mathrm d}s$, $t\in[0,1]$. Since $\mathbb R$ is a commutative algebra, it follows from \cite[Theorem 3.7]{JK} that $f$ is Kurzweil product integrable and 
$$\prod_0^1(1+f(t)\,{\rm d}t)=\exp\left(\int_0^1 f(t)\,{\mathrm d}t\right)=\exp(F(1)-F(0))=\exp(-1).$$
 According to Theorem \ref{substitution}, $F$ is Kurzweil-Stieltjes product integrable and
$$\prod_0^1(1+{\rm d}F(t))=\prod_0^1(1+f(t)\,{\rm d}t)=\exp(-1).$$
On the other hand, $F$ is not Riemann-Stieltjes product integrable. To see this, choose an arbitrary $n\ge 2$ and consider a partition of $[0,1]$ consisting of division points $t_0=0$ and $t_i=\frac{1}{i}$, $i\in\{1,\ldots,n\}$. Then 
$$F(t_i)=\frac{(-1)^i}{\sqrt{i}},\quad i\in\{1,\ldots,n\},$$
$$\left|F(t_i)-F(t_{i-1})\right|=\frac{1}{\sqrt{i}}+\frac{1}{\sqrt{i-1}},\quad i\in\{2,\ldots,n\},$$
$$\left|F(t_i)-F(t_{i-1})\right|^2=\frac{1}{i}+\frac{2}{\sqrt{i(i-1)}}+\frac{1}{i-1}>\frac{1}{i},\quad i\in\{2,\ldots,n\},$$
$$\sum_{i=1}^n\left|F(t_i)-F(t_{i-1})\right|^2=1+\sum_{i=2}^n\left|F(t_i)-F(t_{i-1})\right|^2>\sum_{i=1}^n\frac{1}{i}.$$
Since $n$ can be arbitrarily large, we see that the $2$-variation of $F$ is infinite. By Theorem \ref{2variation}, $F$ is not Riemann-Stieltjes product integrable.

Taking into account the relation between $A$ and $F$, it is easily seen that $A$ is (strongly) Kurzweil-Stieltjes product integrable with $\prod_0^1(I+{\rm d}A(t))=\exp(-1)I$, but not Riemann-Stieltjes product integrable. 
\end{example}

The rest of this section will be devoted to Stieltjes product integrals of step mappings. We start with a~simple example.

\begin{example}\label{KS-example}
Let $z_a$, $z_b\in E$ be arbitrary and consider the function $A:[a,b]\to E$ given by $A(t)=z_a$ for $t\in[a,b)$, and $A(b)=z_b$.
Then the Riemann-Stieltjes and Kurzweil-Stieltjes product integrals $\prod_a^b(I+{\rm d}A(t))$ exist if and only if $I+z_b-z_a$ is invertible; this is an easy consequence of the fact that for an arbitrary partition $a=t_0<t_1<\cdots<t_m=b$, we have $\prod_{i=m}^1(I+A(t_i)-A(t_{i-1}))=I+z_b-z_a$. 

If $I+z_b-z_a$ is invertible, then $A$ is also strongly Kurzweil-Stieltjes product integrable. To see this, define $W:[a,b]\to E$ by $W(t)=I$ for $t\in[a,b)$ and $W(b)=I+z_b-z_a$.
For an arbitrary partition $a=t_0<t_1<\cdots<t_m=b$, we have
$$
\sum_{i=1}^m\|I+A(t_i)-A(t_{i-1})-W(t_i)W(t_{i-1})^{-1}\|=\|I+A(t_m)-A(t_{m-1})-W(t_m)W(t_{m-1})^{-1}\|=0,
$$
which proves that the strong Kurzweil-Stieltjes integral $\prod_a^b(I+{\rm d}A(t))$ exists and equals $I+z_b-z_a$. 
\end{example} 

Next, we focus on more complicated step mappings having the form \eqref{E30}.

\begin{theorem}\label{T31} Let $A:[a,b]\to E$  be a step mapping with representation \eqref{E30}. Assume that 
for each limit element $\gamma\in\Lambda$,  ${\lim}_{\beta\to\gamma-}(I+z_{\gamma}-z_{\beta})$ exists and is invertible. 
Then the following conditions are equivalent: 
\begin{enumerate}
\item $A$ is strongly Kurzweil-Stieltjes product integrable.
\item $A$ is Kurzweil-Stieltjes product integrable.
\item The family  $(x_\alpha)_{\alpha\in\Lambda}$ given by
\begin{equation}\label{familyDef}
x_\alpha=\begin{cases}
I& \mbox{ if }\alpha=a,\\
I+z_{\alpha}-z_{\beta}& \mbox{ if $\alpha=S(\beta)$},\\
\lim\limits_{\beta\to\alpha-}(I+z_{\alpha}-z_{\beta})&  \mbox{ if $\alpha$ is a limit element}
\end{cases}
\end{equation}
 is multipliable,  and its elements as well as its product are invertible.
\end{enumerate} 
If any of these conditions is satisfied, we have
$$\prod_a^b(I+{\rm d}A(t))=\underset{\alpha\in\Lambda}{\Pi}x_\alpha.$$
\end{theorem}

\begin{proof}
Let us start by checking whether the function $V:[a,b] \times \mathcal I \to E$ given by $V(t,[x,y])=I+A(y)-A(x)$ satisfies conditions (V1)--(V4) from Section 3.

The statement of (V1) is obviously true. To prove that condition (V2) holds, assume first that 
$t\in [a,b)$. Because $A$ has the representation \eqref{E30}, then $t\in[\alpha,S(\alpha))$ for some $\alpha\in\Lambda$. 
Choosing $\sigma=S(\alpha)-t$, then  $A(y)=A(t)=z_\alpha$ when $t\le y < t+\sigma$, whence  
\begin{equation}\label{V2check}
\|I+A(y)-A(x)-(I+A(y)-A(t))(I+A(t)-A(x))\|=0
\end{equation}
for all $x,y\in[a,b]$ such that $t-\sigma<x\le t\le y< t+\sigma$.
Eq.~\eqref{V2check} also holds when $t=b$, since then $y=t$ and $A(y)=A(t)=z_b$.
This proves that (V2) is satisfied.

Since $A$ is right-continuous, it follows immediately that condition (V3) holds with $V_+(t)=I$. 
To prove~(V4), assume first that $t\in[a,b]\setminus\Lambda$. Then 
$t\in(\beta,S(\beta))$ for some $\beta\in\Lambda$, and 
$$V_-(t)=\underset{x\to t-}{\lim}(I+A(t)-A(x))=I+z_\beta-z_\beta=I$$ 
is invertible.
Next, if $t=\gamma$ for some limit element $\gamma\in\Lambda$, then 
$$V_-(t)=\underset{x\to t-}{\lim}(I+A(t)-A(x))=\underset{\beta\to\gamma-}{\lim}(I+z_\gamma-z_\beta),$$
and the last limit exists and is invertible.
Assume finally that $t=\gamma\in\Lambda$ is a successor, say $\gamma=S(\beta)$, $\beta\in \Lambda$. Then
$$V_-(t)=\underset{x\to t-}{\lim}(I+A(t)-A(x))=I+z_{\gamma}-z_\beta.$$
We do not a priori know whether the last element is invertible. However, if condition 1 or 2 is satisfied, then the product integral $\prod_\beta^\gamma(I+{\rm d}A(t))$ exists, and it follows 
from Example \ref{KS-example} that $I+z_{\gamma}-z_\beta$ has to be invertible. Also, if condition 3 is satisfied, then $I+z_{\gamma}-z_\beta$ is obviously invertible.  
This shows that condition (V4) is satisfied if at least one of the conditions 1, 2, and 3 holds.

Now, let us show that conditions 1, 2, and 3 are equivalent. 

We begin with the implication $3\Rightarrow 1$. We use transfinite induction to prove that for every $\gamma\in\Lambda$,
the strong product integral $\prod_a^\gamma(I+{\rm d}A(t))$ exists and   
$$\prod_a^\gamma(I+{\rm d}A(t))=\underset{\alpha\in\Lambda^{\le\gamma}}{\Pi}x_\alpha.$$
The statement is obvious for $\gamma=a$.
Next, we make an induction hypothesis: Suppose that $\gamma\in\Lambda\setminus\{a\}$ and
$\prod_a^\beta(I+{\rm d}A(t))=\underset{\alpha\in\Lambda^{\le\beta}}{\Pi}x_\alpha$ for every $\beta\in\Lambda^{<\gamma}$.
 
Assume first that $\gamma$ is a successor, i.e., $\gamma=S(\beta)$ for a certain $\beta\in\Lambda$. By Example \ref{KS-example}, the strong Kurzweil-Stieltjes product integral $\prod_\beta^{\gamma}(I+{\rm d}A(t))$ exists and equals $I+z_{\gamma}-z_\beta$. Consequently,
$$\prod_a^\gamma(I+{\rm d}A(t))=\prod_\beta^\gamma(I+{\rm d}A(t))\prod_a^\beta(I+{\rm d}A(t))
=(I+z_{\gamma}-z_\beta)\underset{\alpha\in\Lambda^{\le\beta}}{\Pi}x_\alpha=x_\gamma\underset{\alpha\in\Lambda^{\le\beta}}{\Pi}x_\alpha
=\underset{\alpha\in\Lambda^{\le\gamma}}{\Pi}x_\alpha.$$

Assume next that $\gamma\in\Lambda$ is a limit element. 
Then $\underset{\alpha\in\Lambda^{<\gamma}}{\Pi}x_\alpha=\lim\limits_{\beta\to\gamma-}\left(\underset{\alpha\in\Lambda^{<\beta}}{\Pi}x_\alpha\right)$.

For an arbitrary $s\in[a,\gamma)$, there is a  $\beta\in\Lambda$ such that $s\in[\beta,S(\beta))$.
Since $A$ is constant on $[\beta,S(\beta))$, we have $\prod_\beta^s(I+{\rm d}A(t))=I$, and
$$
\prod_a^s(I+{\rm d}A(t))=\prod_\beta^s(I+{\rm d}A(t))\cdot\prod_a^\beta(I+{\rm d}A(t))=\prod_a^\beta(I+{\rm d}A(t))
=\underset{\alpha\in\Lambda^{\le\beta}}{\Pi}x_\alpha.
$$ 
Consequently,
\begin{equation}\label{E33}
\underset{s\to\gamma-}{\lim}\prod_a^s(I+{\rm d}A(t)) 
=\lim\limits_{\beta\to\gamma-}\left(\underset{\alpha\in\Lambda^{\le\beta}}{\Pi}x_\alpha\right)
=\underset{\beta\to\gamma-}{\lim}\left(x_\beta\underset{\alpha\in\Lambda^{<\beta}}{\Pi}x_\alpha\right)
=\lim\limits_{\beta\to\gamma-}\left(\underset{\alpha\in\Lambda^{<\beta}}{\Pi}x_\alpha\right)=\underset{\alpha\in\Lambda^{<\gamma}}{\Pi}x_\alpha,
\end{equation}
where the third equality follows from Lemma \ref{L22}.
Note that by Lemma \ref{invertible-partial-products}, the product on the right-hand side of \eqref{E33} is invertible.
These facts imply that
\begin{equation}\label{E34}
\underset{s\to\gamma-}{\lim}\left((I+A(\gamma)-A(s))\cdot\prod_a^s(I+{\rm d}A(t))\right)
=\left(\underset{\beta\to\gamma-}{\lim}(I+z_{\gamma}-z_{\beta})\right)\underset{\alpha\in\Lambda^{<\gamma}}{\Pi}x_\alpha
=x_\gamma\underset{\alpha\in\Lambda^{<\gamma}}{\Pi}x_\alpha,
\end{equation}
and the right-hand side is invertible. According to Theorem \ref{strongHake},  
the strong Kurzweil-Stieltjes product integral  $\prod_a^\gamma(I+{\rm d}A(t))$ exists and equals 
$x_\gamma\underset{\alpha\in\Lambda^{<\gamma}}{\Pi}x_\alpha=\underset{\alpha\in\Lambda^{\le \gamma}}{\Pi}x_\alpha$,
which completes the proof by transfinite induction.

The implication $1\Rightarrow2$ follows immediately from Theorem \ref{K-prod-existence}. 

It remains to verify the implication $2\Rightarrow 3$. We use transfinite induction to prove that for every $\gamma\in\Lambda$,
the family $(x_\alpha)_{\alpha\in\Lambda^{\le\gamma}}$ is multipliable, and   
$$\underset{\alpha\in\Lambda^{\le\gamma}}{\Pi}x_\alpha=\prod_a^\gamma(I+{\rm d}A(t)).$$
The statement is obvious for $\gamma=a$. Next, we make an induction hypothesis: Suppose that $\gamma\in\Lambda\setminus\{a\}$ and for every $\beta\in\Lambda^{<\gamma}$,
$(x_\alpha)_{\alpha\in\Lambda^{\le\beta}}$ is multipliable and its product equals $\prod_a^\beta(I+{\rm d}A(t))$.

Assume first that $\gamma$ is a successor, i.e., $\gamma=S(\beta)$ for a certain $\beta\in\Lambda$. Then 
$(x_\alpha)_{\alpha\in\Lambda^{\le\gamma}}$ is obviously multipliable. By Example \ref{KS-example},
we have $\prod_\beta^{\gamma}(I+{\rm d}A(t))=I+z_{\gamma}-z_\beta$. Therefore, $x_\gamma=I+z_{\gamma}-z_\beta$ has to be invertible, and 
$$\underset{\alpha\in\Lambda^{\le\gamma}}{\Pi}x_\alpha=x_\gamma\underset{\alpha\in\Lambda^{\le\beta}}{\Pi}x_\alpha
=(I+z_{\gamma}-z_\beta)\underset{\alpha\in\Lambda^{\le\beta}}{\Pi}x_\alpha=
\prod_\beta^\gamma(I+{\rm d}A(t))\prod_a^\beta(I+{\rm d}A(t))=\prod_a^\gamma(I+{\rm d}A(t)).$$
Assume next that $\gamma\in\Lambda$ is a limit element. 
Then we claim that the family $(x_\alpha)_{\alpha\in\Lambda^{<\gamma}}$ is multipliable with its product being equal to
$$\underset{\alpha\in\Lambda^{<\gamma}}{\Pi}x_\alpha=\underset{\beta\to\gamma-}{\lim}\prod_a^\beta(I+{\rm d}A(t))$$
(the existence of the limit is guaranteed by Theorem \ref{indefinite-regulated}).
Indeed, the previous equality together with the induction hypothesis imply that
$$\underset{\alpha\in\Lambda^{<\gamma}}{\Pi}x_\alpha=\lim_{\beta\to\gamma-}\left(\underset{\alpha\in\Lambda^{\le\beta}}{\Pi}x_\alpha\right)
=\lim\limits_{\beta\to\gamma-}\left(\underset{\alpha\in\Lambda^{<\beta}}{\Pi}x_\alpha\right),$$
and therefore condition (ii) of Definition \ref{D1} is satisfied. Consequently,
the family $(x_\alpha)_{\alpha\in\Lambda^{\le\gamma}}$ is multipliable as well, and we get
$$\underset{\alpha\in\Lambda^{\le\gamma}}{\Pi}x_\alpha=x_\gamma\underset{\alpha\in\Lambda^{<\gamma}}{\Pi}x_\alpha
=\underset{\beta\to\gamma-}{\lim}(I+z_{\gamma}-z_{\beta})\underset{\beta\to\gamma-}{\lim}\prod_a^\beta(I+{\rm d}A(t))=V_-(\gamma)\underset{\beta\to\gamma-}{\lim}\prod_a^\beta(I+{\rm d}A(t))
=\prod_a^\gamma(I+{\rm d}A(t)),$$
where the last equality follows from Theorem \ref{indefinite-regulated}); this completes the proof by transfinite induction.
\end{proof}

\begin{remark}
In Theorem \ref{T31}, we encountered the assumptions that $I+z_{\alpha}-z_{\beta}$ is invertible whenever $\alpha=S(\beta)$,
and that $\underset{\beta\to\alpha-}{\lim}(I+z_{\alpha}-z_{\beta})$ exists and is invertible whenever $\alpha$ is a limit element. 
In terms of the step mapping~$A$, one can equivalently say that $I+\Delta^-A(t)$ exists and is invertible for all $t\in(a,b]$;
note that the symmetric expression $I+\Delta^+A(t)$ is always invertible since $A$ is right-continuous.
\end{remark}

The following consequence of Theorem \ref{T31} is useful in applications.

\begin{corollary}\label{C31} Let $A:[a,b]\to E$  be a step mapping with representation \eqref{E30}. Assume that 
for each limit element $\gamma\in\Lambda$, we have $\underset{\beta\to\gamma-}{\lim}z_\beta=z_\gamma$.
Then the following conditions are equivalent: 
\begin{enumerate}
\item $A$ is strongly Kurzweil-Stieltjes product integrable.
\item $A$ is Kurzweil-Stieltjes product integrable.
\item The family  $(x_\alpha)_{\alpha\in\Lambda}$ given by
\begin{equation}\label{E601}\begin{cases}
x_\alpha=I & \mbox{ if $\alpha=a$, or if $\alpha$ is a limit element},\\
x_{S(\beta)}=I+z_{S(\beta)}-z_{\beta}& \mbox{ if $\beta\in\Lambda^{<b}$}
\end{cases}
\end{equation}
 is multipliable,  and its elements as well as its product are invertible.
\end{enumerate} 
If any of these conditions is satisfied, we have
$$\prod_a^b(I+{\rm d}A(t))=\underset{\alpha\in\Lambda}{\Pi}x_\alpha.$$
\end{corollary}

In the following examples, we make the convention that $\sum_{n=1}^0 x_n=0$. 
The first example provides an illustration of Corollary \ref{C31} and Theorem \ref{L31}. 

\begin{example}\label{Ex32} Let $q\in(0,2)$, $C>\frac{1}{2}\left(\frac{2}{3}\right)^{1/q}$, and an interval $[a,b]\subset\mathbb R$ be given. Define    
\begin{equation*}
\begin{aligned}
&\Lambda=\{\alpha(n)=b-2^{-n}(b-a);\,n\in\mathbb N_0\}\cup\{b\},\ \hbox{ and}\\ 
&z_\alpha=z_{\alpha(n)}=\sum_{k=1}^n\frac{(-1)^{k+1}}{C\left(k+\frac{(-1)^{k+1}}2\right)^{1/q}+\frac{(-1)^k}2}I, \quad \alpha=\alpha(n)\in \Lambda^{<b}.
\end{aligned}
\end{equation*}
The fractions in the last sum make sense because the choice of $C$ ensures that the denominator of every fraction is positive.  

The only limit element of $\Lambda$ is $b$. Let 
$$
z_b=\underset{\alpha\to b-}{\lim}z_\alpha=\underset{n\to\infty}{\lim}z_{\alpha(n)}=\sum_{k=1}^\infty\frac{(-1)^{k+1}}{C\left(k+\frac{(-1)^{k+1}}2\right)^{1/q}+\frac{(-1)^k}2}I.
$$
We claim that the last series is convergent.
Indeed, the terms of this series approach zero as $k\to\infty$, and by summing pairs of consecutive terms corresponding to $k=2n-1$ and $k=2n$, $n\in\mathbb N$, we get
$$z_b=\sum_{n=1}^\infty \frac 1{C^2\left(2n-\frac 12\right)^{2/q}-\frac 14}I,$$
which is finite because $q\in(0,2)$.

The mapping $A:[a,b]\to E$ defined by
\begin{equation*}
A(t)= z_{\alpha(n)}, \quad t\in[\alpha(n),\alpha(n+1)), \quad n\in\mathbb N_0,\quad 
A(b)=z_b,
\end{equation*}
has the representation \eqref{E30}, and the  hypothesis of Corollary \ref{C31} is satisfied. 
We are now going to show that condition 3 of this corollary holds.
Indeed, consider the family $(x_\alpha)_{\alpha\in\Lambda}$ given by \eqref{E601}.
Because $S(\alpha(n))=\alpha(n+1)$, $n\in\mathbb N_0$, it follows that
\begin{equation}\label{E313}
x_{S(\alpha(n))}=I+z_{S(\alpha(n))}-z_{\alpha(n)}=\left(1+\frac{(-1)^{n}}{C\left(n+1+\frac{(-1)^{n}}2\right)^{1/q}+\frac{(-1)^{n+1}}2}\right)I, \quad n\in\mathbb N_0.
\end{equation}
The assumption $C>\frac{1}{2}\left(\frac{2}{3}\right)^{1/q}$ guarantees that the last element is a nonzero multiple of $I$, i.e., it is invertible.

Next, observe that 
\begin{equation}\label{E314}
\begin{aligned}
&\prod_{n=0}^\infty\left(1+\frac{(-1)^{n}}{C(n+1+\frac{(-1)^{n}}2))^{1/q}+\frac{(-1)^{n+1}}2}\right)\\
&=\left(1+\frac{1}{C(\frac 32)^{1/q}-\frac 12}\right)\left(1-\frac{1}{C(\frac 32)^{1/q}+\frac 12}\right)\left(1+\frac{1}{C(\frac 72)^{1/q}-\frac 12}\right)\left(1-\frac{1}{C(\frac 72)^{1/q}+\frac 12}\right)\cdots=1, 
\end{aligned}
\end{equation}
because the terms of this product approach 1 as $n\to\infty$, and the products of consecutive pairs of terms are equal to 1.
Thus it follows from \eqref{E313} and \eqref{E314} that  $(x_\alpha)_{\alpha\in{\Lambda^{<b}}}$ is multipliable and
$$
\underset{\alpha\in\Lambda^{<b}}{\Pi}x_\alpha=I. 
$$
By Corollary \ref{C31}, $A$ is strongly Kurzweil-Stieltjes product integrable, and 
$$
\prod_a^b(I+{\rm d}A(t))=\underset{\alpha\in\Lambda^{\le b}}{\Pi}x_\alpha=\underset{\alpha\in\Lambda^{<b}}{\Pi}x_\alpha=I.
$$
Since
$$\|z_{S(\alpha(n))}-z_{\alpha(n)}\|=\frac{1}{C\left(n+1+\frac{(-1)^{n}}2\right)^{1/q}+\frac{(-1)^{
n+1}}2}, \quad n\in\mathbb N_0,$$
it is not difficult to see that the $p$-variation of $A$ is finite when $0<p\le q$ and infinite when $p>q$.
Hence, by Theorem \ref{L31}, $A$ is also Riemann-Stieltjes product integrable.
\end{example}

Next, we present an example of a mapping $A:[a,b]\to E$ that does not have finite $p$-variation for any\break $p\in(0,2)$, but is both Riemann-Stieltjes product integrable and strongly Kurzweil-Stieltjes product integrable. The example provides an illustration of Theorem \ref{scalarRS} and Corollary \ref{C31}.

\begin{example}\label{Ex33} Let $\Lambda=\{\alpha(n)=b-2^{-n}(b-a);\,n\in\mathbb N_0\}\cup\{b\}$. 
Consider the family $(z_\alpha)_{\alpha\in\Lambda}$ given by
\begin{eqnarray}\label{E302}
z_\alpha&=&z_{\alpha(n)}=\sum_{k=1}^n\frac{(-1)^{k+1}}{\sqrt{k+1}\log(k+1)}I, \quad \alpha=\alpha(n)\in \Lambda^{<b},\\
z_b&=&\underset{\alpha\to b-}{\lim}z_\alpha=\underset{n\to\infty}{\lim}z_{\alpha(n)}=\sum_{k=1}^\infty{\frac{(-1)^{k+1}}{\sqrt{k+1}\log(k+1)}}I
\end{eqnarray}
According to the Leibniz criterion, the last infinite series  is convergent.
Define the mapping $A:[a,b]\to E$ by
$$
\begin{aligned}
&A(t)= z_{\alpha(n)}, \quad [\alpha(n),\alpha(n+1)), \quad n\in\mathbb N_0,\\ 
&A(b)=z_b.
\end{aligned}
$$
Since
$$\|z_{S(\alpha(n))}-z_{\alpha(n)}\|=\frac{1}{\sqrt{n+2}\log(n+2)}, \quad n\in\mathbb N_0,$$
it is not difficult to see that the $p$-variation of $A$ is finite for $p\ge 2$ and infinite for $p\in(0,2)$. Thus the assumptions of  Theorem \ref{L31} are not satisfied. Nevertheless, $A$ is Riemann-Stieltjes product integrable. Since $A(t)=f(t)I$, where $f$ is real-valued, it is enough to show that $f$ is Riemann-Stieltjes product integrable. Let us verify that $f$ satisfies the conditions of Theorem \ref{scalarRS}.

$f$ is right-continuous, and has discontinuities from the left at the points $\alpha(n)$, $n\in\mathbb N$. We have $$1+\Delta^-f(\alpha(n))=1+\frac{(-1)^{n+1}}{\sqrt{n+1}\log(n+1)}\ne 0,\quad n\in\mathbb N,$$
which shows that the first condition of Theorem \ref{scalarRS} is satisfied. To verify the second condition, it is enough to observe that
$$\sum_{n=1}^\infty |\Delta^-f(\alpha(n))|^2=\sum_{n=1}^\infty \frac{1}{(n+1)\log^2(n+1)}<\infty.$$
Finally, consider an arbitrary $\varepsilon>0$. There is an $n_0\in\mathbb N$ such that 
$$\sum_{n=n_0}^\infty \frac{1}{(n+1)\log^2(n+1)}<\varepsilon.$$
Denote by $D$ the partition of $[a,b]$ consisting of the division points $t_i=\alpha(i)$, $i\in\{0,\ldots,n_0\}$, and $t_{n_0+1}=b$.
Let $a=y_0<y_1<\cdots<y_m=b$ be a~refinement of $D$.
Note that for each $i$ with $y_i\le \alpha(n_0)$, the function $f$ is constant on $(y_{i-1},y_i)$. Therefore,
\begin{align*}
&\sum_{i=1}^m|f(y_i-)-f(y_{i-1}+)|^2=
\sum_{i;\, y_i>\alpha(n_0)}|f(y_i-)-f(y_{i-1}+)|^2\\
&<\sum_{n=n_0+1}^\infty\|z_{\alpha(n)}-z_{\alpha(n-1)}\|^2=\sum_{n=n_0}^\infty \frac{1}{(n+1)\log^2(n+1)}<\varepsilon.
\end{align*}
Hence, the third condition of Theorem \ref{scalarRS} holds, and $f$ is Riemann-Stieltjes product integrable.
Taking into account the relation between $f$ and $A$, we conclude that $A$ is Riemann-Stieltjes product integrable, and consequently also Kurzweil-Stieltjes product integrable.

Since the only limit element of $\Lambda$ is $b$ and
$\lim\limits_{\beta\to b-}z_\beta=z_b$, we see that the assumption of Corollary \ref{C31} is satisfied.  Therefore, $A$ is strongly Kurzweil-Stieltjes product integrable. Moreover, Corollary \ref{C31}  implies that the family $(x_\alpha)_{\alpha\in\Lambda}$ given by \eqref{E601} is multipliable, 
and (since $x_a=x_b=I$), we have
$$
\prod_a^b(I+{\rm d}A(t))=\underset{\alpha\in\Lambda}{\Pi}x_\alpha=\prod_{n=0}^\infty x_{S(\alpha(n))}=\prod_{n=0}^\infty \left(1+\frac{(-1)^{n}}{\sqrt{n+2}\log(n+2)}\right)I
=\prod_{n=2}^\infty \left(1+\frac{(-1)^{n}}{\sqrt{n}\log n}\right)I, 
$$
where the infinite product on the right-hand side is guaranteed to be convergent and nonzero.
\end{example}

The next theorem deals with Kurzweil-Stieltjes product integrability of step mappings with idempotent values, and will be needed in the following section.

\begin{theorem}\label{T41} Let $A:[a,b]\to E$ be a step mapping with representation \eqref{E30}. Assume that $A$
satisfies the following assumptions:
\begin{enumerate}
\item $z_\alpha\cdot z_\alpha=z_\alpha$ for all $\alpha\in\Lambda$.
\item For each limit element $\gamma\in\Lambda$, we have $\underset{\beta\to\gamma-}{\lim}z_\beta=z_\gamma$.
\item The family  $(x_\alpha)_{\alpha\in\Lambda}$ given by
\begin{equation*}
\begin{cases}
x_\alpha=I & \mbox{ if $\alpha=a$, or if $\alpha$ is a limit element},\\
x_{S(\beta)}=I+z_{S(\beta)}-z_{\beta}& \mbox{ if $\beta\in\Lambda^{<b}$}
\end{cases}
\end{equation*}
 is multipliable,  and its elements as well as its product are invertible.
\end{enumerate} 
Then $A$ is Kurzweil-Stieltjes product integrable, the family $(z_\alpha)_{\alpha\in\Lambda}$ is multipliable, and
\begin{equation}\label{E43}
\prod_a^b(I+{\rm d}A(t))\cdot A(a)=\left(\underset{\alpha\in\Lambda}{\Pi}x_\alpha\right)z_a=\underset{\alpha\in\Lambda}{\Pi}z_\alpha.
\end{equation}
\end{theorem}

\begin{proof}
The Kurzweil-Stieltjes product integrability of $A$ is guaranteed  by Corollary \ref{C31}.
The first equality of \eqref{E43} follows also from Corollary \ref{C31} and the fact that $A(a)=z_a$.
It remains to prove the second equality of \eqref{E43}. We assume that $z_a\ne 0$ (the equality is obvious if $z_a=0$).

We apply transfinite induction to prove that for every $\gamma\in\Lambda$, the product $\underset{\alpha\in\Lambda^{\le\gamma}}{\Pi}z_\alpha$
is well defined and 
$$\underset{\alpha\in\Lambda^{\le\gamma}}{\Pi}z_\alpha=\left(\underset{\alpha\in\Lambda^{\le\gamma}}{\Pi}x_\alpha\right)\cdot z_a.$$
Notice first that  $(z_\alpha)_{\alpha\in\Lambda^{\le a}}$ is obviously multipliable, and 
\begin{equation}\label{trans}
\underset{\alpha\in\Lambda^{\le a}}{\Pi}z_\alpha=z_a=x_a\cdot z_a=\left(\underset{\alpha\in\Lambda^{\le a}}{\Pi}x_\alpha\right)\cdot z_a.
\end{equation}
Next, make an induction hypothesis: assume that for every $\beta\in\Lambda^{<\gamma}$, $(z_\alpha)_{\alpha\in\Lambda^{\le\beta}}$ is multipliable, 
and its product is $\left(\underset{\alpha\in\Lambda^{\le\beta}}{\Pi}x_\alpha\right) \cdot z_a$.

If $\gamma=S(\beta)$ for some $\beta\in\Lambda^{<\gamma}$, then $(z_\alpha)_{\alpha\in\Lambda^{\le \gamma}}$ is obviously multipliable, 
$x_\gamma=I+z_{S(\beta)}-z_\beta$, and
\begin{equation*}
\left(\underset{\alpha\in\Lambda^{\le\gamma}}{\Pi}x_\alpha\right)\cdot z_a
=x_\gamma\cdot\left(\underset{\alpha\in\Lambda^{\le\beta}}{\Pi}x_\alpha\right)\cdot z_a
=(I+z_{S(\beta)}-z_\beta)\left(\underset{\alpha\in\Lambda^{\le\beta}}{\Pi}z_\alpha\right)
=(I+z_{S(\beta)}-z_\beta)\cdot z_\beta\cdot\left(\underset{\alpha\in\Lambda^{<\beta}}{\Pi}z_\alpha\right).
\end{equation*}
Because $z_\beta z_\beta = z_\beta$, we get
\begin{equation*}
(I+z_{S(\beta)}-z_\beta)z_\beta=z_\beta+z_{S(\beta)}z_\beta - z_\beta z_\beta=z_\beta+z_{S(\beta)}z_\beta - z_\beta=z_{S(\beta)}z_\beta,
\end{equation*}
and therefore 
$$
\left(\underset{\alpha\in\Lambda^{\le\gamma}}{\Pi}x_\alpha\right)\cdot z_a 
=z_{S(\beta)}z_\beta\cdot\left(\underset{\alpha\in\Lambda^{<\beta}}{\Pi}z_\alpha\right)
=\underset{\alpha\in\Lambda^{\le\gamma}}{\Pi}z_\alpha.
$$ 
Assume next that $\gamma$ is a limit element of $\Lambda$, and let $\varepsilon > 0$ be given. Since the family $(x_\alpha)_{\alpha\in\Lambda}$ is multipliable, there is by Definition \ref{D1} (ii) a $\beta_\varepsilon\in \Lambda^{<\gamma}$ such that
$$
\left\|\underset{\alpha\in \Lambda^{\le\beta}}{\Pi}x_\alpha - \underset{\alpha\in \Lambda^{<\gamma}}{\Pi}x_\alpha\right\|=\left\|\underset{\alpha\in \Lambda^{<S(\beta)}}{\Pi}x_\alpha - \underset{\alpha\in \Lambda^{<\gamma}}{\Pi}x_\alpha\right\|<\frac{\varepsilon}{\|z_a\|}, \quad \beta\in \Lambda\cap[\beta_\varepsilon,\gamma).
$$
Then
$$
\left\|\underset{\alpha\in \Lambda^{\le\beta}}{\Pi}x_\alpha\cdot z_a - \underset{\alpha\in \Lambda^{<\gamma}}{\Pi}x_\alpha\cdot z_a\right\|
\le \|z_a\|\left\|\underset{\alpha\in \Lambda^{\le\beta}}{\Pi}x_\alpha - \underset{\alpha\in \Lambda^{<\gamma}}{\Pi}x_\alpha\right\|
<\varepsilon, \quad \beta\in \Lambda\cap[\beta_\varepsilon,\gamma).
$$
In view of the above result and the induction hypothesis we get
$$ 
\left\|\underset{\alpha\in \Lambda^{\le \beta}}{\Pi}z_\alpha - \underset{\alpha\in \Lambda^{<\gamma}}{\Pi}x_\alpha\cdot z_a\right\|<\varepsilon, \quad \beta\in \Lambda\cap[\beta_\varepsilon,\gamma).
$$
This result, condition (ii) of Definition \ref{D1}, and the fact that $x_\gamma=I$ imply that 
$$
\underset{\alpha\in\Lambda^{\le\gamma}}{\Pi}z_\alpha=\underset{\alpha\in\Lambda^{<\gamma}}{\Pi}x_\alpha\cdot z_a=\underset{\alpha\in\Lambda^{\le\gamma}}{\Pi}x_\alpha\cdot z_a.
$$
This completes the proof by transfinite induction. 
\end{proof}

\section{Parallel translation}\label{S5}

Parallel translation (also known as parallel transport)
of vectors along curves on manifolds
is one of the basic concepts of
differential geometry, and has important applications in physics. Given an oriented path in a Riemannian manifold, the parallel translation is a certain mapping from the tangent space at the initial point to the tangent space at the endpoint, which is linear and isometric. In the simplest case when the path is a geodesic on a two-dimensional surface, the vector which is being translated moves continuously along the path so that its length and its angle with the curve  remain constant. For a nice description of the geometrical meaning of parallel translation including the higher-dimensional case, see \cite{Arnold}. In differential geometry textbooks, the definition of parallel translation is usually based on the concept of covariant derivative; see~\cite{Spivak}.

The observation that the concept of parallel translation is related to Stieltjes-type product integrals goes back to the paper~\cite{HH78} by H.~Haahti and S.~Heikkil\"a, 
who considered translation of vectors along paths on Banach manifolds, and
employed the Riemann-Stieltjes product integral (although they did not call it by that name).
To see the connection with product integrals, let us begin by recalling the alternative description of parallel translation from the beginning of \cite{HH78}.
\smallskip

Consider a polyhedral surface $\mathcal M$ in $\mathbb R^3$ and an oriented path $\ell$ on $\mathcal M$ which does not cross any vertex of~$\mathcal M$. 
Assume that $\ell$ can be decomposed into a finite union of subpaths $\ell_0\cup\cdots\cup\ell_{m}$, where for every $i\in\{1,\ldots,m-1\}$, the endpoints of $\ell_i$ are the only points of $\ell_i$ lying on the edges of~$\mathcal M$. For every $i\in\{0,\ldots,m\}$, let $H_i$ be the face of $\mathcal M$ that contains $\ell_i$. Also, let $M_i$ denote the tangent space of $H_i$, i.e., the 2-dimensional subspace of $\mathbb R^3$ parallel to $H_i$. We would like to define the parallel translation  of tangent vectors of  $\mathcal M$ along $\ell$, i.e., a mapping $T_\ell:M_0\to M_m$.

If $m=0$, i.e., the whole path is contained within the single face $H_0$, then  
the parallel translation~$T_\ell$ of tangent vectors along $\ell$
is just the Euclidean parallel translation.  Thus $T_\ell$ is the identity operator on the tangent space $M_0$. Alternatively, $T_\ell$
can be interpreted as the restriction to $M_0$ of the orthogonal projection operator $P_0$ which maps  $\mathbb R^3$ onto $M_0$; we write $T_\ell=P_0$.   

If $m=1$, the tangent vectors are first translated in the Euclidean sense along $\ell_0$. At the terminal point of $\ell_0$, which is in $H_0\cap H_1$, they are projected by the orthogonal projection operator $P_1$ from  $\mathbb R^3$ onto the tangent space $M_1$, and finally translated in the Euclidean sense along $\ell_1$.  This yields the translation operator $T_\ell=P_1P_0$.

Continuing in this way, we conclude that in the general case, we get the translation operator $T_\ell=P_{m}\cdots P_0$, where $P_i$ is the orthogonal projection from  $\mathbb R^3$ onto $M_{i}$. Notice that $T_\ell$ does not depend on the exact shape of $\ell$ on the faces of $\mathcal M$; its values
change only
on the edges crossed by $\ell$.

The process of translating a vector along a path on a polyhedral surface is depicted in Figure 1, which is taken over from \cite{HH78}.

\smallskip
\begin{figure*}[ht]
\centering
\includegraphics{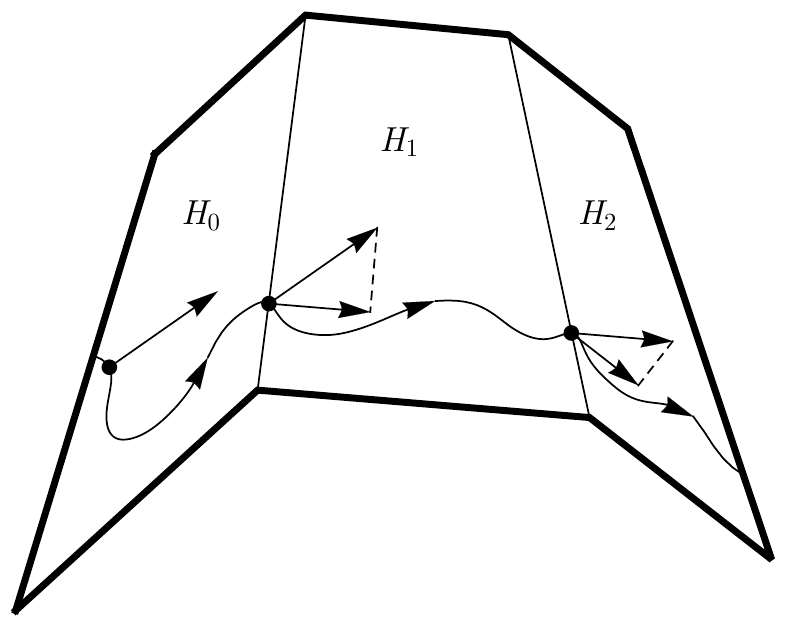}
\caption{Parallel translation along a path on a polyhedral surface}
\end{figure*} 

Next, consider the more complicated situation when $\mathcal M$ is a smooth surface in $\mathbb R^3$, and $\ell:[a,b]\to\mathcal M$ is 
a~path  of finite length with the initial point $x=\ell(a)$ and terminal point $y=\ell(b)$. The parallel translation operator should be a mapping $T_\ell$ from the tangent space at $x$ to the tangent space at $y$.
To obtain $T_\ell$, it is natural to approximate $\mathcal M$ along $\ell$ by a sequence of tangent planes, i.e., by a~polyhedral surface.
Choose $m+1$ successive points $x_i=\ell(t_i)$, $i\in\{0,\ldots,m\}$, corresponding to a~partition $D:a=t_0<t_1<\dots<t_{m}=b$.
For every $i\in\{0,\dots,m\}$, let $H_i$ be the tangent plane of~$\mathcal M$ at~$x_i$, and $M_i$ the tangent space of $H_i$. Assume that if the partition $D$ is fine enough, then each two successive tangent planes $H_{i}$ and $H_{i+1}$ have an intersection. (For example, this assumption is true if $\ell$ is continuously differentiable.)
Let $\ell_D$ be a path on $H_0\cup\cdots\cup H_m$, starting from $x$, passing through $x_1,\ldots,x_{m-1}$, and terminating at $y$.
We already know that the parallel translation operator corresponding to translation along $\ell_D$ is $T_{\ell_D}=P(x_m)P(x_{m-1})\cdots P(x_0)$, where $P(x_i)$ is the orthogonal projection from  $\mathbb R^3$ onto $M_{i}$.

If the limit of $T_{\ell_D}$  exists when the norm of the partition $D$ tends to zero, it is denoted by $T_\ell$ and 
 called the~parallel translation operator along $\ell$. 
Thus $T_\ell=\underset{|D|\to 0}{\lim}\prod_{i=m}^0 P(l(t_i))$, where $|D|=\max\limits_{1\le i\le m}(t_i-t_{i-1})$.
If we denote $A=P\circ\ell$, then $T_\ell$ is the Riemann-Haahti product of $A$ in the sense of the following definition.

\begin{definition}\label{E71}
Consider a mapping $A:[a,b]\to E$, where $E$ is a Banach algebra. Assume there exists an
element $P_A\in E$ with the following property: For each $\varepsilon > 0$, there exists 
 a gauge $\delta:[a,b]\to\mathbb R^+$ such that 
if $(\xi_i,[t_{i-1},t_i])_{i=1}^m$ is a $\delta$-fine tagged partition of $[a,b]$, then 
$\left\|P_A-\prod_{i=m}^0 A(t_i)\right\|<\varepsilon$.
In this case, $P_A$ is called the Kurzweil-Haahti product of $A$, and will be denoted by $\prod_a^b A(t)$.

If  $\delta$ is assumed to be constant on $[a,b]$, then $P_A=\prod_a^b A(t)$ is called the Riemann-Haahti product of $A$. 
\end{definition}

Note that the definition of $\prod_a^b A(t)$ is very similar to the definition of the product integral $\prod_a^b V(t,{\rm d}t)$, where $V(t,[x,y])=A(y)$; the only difference is that we don't require the element $P_A$ to be invertible.

\smallskip

The next theorem clarifies the relation between Haahti products and Stieltjes product integrals: it shows that the Haahti product $\prod_a^b A(t)$ exists if $A$ is an idempotent-valued mapping, i.e., if $A(t)\cdot A(t)=A(t)$ for all $t\in[a,b]$, and if $A$ is Stieltjes product integrable. The idea of the proof is borrowed from \cite[Lemma 2.1]{HH78}.

\begin{theorem}\label{L41} Assume that $E$ is a unital Banach algebra, and that $A:[a,b]\to E$ is an idempotent-valued mapping.
If the Kurzweil-Stieltjes or the Riemann-Stieltjes product integral $\prod_a^b(I+{\rm d}A(t))$  exists, then the Kurzweil-Haahti or Riemann-Haahti product $\prod_a^b A(t)$ exists as well, and 
$\prod_a^b A(t)=\left(\prod_a^b(I+{\rm d}A(t))\right)\cdot A(a)$.
\end{theorem}

\begin{proof} Consider an arbitrary partition $a=t_0<t_1<\cdots<t_m=b$. For every $i\in\{1,\dots,m\}$, we have 
$$ 
\begin{aligned}
A(t_i)A(t_{i-1})&= (A(t_{i-1})+(A(t_i)-A(t_{i-1})))A(t_{i-1})=A(t_{i-1})A(t_{i-1})+(A(t_i)-A(t_{i-1}))A(t_{i-1})\\
&= A(t_{i-1})+ (A(t_i)-A(t_{i-1}))A(t_{i-1})=(I+(A(t_i)-A(t_{i-1})))A(t_{i-1}).
\end{aligned}
$$
Using this result repeatedly for $k=m$, $m-1$, \dots, 1, we obtain
\begin{equation*}
\prod_{i=m}^0A(t_i)=\prod_{i=m}^1(I+A(t_i)-A(t_{i-1}))A(a),
\end{equation*}
which implies the statement of the theorem.
\end{proof}

The following example shows that the Riemann-Haahti product can exist although neither the Riemann-Stieltjes nor the Kurzweil-Stieltjes product integral exist.

\begin{example}\label{Ex711} 
Consider the Banach algebra $E=l^\infty$ of all bounded sequences equipped with componentwise 	multiplication. For every
$n\in\mathbb N$, let $e^n\in l^\infty$ be the sequence $(e_i^n)_{i=1}^\infty$ given by 
$$e_i^n=\begin{cases} 1,&\quad n=i,\\ 0,&\quad n\ne i.\end{cases}$$    
Let $A:[0,1]\to l^\infty$ be defined by
$$
A(t)=\begin{cases} e^1+e^{n+2},&\quad 1-2^{-n}< t \le 1-2^{-(n+1)}, \ n\in\mathbb N_0,\\ e^1,&\quad t\in\{0,1\}. \end{cases} 
$$
Clearly, $A$ is an idempotent-valued mapping. For every partition $0=t_0<t_1<\cdots<t_m=1$, we have $\prod_{i=m}^0 A(t_i)=e^1$. Hence, the Riemann-Haahti product of $A$ exists and $\prod_0^1 A(t)=e^1$.

On the other hand, $A$ is neither Riemann-Stieltjes nor Kurzweil-Stieltjes product integrable. To see this, consider again a partition $D:0=t_0<t_1<\cdots<t_m=1$. For every $i\in\{1,\ldots,m-1\}$, we have $A(t_i)=e^1+e^{n_i}$ for a certain $n_i\ge 2$. Hence,
$$P(D)=\prod_{i=m}^1 (I+A(t_i)-A(t_{i-1}))=(I-e^{n_{m-1}})(I+e^{n_{m-1}}-e^{n_{m-2}})\cdots(I+e^{n_{2}}-e^{n_{1}})(I+e^{n_1}).$$
The right-hand side represents an element of $l^\infty$ whose components at positions $n_1,\ldots,n_{m-1}$ are zero, and all other components are equal to 1. Thus, no matter how we choose a gauge $\delta:[0,1]\to\mathbb R^+$, we can always find two $\delta$-fine partitions $D_1$, $D_2$ of $[0,1]$ such that $\|P(D_1)-P(D_2)\|=1$. It follows that $A$ cannot be Kurzweil-Stieltjes product integrable.
\end{example}

\begin{remark}\label{R710} The Banach algebra $E$ in Definition \ref{E71} need not be unital. In this case, neither the Kurzweil-Stieltjes nor the Riemann-Stieltjes product integral is defined. On the other hand, the Riemann-Haahti or Kurzweil-Haahti product may exist. For instance,  
it is enough to replace $l^\infty$ in Example \ref{Ex711} by $c_0$, which is not unital.
\end{remark}

The problem of parallel translation described in the introduction 
can be reformulated in a more abstract setting where
$\mathcal M$ is a Hausdorff topological space, $\ell:[a,b]\to\mathcal M$ is a~continuous path, and $P:\mathcal M\to L(X)$ is a projection-valued mapping (i.e., $P^2=P$) from $\mathcal M$ to the space of bounded linear operators on a~certain Banach space $X$. It is then natural to define the corresponding parallel translation operator as the Riemann-Haahti product $\prod_a^bP(\ell(t))$ whenever it exists.

The next result is an immediate consequence of Theorems \ref{L31} and \ref{L41}.

\begin{theorem}\label{P51}  Let $\mathcal M$ be a Hausdorff topological space, $X$ a Banach space, $P:\mathcal M\to L(X)$ a projection-valued mapping, and $\ell:[a,b]\to \mathcal M$ a continuous path in $\mathcal M$. Assume that $A=P\circ\ell$ is right- or left-continuous at each point of $(a,b)$, has a finite $p$-variation for a certain $p\in(0,2)$, $I+\Delta^+A(t)$ is invertible for all $t\in[a,b)$, and $I+\Delta^-A(t)$ is invertible for all $t\in(a,b]$. Then both the  Riemann-Stieltjes product integral  and  the Riemann-Haahti product  of $P\circ\ell$ exist, and 
$$\prod_a^bP(\ell(t))=\prod_a^b(I+{\rm d}P(\ell(t)))\cdot P(\ell(a)).$$
\end{theorem}

In \cite{HH78}, the authors were dealing with the case where $\mathcal M$ is a~$C^0$-manifold modeled on Banach spaces. Both the Riemann-Stieltjes product integral $\prod_a^b(I+{\rm d}P(\ell(t)))$ and the Riemann-Haahti product $\prod_a^bP(\ell(t))$ were referred to as the parallel translation operators and denoted by $B_\ell$, $T_\ell$. A sufficient condition for the existence of these operators was presented in \cite[Theorem 3.1]{HH78}, where it was assumed that $A=P\circ\ell$ has bounded variation, is right- or left-continuous at each point of $(a,b)$, and has only a~finite number of discontinuity points. 
Theorem \ref{P51} replaces the assumption of bounded variation by the finiteness of\break $p$-variation for some $p\in(0,2)$. Also, $A$ can have up to countably many discontinuities (recall that a mapping with finite $p$-variation is necessarily regulated, and therefore has at most countably many discontinuities). The extra assumptions concerning the invertibility of $I+\Delta^+A(t)$ and $I+\Delta^-A(t)$ guarantee that the parallel translation operator $B_\ell$ (i.e., the product integral $\prod_a^b(I+{\rm d}P(\ell(t)))$) is invertible.   
The invertibility of $B_\ell$ was investigated in \cite[Proposition 4.1]{HH78}, where it was assumed that $A$ is continuous. Hence, our conditions are less restrictive. 

\smallskip

It was also shown in \cite{HH78} that if $P\circ\ell$ is continuous (an essential hypothesis) and has  bounded variation, then the Riemann-Stieltjes integral equations
\begin{equation*}
B(t)=I+\int_a^t{\rm d}(P(\ell(s)))B(s), \quad \hbox{ and } \quad T(t) = P(\ell(a)) + \int_a^t{\rm d}(P(\ell(s)))T(s), \quad t\in[a,b]
\end{equation*}
have unique solutions $B,\,T:[a,b]\to L(X)$, and $B_\ell =B(b)$, $T_\ell=T(b)$. 
If $\mathcal M$ is a $C^1$-manifold and the path~$\ell$ is smooth, the two integral equations reduce to the initial value problems
\begin{equation*}
B'(t)=P'(\ell(t))\ell'(t)B(t), \ B(a)=I, \quad \hbox{ and } \quad T'(t) = P'(\ell(t))\ell'(t)T(t), \ T(a)=P(\ell(a)).
\end{equation*}
When $\mathcal M$ is a $C^0$-manifold and $P:\mathcal M\to L(X)$ is continuous and projection-valued, the  mappings $T=\ell \mapsto T_\ell$ and $B=\ell\mapsto B_\ell$, defined for those $\ell$ for which $P\circ\ell$ has bounded variation, are called in \cite{HH78} $P$-connexions. A~result on the invariance of a scalar product,
defined by a bounded bilinear function of $X$, under these connexions generalizes the result that in the classical Levi-Civita parallelism, the scalar product of any two tangent vectors remains constant. 

\smallskip

The definitions of parallel translation operators can also be based on  Kurzweil-Stieltjes product integrals and Kurzweil-Haahti products. The next  
theorem provides a new existence result for the Kurzweil-Stieltjes product integral and the  
Kurzweil-Haahti product of $P\circ\ell$. It is a~straightforward consequence of Corollary~\ref{C31}, Theorems \ref{T41} and \ref{L41}.

\begin{theorem}\label{T51} Let $\mathcal M$ be a Hausdorff topological space, $X$ a Banach space, $P:\mathcal M\to L(X)$ a projection-valued mapping, and $\ell:[a,b]\to \mathcal M$ a continuous path in $\mathcal M$. Assume $A=P\circ\ell$ has representation \eqref{E30}, where $\lim_{\beta\to\gamma-}z_\beta=z_\gamma$ for each limit element $\gamma\in\Lambda$, the family $(x_\alpha)_{\alpha\in\Lambda}$ given by \eqref{E601} is multipliable,  and its elements as well as its product are invertible.  
Then both the Kurzweil-Stieltjes product integral and the Kurzweil-Haahti product of $P\circ\ell$ exist, and
$$
\prod_a^bP(\ell(t))=\prod_a^b(I+{\rm d}P(\ell(t))) \cdot P(\ell(a))=\underset{\alpha\in\Lambda}{\Pi}z_\alpha.
$$
\end{theorem}

\section{Linear generalized differential equations}

In this section, we discuss the relation between strong Kurzweil product integrals $\prod_a^b V(t,{\rm d}t)$ introduced in Definition \ref{strongDef} and J.~Kurzweil's theory of generalized differential equations (see \cite{Kurzweil 2012,Sch}). In particular, we focus on linear generalized differential equations of the form
\begin{equation}\label{linearGODE}
x(s)=x(a)+\int_a^s D_t[U(\tau,t)x(\tau)],\quad s\in[a,b],
\end{equation}
where $x$ and $U$ take values in a unital Banach algebra $E$, and the integral on the right-hand side is the strong Kurzweil-Henstock integral introduced in \cite[Definition 14.5]{Kurzweil 2012}. Readers who are not familiar with the definition of this integral might regard the following Lemma \ref{GODE-equivalence} as an equivalent definition of solutions to Eq.~\eqref{linearGODE}.

Equation \eqref{linearGODE} represents a special case of the nonlinear generalized differential equations studied in\break J.~Kurzweil's recent book \cite{Kurzweil 2012}. Some results specific for the linear case were obtained in \cite{Slavik 2013}. Moreover, there is an extensive literature devoted to the particular case of Eq.~\eqref{linearGODE} where $U$ does not depend on $\tau$; see e.g.~\cite{MS,MT,Sch} and the references there. 

The next lemma is a special case of \cite[Theorem 15.3]{Kurzweil 2012}, and provides a necessary and sufficient condition for a function $x:[a,b]\to E$ to be a~solution of Eq.~\eqref{linearGODE}.

\begin{lemma}\label{GODE-equivalence}
	Consider a pair of functions $U:[a,b]\times[a,b]\to E$, $x:[a,b]\to E$. Then the following statements are equivalent:
	\begin{enumerate}
		\item $x$ is a strong Kurzweil-Henstock solution of the linear generalized differential equation \eqref{linearGODE}.
		\item For every $\varepsilon>0$, there is a gauge $\delta:[a,b]\to\mathbb R^+$ such that
		$$\sum_{i=1}^m\|x(t_i)-x(t_{i-1})-(U(\xi_i,t_i)-U(\xi_i,t_{i-1}))x(\xi_i)\|<\varepsilon$$
		for every $\delta$-fine tagged partition of $[a,b]$.	
	\end{enumerate}
\end{lemma}

The second statement of the previous lemma makes it obvious that only the values $U(\tau,t)$ where $t$ is sufficiently close to $\tau$ are important. More precisely, assume that $\delta_0:[a,b]\to\mathbb R^+$ is an arbitrary function and for each $\tau\in[a,b]$, $U(\tau,t)$ is defined only for $t\in(\tau-\delta_0(\tau),\tau+\delta_0(\tau))\cap[a,b]$. Since
 the gauge $\delta$ in Lemma \ref{GODE-equivalence} can be always chosen to be smaller than or equal to $\delta_0$, it does not matter that $U$ is not defined on the whole square $[a,b]\times[a,b]$. This observation will be utilized in the following theorems.

To establish the relation between generalized differential equations and product integrals, we need the
next lemma.

\begin{lemma}\label{Vestimate}
If $V:[a,b]\times \mathcal I \to E$ is strongly Kurzweil product integrable and satisfies conditions (V1)--(V4), then there exist a constant $K\ge 1$ and a gauge $\delta_1:[a,b]\to\mathbb R^+$ with the following properties:
\begin{enumerate}
\item If $x\in(\tau-\delta_1(\tau),\tau]$, then $\|V(\tau,[x,\tau])\|\le K$ and $\|V(\tau,[x,\tau])^{-1}\|\le K$.
\item If $y\in[\tau,\tau+\delta_1(\tau))$, then $\|V(\tau,[\tau,y])\|\le K$ and $\|V(\tau,[\tau,y])^{-1}\|\le K$.
\end{enumerate}
\end{lemma}

\begin{proof}
We prove only the first statement, since the second one is similar.
Consider the indefinite product integral $W(t)=\prod_a^t V(s,{\rm d}s)$, $t\in[a,b]$.
Both $W$ and $W^{-1}$ are bounded, i.e., there exists a constant $M$ such that 
	$\|W(\tau)\|\le M$ and $\|W(\tau)^{-1}\|\le M$ for all $\tau\in[a,b]$. Let $K=1+M^2$.
	
		Choose an arbitrary $\tau\in(a,b]$.
    Since ${\lim}_{x\to \tau-} V(\tau,[x,\tau])=V_-(\tau)$ and ${\lim}_{x\to \tau-} V(\tau,[x,\tau])^{-1}=V_-(\tau)^{-1}$, there exists a $\delta_1(\tau)>0$ such that $\|V(\tau,[x,\tau])-V_-(\tau)\|<1$ and $\|V(\tau,[x,\tau])^{-1}-V_-(\tau)^{-1}\|<1$ whenever $x\in(\tau-\delta_1(\tau),\tau)$.
		(Note that $V_-(\tau)$ is invertible and the set of all invertible elements in $E$ is open, which means that $V(\tau,[x,\tau])$ is invertible whenever $x$ is sufficiently close to $\tau$.)
		By Theorem \ref{indefinite-regulated}, we have
			$$V_{-}(\tau)=\left(\prod _a^\tau V(t,{\rm d}t)\right)\cdot\left(\underset{s\to \tau-}{\lim}\prod_a^s V(t,{\rm d}t)\right)^{-1}=W(\tau)W(\tau-)^{-1}$$
		and therefore 
		$$\begin{aligned}
		&\|V_-(\tau)\|\le\|W(\tau)\|\cdot\|W(\tau-)^{-1}\|\le M^2,\\
		&\|V_-(\tau)^{-1}\|=\|W(\tau-)W(\tau)^{-1}\|\le\|W(\tau-)\|\cdot\|W(\tau)^{-1}\|\le M^2.
		\end{aligned}$$
		Consequently, if $x\in(\tau-\delta_1(\tau),\tau)$, we get
$$\begin{aligned}
&\|V(\tau,[x,\tau])\|\le \|V(\tau,[x,\tau])-V_-(\tau)\|+\|V_-(\tau)\|<1+M^2=K,\\
&\|V(\tau,[x,\tau])^{-1}\|\le\|V(\tau,[x,\tau])^{-1}-V_-(\tau)^{-1}\|+\|V_-(\tau)^{-1}\|<1+M^2=K.
\end{aligned}$$
		Obviously, the estimates $\|V(\tau,[x,\tau])\|\le K$ and $\|V(\tau,[x,\tau])^{-1}\|\le K$ hold also when $x=\tau$.
\end{proof}

The next result shows that if $V:[a,b]\times \mathcal I \to E$ satisfies conditions (V1)--(V4), then the strong Kurzweil product integral provides a solution to a certain linear generalized differential equation. 

\begin{theorem}\label{GODEsolution}
	Assume that $V:[a,b]\times \mathcal I \to E$ is strongly Kurzweil product integrable and satisfies conditions (V1)--(V4). 
	Then the indefinite product integral $W(t)=\prod_a^t V(s,{\rm d}s)$, $t\in[a,b]$, is a strong Kurzweil-Henstock solution of the linear generalized differential equation \eqref{linearGODE},
	where
	$U$ is given by
	\begin{equation}\label{Udef}
	U(\tau,t)=\begin{cases}
	V(\tau,[\tau,t]),&\quad t\ge\tau,\\
	V(\tau,[t,\tau])^{-1},&\quad t<\tau.
	\end{cases}
	\end{equation}
\end{theorem}

\begin{proof}
	We know from the proof of Lemma \ref{Vestimate} that for each $\tau\in(a,b]$, $V(\tau,[x,\tau])$ is invertible whenever $x$ is sufficiently close to $\tau$. Hence, the definition of $U$ is meaningful.
	
	Since both $W$ and $W^{-1}$ are bounded, there exists a constant $M$ such that 
	$\|W(\tau)\|\le M$ and $\|W(\tau)^{-1}\|\le M$ for all $\tau\in[a,b]$.
	Choose an arbitrary $\varepsilon>0$ and let $\delta:[a,b]\to\mathbb R^+$ be the corresponding gauge from Definition~\ref{strongDef} such that
	if $D=(\xi_i,[t_{i-1},t_i])_{i=1}^m$ is a $\delta$-fine tagged partition of $[a,b]$, then
	$$
	\sum_{i=1}^m\|V(\xi_i,[t_{i-1},t_i])-W(t_i)W(t_{i-1})^{-1}\|<\varepsilon.
	$$
	Without loss of generality, assume that $\delta\le\delta_1$, where $\delta_1$ is the gauge from Lemma \ref{Vestimate}; let $K\ge 1$ be the constant from this lemma.
	 
	Consider a tagged partition $D^*$ that is obtained from $D$ by splitting each interval $[t_{i-1},t_i]$ at $\xi_i$. In other words, $D^*$ consists of the intervals $[t_{i-1},\xi_i]$ and $[\xi_i,t_i]$, $i\in\{1,\ldots,m\}$, which share the same tag $\xi_i$. Note that $D^*$ is $\delta$-fine, and therefore
	\begin{equation}\label{est2}
	\sum_{i=1}^m\|V(\xi_i,[t_{i-1},\xi_i])-W(\xi_i)W(t_{i-1})^{-1}\|+\sum_{i=1}^m\|V(\xi_i,[\xi_i,t_i])-W(t_i)W(\xi_i)^{-1}\|<\varepsilon.
	\end{equation}
	Using successively the triangle inequality, the definition of~$U$,  condition~(V1), and the previous estimates, we obtain
	$$\sum_{i=1}^m\|W(t_i)-W(t_{i-1})-(U(\xi_i,t_i)-U(\xi_i,t_{i-1}))W(\xi_i)\|$$
	$$\le\sum_{i=1}^m\|W(t_i)-W(\xi_i)-(U(\xi_i,t_i)-U(\xi_i,\xi_i))W(\xi_i)\|
	+\sum_{i=1}^m\|W(\xi_i)-W(t_{i-1})-(U(\xi_i,\xi_i)-U(\xi_i,t_{i-1}))W(\xi_i)\|$$
	$$\begin{aligned}
	=&\sum_{i=1}^m\|W(t_i)-W(\xi_i)-(V(\xi_i,[\xi_i,t_i])-V(\xi_i,[\xi_i,\xi_i]))W(\xi_i)\|\\
	+&\sum_{i=1}^m\|W(\xi_i)-W(t_{i-1})-(V(\xi_i,[\xi_i,\xi_i])-V(\xi_i,[t_{i-1},\xi_i])^{-1})W(\xi_i)\|\\
	\end{aligned}
	$$
	$$=\sum_{i=1}^m\|W(t_i)-V(\xi_i,[\xi_i,t_i])W(\xi_i)\|+\sum_{i=1}^m\|V(\xi_i,[t_{i-1},\xi_i])^{-1}W(\xi_i)-W(t_{i-1})\|$$
	$$\le\sum_{i=1}^m\|W(t_i)W(\xi_i)^{-1}-V(\xi_i,[\xi_i,t_i])\|\cdot\|W(\xi_i)\|$$
	$$+\sum_{i=1}^m\|V(\xi_i,[t_{i-1},\xi_i])^{-1}\|\cdot\|W(\xi_i)W(t_{i-1})^{-1}-V(\xi_i,[t_{i-1},\xi_i])\|\cdot\|W(t_{i-1})\|$$
	$$\le M\sum_{i=1}^m\|W(t_i)W(\xi_i)^{-1}-V(\xi_i,[\xi_i,t_i])\|
	+KM\sum_{i=1}^m\|W(\xi_i)W(t_{i-1})^{-1}-V(\xi_i,[t_{i-1},\xi_i])\|\le KM\varepsilon.$$
The statement of the theorem now follows immediately from Lemma \ref{GODE-equivalence}.
\end{proof}

The following concept of equivalent functions was introduced in \cite[Definition 2.8]{JK}.

\begin{definition}
Two functions $V_1,V_2:[a,b]\times \mathcal I \to E$ are called equivalent if for every $\varepsilon>0$, there is a~gauge $\delta:[a,b]\to\mathbb R^+$ such that
$$\sum_{i=1}^m\|V_1(\xi_i,[t_{i-1},t_i])-V_2(\xi_i,[t_{i-1},t_i])\|<\varepsilon$$
for each $\delta$-fine partition of $[a,b]$.
\end{definition}

The usefulness of the previous definition stems from the following obvious fact: If $V_1$, $V_2:[a,b]\times \mathcal I \to E$ are equivalent, then $V_1$ is strongly Kurzweil product integrable if and only if  $V_2$ is strongly Kurzweil product integrable; in this case,  both integrals have the same value.

\begin{theorem}
	Assume that $V:[a,b]\times \mathcal I \to E$ is strongly Kurzweil product integrable and satisfies conditions (V1)--(V4). Then $V$ is equivalent to the function $\tilde V:[a,b]\times \mathcal I \to E$ given by
	\begin{equation}\label{tildeV}
	\tilde V(\xi,[x,y])=V(\xi,[\xi,y])V(\xi,[x,\xi]).
	\end{equation}
\end{theorem}

\begin{proof}
Consider the indefinite product integral $W(t)=\prod_a^t V(s,{\rm d}s)$, $t\in[a,b]$.
Both $W$ and $W^{-1}$ are bounded, i.e., there exists a constant $M$ such that 
	$\|W(\tau)\|\le M$ and $\|W(\tau)^{-1}\|\le M$ for all $\tau\in[a,b]$. 
	
Choose an arbitrary $\varepsilon>0$ and let $\delta:[a,b]\to\mathbb R^+$ be the corresponding gauge from Definition \ref{strongDef} such that if $D=(\xi_i,[t_{i-1},t_i])_{i=1}^m$ is a $\delta$-fine tagged partition of $[a,b]$, then
$$
\sum_{i=1}^m\|V(\xi_i,[t_{i-1},t_i])-W(t_i)W(t_{i-1})^{-1}\|<\varepsilon
$$
The partition $D^*$ that is obtained from $D$ by splitting each interval $[t_{i-1},t_i]$ at $\xi_i$ is also $\delta$-fine, and therefore
$$\sum_{i=1}^m\|V(\xi_i,[t_{i-1},\xi_i])-W(\xi_i)W(t_{i-1})^{-1}\|+\sum_{i=1}^m\|V(\xi_i,[\xi_i,t_i])-W(t_i)W(\xi_i)^{-1}\|<\varepsilon.$$
			
Without loss of generality, assume that $\delta\le\delta_1$ on $[a,b]$, where $\delta_1$ is the gauge from Lemma \ref{Vestimate}; let $K$ be the constant from this lemma. Then, for each $\delta$-fine tagged partition of $[a,b]$, we obtain the estimate
		$$\begin{aligned}
&\sum_{i=1}^m\|V(\xi_i,[t_{i-1},t_i])-\tilde V(\xi_i,[t_{i-1},t_i])\|
		=\sum_{i=1}^m\|V(\xi_i,[t_{i-1},t_i])-V(\xi_i,[\xi_i,t_i])V(\xi_i,[t_{i-1},\xi_i])\|\\
		\le&\sum_{i=1}^m\|V(\xi_i,[t_{i-1},t_i])-W(t_i)W(\xi_i)^{-1}W(\xi_i)W(t_{i-1})^{-1}\|\\
		+&\sum_{i=1}^m\|W(t_i)W(\xi_i)^{-1}W(\xi_i)W(t_{i-1})^{-1}-V(\xi_i,[\xi_i,t_i])V(\xi_i,[t_{i-1},\xi_i])\|\\
		<\varepsilon+&\sum_{i=1}^m\|W(t_i)\|\cdot\|W(\xi_i)^{-1}\|\cdot\|W(\xi_i)W(t_{i-1})^{-1}-V(\xi_i,[t_{i-1},\xi_i])\|\\
		+&\sum_{i=1}^m\|W(t_i)W(\xi_i)^{-1}-V(\xi_i,[\xi_i,t_i])\|\cdot\|V(\xi_i,[t_{i-1},\xi_i])\|<\varepsilon+(M^2+K)\varepsilon,
		\end{aligned}$$
		which proves that $V$ and $\tilde V$ are equivalent.
\end{proof}

\begin{remark}
The previous theorem can be rephrased as follows: If  $V:[a,b]\times \mathcal I \to E$ is strongly Kurzweil product integrable and satisfies conditions (V1)--(V4), then for each $\varepsilon>0$, there is a gauge $\delta:[a,b]\to\mathbb R^+$ such that
$$\sum_{i=1}^m\|V(\xi_i,[t_{i-1},t_i])-V(\xi_i,[\xi_i,t_i])V(\xi_i,[t_{i-1},\xi_i])\|<\varepsilon$$
for every $\delta$-fine partition of $[a,b]$. Note that this statement represents a stronger version of condition~(V2).
\end{remark}

The next result can be regarded as a converse to Theorem \ref{GODEsolution}.

\begin{theorem}\label{solvability}
Consider a function $V:[a,b]\times \mathcal I \to E$ and assume that the following conditions are satisfied:
\begin{enumerate}
\item $V(\xi,[\xi,\xi])=I$ for each $\xi\in[a,b]$.
\item For each $\tau\in(a,b]$, $V(\tau,[x,\tau])$ is invertible whenever $x$ is sufficiently close to $\tau$.
\item The generalized differential equation \eqref{linearGODE},
	where $U$ is given by \eqref{Udef}, has a strong Kurzweil-Henstock solution $W:[a,b]\to E$ such that $W(t)^{-1}$ exists for all $t\in[a,b]$, and both $W$ and $W^{-1}$ are bounded.
\item $V$ is equivalent to the function $\tilde V$ given by \eqref{tildeV}.
\item There exist a constant $K\ge 1$ and a gauge $\delta_1:[a,b]\to\mathbb R^+$ such that for each $\xi\in[a,b]$, we have $\|V(\xi,[x,\xi])\|\le K$ if $x\in(\xi-\delta_1(\xi),\xi]$, and $\|V(\xi,[\xi,y])\|\le K$ if $y\in[\xi,\xi+\delta_1(\xi))$.
\end{enumerate}
Then $V$ is strongly Kurzweil product integrable and $\prod_a^b V(t,{\rm d}t)=W(b)W(a)^{-1}$.
\end{theorem}

\begin{proof}
According to the second condition, the definition of $U$ is meaningful. By the third condition, there exists a constant $M$ such that 
	$\|W(\tau)\|\le M$ and $\|W(\tau)^{-1}\|\le M$ for all $\tau\in[a,b]$.

	Choose an arbitrary $\varepsilon>0$. 
Since $W$ is a strong Kurzweil-Henstock solution of Eq.~\eqref{linearGODE}, there exists a gauge $\delta:[a,b]\to\mathbb R^+$ such that if $D=(\xi_i,[t_{i-1},t_i])_{i=1}^m$ is a $\delta$-fine tagged partition of $[a,b]$, then
$$\sum_{i=1}^m\|W(t_i)-W(t_{i-1})-(U(\xi_i,t_i)-U(\xi_i,t_{i-1}))W(\xi_i)\|<\varepsilon.$$
Without loss of generality, assume that $\delta\le\delta_1$ on $[a,b]$.

The partition $D^*$ that is obtained from $D$ by splitting each interval $[t_{i-1},t_i]$ at $\xi_i$ is also $\delta$-fine, and therefore
	$$\sum_{i=1}^m\|W(t_i)-W(\xi_i)-(U(\xi_i,t_i)-U(\xi_i,\xi_i))W(\xi_i)\|+\sum_{i=1}^m\|W(\xi_i)-W(t_{i-1})-(U(\xi_i,\xi_i)-U(\xi_i,t_{i-1}))W(\xi_i)\|<\varepsilon.$$ 
Using the definition of $U$ and the fact that $U(\xi,\xi)=V(\xi,[\xi,\xi])=I$, the last estimate can be simplified to
	$$\sum_{i=1}^m\|W(t_i)-V(\xi_i,[\xi_i,t_i])W(\xi_i)\|+\sum_{i=1}^m\|V(\xi_i,[t_{i-1},\xi_i])^{-1}W(\xi_i)-W(t_{i-1})\|<\varepsilon.$$

Consequently, we obtain 
	$$\sum_{i=1}^m\|\tilde V(\xi_i,[t_{i-1},t_i])-W(t_i)W(t_{i-1})^{-1}\|=\sum_{i=1}^m\|V(\xi_i,[\xi_i,t_i])V(\xi_i,[t_{i-1},\xi_i])-W(t_i)W(t_{i-1})^{-1}\|$$
	$$\begin{aligned}
\le&\sum_{i=1}^m\|V(\xi_i,[\xi_i,t_i])V(\xi_i,[t_{i-1},\xi_i])-V(\xi_i,[\xi_i,t_i])W(\xi_i)W(t_{i-1})^{-1}\|\\
	+&\sum_{i=1}^m\|V(\xi_i,[\xi_i,t_i])W(\xi_i)W(t_{i-1})^{-1}-W(t_i)W(t_{i-1})^{-1}\|\\ 
	\le&\sum_{i=1}^m\|V(\xi_i,[\xi_i,t_i])\|\cdot\|V(\xi_i,[t_{i-1},\xi_i])\|\cdot\|W(t_{i-1})-V(\xi_i,[t_{i-1},\xi_i])^{-1}W(\xi_i)\|\cdot\|W(t_{i-1})^{-1}\|\\
	+&\sum_{i=1}^m\|V(\xi_i,[\xi_i,t_i])W(\xi_i)-W(t_i)\|\cdot\|W(t_{i-1})^{-1}\|<K^2M\varepsilon,
	\end{aligned}$$
	which means that $\tilde V$ is strongly Kurzweil product integrable and $\prod_a^b \tilde V(t,{\rm d}t)=W(b)W(a)^{-1}$. Since $V$ is equivalent to $\tilde V$, it is also strongly Kurzweil product integrable and $\prod_a^b V(t,{\rm d}t)=W(b)W(a)^{-1}$.
\end{proof}

\begin{remark}
The reader can verify that the previous theorem remains true if the conditions 1, 2, 5 are replaced by the assumption that $V$ satisfies conditions (V1)--(V4), where $V_-$ and $V_+$ are bounded on $[a,b]$.
\end{remark}

We now focus on the simpler situation when the function $U$ in Eq.~\eqref{linearGODE} does not depend on $\tau$, i.e., the equation has the form
\begin{equation}\label{simpleGODE}
x(t)=x(a)+\int_a^t D_s[A(s)]x(s),\quad t\in[a,b],
\end{equation}
for a certain function $A:[a,b]\to E$.
In this case, the strong Kurzweil-Henstock integral $\int_a^t D_s[A(s)]x(s)$ on the right-hand side is simply the strong Kurzweil-Stieltjes integral, which is  usually denoted by $\int_a^t {\mathrm d}[A(s)]x(s)$. The next result might be regarded as an infinite-dimensional counterpart to Theorem 2.15 from \cite{Sch90}. In contrast to \cite{Sch90}, we do not assume that $A$ has bounded variation.

\begin{theorem}\label{GODEsolution2}
	Assume that $A:[a,b]\to E$ is a regulated function such that $I+\Delta^+A(t)$ is invertible for all $t\in[a,b)$, and $I-\Delta^-A(t)$ is invertible for all $t\in(a,b]$. 
	Let $V:[a,b]\times \mathcal I \to E$ be given by 
	\begin{equation}\label{Vdef}
	V(\xi,[x,y])=(I+A(y)-A(\xi))(I+A(x)-A(\xi))^{-1}.
	\end{equation}
	Then the following statements are equivalent:
	\begin{enumerate}
	\item $V$ is strongly Kurzweil product integrable.
	\item The linear generalized differential equation \eqref{simpleGODE} has a strong Kurzweil-Henstock solution $W:[a,b]\to E$ such that $W(t)^{-1}$ exists for all $t\in[a,b]$, and both $W$ and $W^{-1}$ are bounded.
	\end{enumerate}
\end{theorem}

\begin{proof}
Using \eqref{Vdef}, we find that the definition of $U$ given in \eqref{Udef} reduces to $U(\tau,t)=I+A(t)-A(\tau)$ for all pairs~$\tau,t$. Thus
$$\sum_{i=1}^m\|W(t_i)-W(t_{i-1})-(U(\xi_i,t_i)-U(\xi_i,t_{i-1}))W(\xi_i)\|=\sum_{i=1}^m\|W(t_i)-W(t_{i-1})-(A(t_i)-A(t_{i-1}))W(\xi_i)\|,
$$
which means that $W$ is a strong Kurzweil-Henstock solution of Eq.~\eqref{linearGODE} if and only if $W$ is a strong Kurzweil-Henstock solution of Eq.~\eqref{simpleGODE}.

Using the assumptions on $A$, one can easily check that the function $V$ given by \eqref{Vdef} satisfies conditions (V1)--(V4); in particular, condition (V2) is satisfied because
\begin{equation}\label{V2}
V(\xi,[\xi,y])V(\xi,[x,\xi])=(I+A(y)-A(\xi))(I+A(\xi)-A(x))^{-1}=V(\xi,[x,y]).
\end{equation}

Let us prove the implication $1\Rightarrow2$.
According to Theorem~\ref{GODEsolution}, the indefinite product integral $W$ is a~strong Kurzweil-Henstock solution of Eq.~\eqref{linearGODE}, and therefore also a strong solution of Eq.~\eqref{simpleGODE}. Using the properties of the indefinite product integral, it is obvious that $W(t)^{-1}$ exists for all $t\in[a,b]$, and both $W$ and $W^{-1}$ are bounded.

It remains to prove the implication $2\Rightarrow 1$. We already know that the solution $W$ of Eq.~\eqref{simpleGODE} is also a solution of Eq.~\eqref{linearGODE}.
Consider the function $\tilde V$ given by \eqref{tildeV}. It follows from Eq.~\eqref{V2} that $\tilde V$ and $V$ are identical, and therefore equivalent. Since $A$ is regulated, there exists a constant $M$ such that $\|A(t)\|\le M$ for all $t\in[a,b]$. It follows that 
$$\|I+\Delta^+A(t)\|=\|I+A(t+)-A(t)\|\le 1+2M.$$
There exist only finitely many points $t_1,\ldots,t_l\in(a,b]$ such that $\|\Delta^-A(t)\|\ge\frac{1}{2}$. For each $t\in(a,b]\setminus\{t_1,\ldots,t_l\}$, we have
$(I-\Delta^-A(t))^{-1}=\sum_{k=0}^\infty (\Delta^-A(t))^k$, and therefore
$$\|(I-\Delta^-A(t))^{-1}\|\le\sum_{k=0}^\infty \|\Delta^-A(t)\|^k<\sum_{k=0}^\infty \frac{1}{2^k}=2.$$
If we let 
$L=\max(2,\|(I-\Delta^-A(t_1))^{-1}\|,\ldots,\|(I-\Delta^-A(t_l))^{-1}\|)$,
it follows that 
$$\|(I-\Delta^-A(t))^{-1}\|\le L,\quad t\in(a,b].$$

For each $\xi\in[a,b]$, there exists a number $\delta_1(\xi)>0$ such that
$$\|I+A(y)-A(\xi)-(I+\Delta^+A(\xi))\|<1,\quad y\in(\xi,\xi+\delta_1(\xi)),$$
$$\|(I-(A(\xi)-A(x)))^{-1}-(I-\Delta^-A(\xi))^{-1}\|<1,\quad x\in(\xi-\delta_1(\xi),\xi).$$
It follows that
$$\|I+A(y)-A(\xi)\|\le \|I+A(y)-A(\xi)-(I+\Delta^+A(\xi))\|+\|I+\Delta^+A(\xi)\|<2+2M,$$
$$\|(I-(A(\xi)-A(x)))^{-1}\|\le \|(I-(A(\xi)-A(x)))^{-1}-(I-\Delta^-A(\xi))^{-1}\|+\|(I-\Delta^-A(\xi))^{-1}\|<1+L.$$
Hence, if we denote $K=\max(2+2M,1+L)$, we obtain
$$\|V(\xi,[\xi,y])\|=\|I+A(y)-A(\xi)\|\le K,\quad y\in[\xi,\xi+\delta_1(\xi)),$$
$$\|V(\xi,[x,\xi])^{-1}\|=\|(I-(A(\xi)-A(x)))^{-1}\|\le K,\quad x\in(\xi-\delta_1(\xi),\xi].$$
The previous considerations imply that the five conditions of Theorem \ref{solvability} are satisfied, and therefore $V$ is strongly Kurzweil product integrable.
\end{proof}

\section{Open problems}

We conclude our paper with several open problems:
\begin{itemize}
\item  It is known that already in the finite-dimensional Banach algebra $E=\mathbb R^{2\times 2}$,  Kurzweil product integrability and Henstock-Kurzweil integrability of a function $A:[a,b]\to E$ do not imply each other (see \cite[Example 4.7]{JK}). However, the existing examples are not quite elementary. 
Is there a step mapping with well-ordered steps which is Kurzweil product integrable but not Henstock-Kurzweil integrable (or vice versa)? Equivalently, is there a well-ordered set $\Lambda\subset[a,b]$ and a family $(x_\alpha)_{\alpha\in\Lambda}$ such that $(\exp x_\alpha)_{\alpha\in\Lambda^{<b}}$ is multipliable and its product is invertible, but $( x_\alpha)_{\alpha\in\Lambda^{<b}}$ is not summable (or vice versa)?
\item Is there an example of a mapping  with well-ordered steps which is Kurzweil-Stieltjes product integrable but not Riemann-Stieltjes product integrable?
\item If $A:[a,b]\to E$ satisfies the assumptions of Theorem \ref{L31}, is it true that $A$ is strongly Kurzweil-Stieltjes product integrable?
\item According to Theorem \ref{K-right-regulated}, the Kurzweil product integral and its strong counterpart are equivalent for all functions $A:[a,b]\to E$ with countably many discontinuities. For Kurzweil-Stieltjes product integrals, we have a~similar  
Theorem \ref{T31} dealing with a~certain class of step functions with well-ordered steps. Is it possible to extend this statement to a~wider class of functions? More generally, are there some simple conditions on a function $V:[a,b]\times \mathcal I \to E$ guaranteeing the equivalence of the ordinary and strong product integral $\prod_a^b V(t,{\rm d}t)$?
\end{itemize}

\section*{Acknowledgement}
We are indebted to an anonymous referee for comments which led to a substantial improvement of the manuscript. In particular, the referee suggested to include the material of Section 8, outlined the proof of Theorem \ref{GODEsolution}, and simplified our original version of Example \ref{Ex33}.


\end{document}